\documentclass{amsart}

\usepackage[breaklinks,bookmarks=false]{hyperref}
\hypersetup{colorlinks,linkcolor=blue,citecolor=blue,urlcolor=blue,plainpages=false}

\usepackage[foot]{amsaddr}

\usepackage{cancel}
\usepackage{amsmath}
\usepackage{amsthm}
\usepackage{mathrsfs}
\usepackage{stmaryrd}
\usepackage{tikz}
\usepackage{pgfplots}
\newcommand{\SlopeTriangle}[6]
{

    \pgfplotsextra
    {
        \pgfkeysgetvalue{/pgfplots/xmin}{\xmin}
        \pgfkeysgetvalue{/pgfplots/xmax}{\xmax}
        \pgfkeysgetvalue{/pgfplots/ymin}{\ymin}
        \pgfkeysgetvalue{/pgfplots/ymax}{\ymax}

        \pgfmathsetmacro{\xArel}{#1}
        \pgfmathsetmacro{\yArel}{#3}
        \pgfmathsetmacro{\xBrel}{#1-#2}
        \pgfmathsetmacro{\yBrel}{\yArel}
        \pgfmathsetmacro{\xCrel}{\xArel}

        \pgfmathsetmacro{\lnxB}{\xmin*(1-(#1-#2))+\xmax*(#1-#2)} 
        \pgfmathsetmacro{\lnxA}{\xmin*(1-#1)+\xmax*#1} 
        \pgfmathsetmacro{\lnyA}{\ymin*(1-#3)+\ymax*#3} 
        \pgfmathsetmacro{\lnyC}{\lnyA+#4*(\lnxA-\lnxB)}
        \pgfmathsetmacro{\yCrel}{\lnyC-\ymin)/(\ymax-\ymin)} 

        \coordinate (A) at (rel axis cs:\xArel,\yArel);
        \coordinate (B) at (rel axis cs:\xBrel,\yBrel);
        \coordinate (C) at (rel axis cs:\xCrel,\yCrel);

        \draw[#6]   (A)-- node[anchor=north] {#5}
                    (B)--
                    (C)--
                    cycle;
    }
}

\usepackage[backend = biber]{biblatex}
\usepackage{a4wide}

\usepackage{amssymb}
\usepackage{mathrsfs}

\newcommand{\eq}{:=}

\newcommand{\grad}{\boldsymbol \nabla}
\renewcommand{\div}{\grad \cdot}

\newcommand{\BA}{\boldsymbol A}
\newcommand{\BB}{\boldsymbol B}

\newcommand{\BI}{\boldsymbol I}

\newcommand{\be}{\boldsymbol e}

\newcommand{\bm}{\boldsymbol m}
\newcommand{\bn}{\boldsymbol n}

\newcommand{\bq}{\boldsymbol q}
\newcommand{\br}{\boldsymbol r}
\newcommand{\bs}{\boldsymbol s}

\newcommand{\bx}{\boldsymbol x}
\newcommand{\by}{\boldsymbol y}
\newcommand{\bz}{\boldsymbol z}

\newcommand{\LH}{\mathscr H}

\newcommand{\N}{\mathbb N}
\newcommand{\Z}{\mathbb Z}
\newcommand{\R}{\mathbb R}
\newcommand{\C}{\mathbb C}

\newcommand{\bzero}{\boldsymbol 0}

\newcommand{\bxx}{\boldsymbol{ x }}
\newcommand{\bxi}{\boldsymbol{\xi}}

\newcommand{\xxhm}{\bxx^{\hbar,\bm}}
\newcommand{\xihn}{\bxi^{\hbar,\bn}}

\newcommand{\dist}{\operatorname{dist}}

\newcommand{\CIB}{C^\infty_{\rm b}}

\newcommand{\gs}[3]{\Psi_{#1,#2,#3}}

\newcommand{\gshmn}{\gs{\hbar}{\bm}{\bn}}

\newcommand{\xxkm}{\bx^{k,\bm}}

\newcommand{\xikn}{\bxi^{k,\bn}}

\newcommand{\balpha}{\boldsymbol{\alpha}}
\newcommand{\bbeta}{\boldsymbol{\beta}}

\newcommand{\HH}{\widehat{H}}

\DeclareMathOperator{\Ker}{Ker}

\newcommand{\wit}{\widetilde}
\newcommand{\h}{{\hbar}}
\newcommand{\Inv}{{\mathcal{I}_\h^0}}
\newcommand{\cD}{{\mathcal{D}}}
\newcommand{\cR}{{\mathcal{R}}}
\newcommand{\Vect}{{\mathrm{Vect}}}

\newcommand{\cP}{\mathcal{P}}

\newcommand{\card}{{\mathrm{Card}}}
\newcommand{\alphavec}{{\boldsymbol{\alpha}}}
\newcommand{\betavec}{{\boldsymbol{\beta}}}

\newcommand{\epsilonvec}{{\boldsymbol{\epsilon}}}

\newcommand{\gammavec}{{\boldsymbol{\gamma}}}
\newcommand{\mvec}{{\boldsymbol{m}}}
\newcommand{\nvec}{{\boldsymbol{n}}}
\newcommand{\qvec}{{\boldsymbol{q}}}
\newcommand{\rvec}{{\boldsymbol{r}}}
\newcommand{\svec}{{\boldsymbol{s}}}
\newcommand{\xvec}{{\boldsymbol{x}}}
\newcommand{\xivec}{{\boldsymbol{\xi}}}
\newcommand{\zerovec}{{\boldsymbol{0}}}
\newcommand{\B}{\mathcal{B}}

\newcommand{\eg}[2][\hbar]{\Psi_{#1, #2}}
\newcommand{\w}[2][\hbar]{\Phi_{#1, #2}}
\newcommand{\xm}[1]{\xvec^{\hbar, #1}}
\newcommand{\xin}[1]{\xivec^{\hbar, #1}}

\newcommand{\vA}{v_{\hbar, 0, 3}}
\newcommand{\vB}{v_{\hbar, 4}}
\newcommand{\vC}{v_{\hbar, 5, J+1}}

\newcommand{\wzero}{w^0_\hbar}

\newcommand{\Csym}{C_{\rm sym}}

\newtheorem{theorem}{Theorem}
\newtheorem{lemma}[theorem]{Lemma}
\newtheorem{corollary}[theorem]{Corollary}
\newtheorem{proposition}[theorem]{Proposition}
\newtheorem{remark}[theorem]{Remark}

\newtheorem{assumption}{Assumption}
\numberwithin{equation}{section}
\numberwithin{theorem}{section}

\newcommand{\cO}{\mathcal{O}}

\begin{document}

\title[A Galerkin method for scattering problems using Wilson bases]%
{An efficient Galerkin method for high-frequency scattering problems using Wilson bases}
\author{T. Chaumont-Frelet$^\sharp$}
\author{M. Ingremeau$^\dagger$}
\author{F. Proust$^\ddagger$}

\begin{abstract}
We propose a new Galerkin discretization scheme for wave scattering problems that is
based on microlocalised basis functions. We show that the proposed method can be
made uniformly accurate for large wavenumbers $k$ with a number of degrees of freedom
only scaling as $k^{d-1/2}$, while leading to an essentially sparse linear system.
In contrast, finite element methods are known to require a number of degrees of
freedom scaling at least as $k^d$ to achieve the same property. A similar method
based on a Gabor frame was previously introduced by two of the authors, but it was
suffering from severe conditioning issues. In the present work, by replacing the
Gabor frame by a Wilson basis, we completely alleviate this problem. We rigorously
establish error estimates and condition number bounds for the proposed method, and
we provide one-dimensional numerical examples illustrating our theoretical findings.
\end{abstract}

\address{\vspace{-.5cm}}
\address{\noindent \tiny \textup{$^\sharp$ Inria, Univ. Lille, CNRS, UMR 8524 -- Laboratoire Paul Painlev\'e}}
\address{\noindent \tiny \textup{$^\dagger$ Univ. Grenoble Alpes, CNRS, IF}}
\address{\noindent \tiny \textup{$^\ddagger$ Universit\'e C\^ote d'Azur, Inria, CNRS, LJAD}}

\maketitle
\thispagestyle{empty}

\section{Introduction}

The numerical simulation of wave propagation phenomena is of paramount
importance in modern technological developments crucial 
to our society. However, designing approximation
methods that are at the same time robust, efficient and accurate remains a
challenging open problem when the size of the domain of interest is large
compared to the wavelength. This situation is often known as the ``high-frequency
regime''.

In this work, we focus on the following model scattering problem
in dimension $d \geq 1$. Given a wavenumber $k> 0$ and an incident
field $u^{\rm inc}: \R^d \to \C$, find the (unique) scattered field
$u^{\rm sca}: \R^d \to \C$ satisfying
\begin{equation}
\label{eq_helmholtz_intro}
\left \{
\begin{array}{rcll}
-k^2 \nu u^{\rm tot} - \div(\BB\nabla u^{\rm tot}) &=& 0 & \text{ in } \R^d
\\
\partial_r u^{\rm sca} -iku^{\rm sca} &\to& 0 & \text{ as } r \to \infty
\end{array}
\right .
\end{equation}
where $u^{\rm tot} \eq u^{\rm sca} + u^{\rm inc}$ is the total field.
In~\eqref{eq_helmholtz_intro}, $\nu: \Omega \to \R$ and $\BB: \Omega \to \mathcal{S}_d(\R)$
are smooth coefficients which describe the heterogeneity causing
the scattering of the incident field. These coefficients are required
to coincide with $1$ and the identity matrix $\BI$ outside a ball $B_R$
of radius $R > 0$. Here, $\partial_r$ denotes the radial derivative, and the
second equation in~\eqref{eq_helmholtz_intro} is known as the Sommerfeld
radiation condition.

To approximate~\eqref{eq_helmholtz_intro}, ``full-wave'' methods,
(such as finite element, finite difference, or boundary element
techniques) directly compute an approximation to~\eqref{eq_helmholtz_intro}.
These methods can be made arbitrarily accurate for any frequency
$k > 0$ if the number of degrees of freedom $N$ employed to approximate
$u^{\rm sca}$ is sufficiently increased. However, as $k$ increases,
they quickly become computationally infeasible, due to the
necessary increase of $N$. Specifically, to remain uniformly accurate,
finite element and finite difference methods require to solve a sparse
linear system with size $N$ growing (at least) as $(kR)^d$,
see~\cite{bernkopf2025wavenumber,chaumontfrelet_nicaise_2019a,melenk_sauter_2010a}. On the other
hand, boundary elements only need $N$ to grow as $(kR)^{d-1}$, but they
have the disadvantage of producing a dense linear system with entries that
might be challenging to compute, see~\cite{galkowski2023does,lohndorf2011wavenumber}.
Another important class of techniques includes
``Trefftz'' methods which employ basis functions satisfying the homogeneous
PDE such as plane waves, see e.g.~\cite{farhat2001discontinuous,hiptmair_moiola_perugia_2016a}.
They are often considered more efficient than finite element methods for high-frequency problems.
However, the frequency-explicit convergence analysis of these methods is
rarely dealt with in the literature. Besides, the references actually
proposing a frequency-explicit convergence analysis tend to show that
there is no qualitative theoretical improvement over finite element
methods, despite the quantitatively better behaviour observed in practice,
see e.g.~\cite{amara2009convergence,chaumontfrelet_valentin_2020a,hiptmair2011plane}.

On the other hand, ``asymptotic'' methods do not immediately 
tackle~\eqref{eq_helmholtz_intro}, but rather, numerically
approximate a simpler PDE (typically, an eikonal or kinetic
equation) which is a formal limit of~\eqref{eq_helmholtz_intro}
as $k \to \infty$, see e.g.~\cite{engquist_runborg_2003a}. These methods typically
exhibit a computational cost independent of $k$, since the limit
problem is itself frequency-independent. However, they can only reach
a fixed accuracy for any fixed value of $k$, and the validity of
the asymptotic model is typically only guaranteed under restrictive
theoretical assumptions which are hard to check in practice.

Recently, in~\cite{CFDI}, two of the authors introduced a
new full-wave method that
approximates~\eqref{eq_helmholtz_intro}, but using a discrete
space built with asymptotic properties of the solution in mind.
This method is based on Gabor wavelets and produces an essentially
sparse matrix,
but with only about $(kR)^{d-1/2}$ degrees of freedom.
The crucial ingredient is to employ shape functions of the form
\begin{equation}
\label{eq_intro_gabor}
\Psi_{k,\bm,\bn}(\bx)
\eq
\prod_{j=1}^d \Psi_{k,m_j,n_j}(x_j)
\eq
\left (
\frac{k}{\pi}
\right )^{d/4}
e^{-\frac{k}{2}|\bx-\bx^{k,\bm}|^2}
e^{ik \bxi^{k,\bn} \cdot (\bx-\bx^{k,\bm})},
\qquad
\bm,\bn \in \Z^d
\end{equation}
with $\bx^{k,\bm} \eq \sqrt{\pi/k} \bm$ and $\bxi^{k,\bn} \eq \sqrt{\pi/k} \bn$,
to build the discretisation space. These functions, called \emph{Gaussian coherent states},
have microlocalisation properties, meaning that both $\Psi_{k,\bm,\bn}$ and its
Fourier transform are localised, respectively around $\bx^{k,\bm}$ and $\bxi^{k,\bn}$.
This localisation aspect enables us to cherry-pick which wavelets are introduced
in the discretisation space, thereby only selecting functions that oscillate
around the natural frequency for solutions to~\eqref{eq_helmholtz_intro}.
Specifically, we only retain indices $\bm,\bn \in \Z^d$ for which
$\BB(\bx^{k,\bm}) \bxi^{k,\bn} \cdot \bxi^{k,\bn}-\nu(\bx^{k,\bm})$
is ``small''. This selection can be done inexpensively in a preprocessing
step before assembling the discretisation matrix.
It is of interest to note that contrary to Trefftz methods which employ
basis functions sitting in the kernel of the differential operator,
the proposed method samples the phase space ``around'' the kernel.

The method initially proposed in~\cite{CFDI} has, however, two major shortcomings
that the present work addresses. First, the family of Gaussian coherent states
$(\Psi_{k,\bm,\bn})_{[\bm,\bn] \in \Z^{2d}}$ is a \emph{frame} for $L^2(\R^d)$, but it is not a \emph{basis}. Such a construction is called
a \emph{Gabor frame}~\cite{Gro}, and from a numerical viewpoint,
it means that discretisations directly built from Gaussian coherent states
are intrinsically ill-conditioned (see in particular~\cite[Theorem 8.4.1]{Gro}).
Second, as far as theoretical error
estimates are concerned, we have only been able to employ the Gabor frame
within a least-squares weak formulation of~\eqref{eq_helmholtz_intro}, but
not in the standard Galerkin weak formulation, which further exacerbates the
conditioning issues in practice. In this work, we remove these two limitations
by employing a \emph{Wilson basis} rather than a Gabor frame.

In a Wilson basis, the Gabor functions are only employed as pairs, meaning
that the discrete space consists of functions
\begin{equation}
\label{eq_intro_wilson}
\Phi_{k,\bm,\bn}(\bx)
\eq
\prod_{j=1}^d(\Psi_{k,\bm_j,\bn_j} + (-1)^{\bm_j+\bn_j} \Psi_{k,\bm_j,-\bn_j})(x_j),
\end{equation}
for $\bm,\bn \in \Z^d$ with $\bn \geq 0$ componentwise. (Note that the definition and
index set we will take in the rest of the paper will slightly differ from~\eqref{eq_intro_wilson}:
see section \ref{subsec:Wilson}.)
The slight drawback of~\eqref{eq_intro_wilson} as compared to~\eqref{eq_intro_gabor}
is that the Fourier transform of $\Phi_{k,\bm,\bn}$ is localised around $2^d$ points
rather than one. However, this does not affect the discretisation method when the
PDE under consideration has a symmetric weak formulation
(as in~\eqref{eq_helmholtz_intro}), and only has a weak impact otherwise.
On the other hand, the crucial advantage of~\eqref{eq_intro_wilson}
is that the corresponding family of functions, called a Wilson basis, is a Riesz basis for
$L^2(\R^d)$. This fact completely removes the ill-conditioning issue associated with the
Gabor frame, and allows us to prove error estimates for the standard Galerkin
weak formulation.

Since we are proposing a ``volumic'' method, we do not immediately
approximate~\eqref{eq_helmholtz_intro}, but rather, its standard
reformulation using a perfectly matched layer (PML). This approach
is also known as ``complex scaling'', and is already routinely
employed within finite element and finite difference methods,
see e.g.~\cite{berenger_1994,galkowski2021perfectly,zworski2012semiclassical,}.
We therefore look for $u \in H^1(\R^d)$ such that
\begin{equation}
\label{eq_helmholtz_intro_PML}
-k^2 \mu u - \div(\BA\nabla u) = f \text{ in } \R^d,
\end{equation}
where $f \eq -\Delta \chi u^{\rm inc} + 2\nabla \chi \cdot \nabla u^{\rm inc}$
with a smooth compactly supported cutoff $\chi$ such that $\chi = 1$
on $B_R$. The coefficients $\mu$ and $\BA$ respectively correspond to
$\nu$ and $\BB$ in $B_R$, but are artificially modified outside $B_R$
to reformulate the radiation condition as the (numerically convenient)
decay condition that $u \in H^1(\R^d)$. This PML technique is designed
such that $u = u^{\rm tot}$ in $B_R$, which is the area where we wish
to compute the total field. If needed, the total field can also be
recovered from $u$ outside $B_R$ using integral representations~\cite{sauter_schwab_2010a}.

Let us assume for simplicity in this introduction that
the coefficients $\mu$ and $\BA$ correspond to a non-trapping setting
(see e.g.~\cite{galkowski_spence_wunsch_2020a}). Then, if $u^{\rm inc}$ is smooth and solves
$-k^2 u^{\rm inc}-\Delta u^{\rm inc} = 0$
in the support of $\nabla \chi$, our main results are as follows. For any
$\lambda>0$ and  $0<\varepsilon< \frac{1}{2}$, we let
\begin{equation*}
V \eq \left \{ \Phi_{k,\bm,\bn} \quad | \quad [\bm,\bn] \in \Z^{2d}; \;
|\BA(\bx^{k,\bm}) \bxi^{k,\bn} \cdot \bxi^{k,\bn}-\mu(\bx^{k,\bm})|
\leq
\lambda k^{-1/2+\varepsilon}
\right \}.
\end{equation*}
This space is inexpensive to assemble, and we have
$\dim V \sim_{k \to +\infty} (kR)^{1/2+\varepsilon}$. Then, for $k$
sufficiently large, there exists a unique $v \in V$
solution to the Galerkin weak formulation
\begin{equation*}
-k^2 \langle \mu v,w \rangle + \langle \BA\nabla v,\nabla w \rangle = \langle f,w \rangle
\qquad
\forall w \in V.
\end{equation*}
In addition, we have
\begin{equation}
\label{eq_intro_error_estimate}
\|u-v\|_{\HH^1_k} = O(k^{-\infty}) \|u^{\rm inc}\|_{\HH^1_k},
\end{equation}
meaning that the convergence is super-algebraic as the frequency $k$
increases. The notation in~\eqref{eq_intro_error_estimate} are rigorously
introduced below. Besides, the estimate in~\eqref{eq_intro_error_estimate}
remains valid under the weaker assumption that the scattering problem is
\emph{polynomially stable}, as per Assumption~\ref{Hyp:PolynResolv} below.
This assumption is known to be essentially true
generically~\cite{lafontaine2021most}, so that~\eqref{eq_intro_error_estimate}
is valid under mild assumptions. Alternatively,
for any fixed $k > 0$, an arbitrary level of accuracy might be achieved
by increasing $\lambda$, though we do not give a proof of this statement in the present paper.
We additionally show that the matrix $\textup{A}$ corresponding
to the Galerkin discretization of~\eqref{eq_helmholtz_intro_PML}
with subspace $V$ has a condition number bounded, in the nontrapping setting, as
\begin{equation*}
\|\textup{A}\|_{\ell^2 \to \ell^2} \|\textup{A}^{-1}\|_{\ell^2 \to \ell^2}
\leq
C k^{(1+\varepsilon)/2}.
\end{equation*}

The remainder of this work is organised as follows. In Section~\ref{sec:setting},
we make the setting precise, and rigorously state our main results.
In Section~\ref{section_gabor_wilson}, we recall the definition of Gabor
frames and Wilson basis, and establish some basic properties used throughout
this work. Section~\ref{section_parametrix} is dedicated to the proof of
a more complicated technical result related to Wilson bases. We prove the
main results stated in Section~\ref{sec:setting} in Sections~\ref{section_approximation}
and~\ref{section_convergence_galerkin}. Finally, we present numerical examples that
illustrate the theoretical findings in Section~\ref{section_numerics}.

\subsection*{Acknowledgements}
This project was partially funded by Inria, through the exploratory action POPEG,
and by the Agence Nationale de la Recherche, through the NuHeMiBa project (ANR-24-CE40-3760-01). 

\section{Setting and main results}
\label{sec:setting}

The proposed method actually applies in an abstract setting that is more general
than the scattering problem considered in~\eqref{eq_helmholtz_intro_PML}
in the introduction. In this section, we make this setting precise and
state our main results.

\subsection{Notations}
\label{section_h_notations}

\subsubsection{Frequency set}

Throughout this work $\hbar \in \LH \subset (0,1]$ will denote a small parameter.
When applying our general results to the Helmholtz equation
in~\eqref{eq_helmholtz_intro_PML}, we will have
$\hbar \sim (kR)^{-1}$ where $k$ is the wave number and $R$ is the size of the scatterer.
As a result, considering the set $(0,1]$ amounts to ignoring low frequencies, and focusing on
high frequencies corresponds to the regime $\hbar \to 0$. Our motivation for considering
a subset $\LH$ will be explained in Assumption~\ref{Hyp:PolynResolv} below. 
Notice that the case $\LH = (0,1]$ is not excluded.
In fact, many of the technical results of this paper hold for all $\h\in (0,1]$, and we will only need to consider $\h$ in a subset $\LH$ of $(0,1]$ when we apply Assumption \ref{Hyp:PolynResolv}.

\subsubsection{Basic notation}

Throughout, $d \geq 1$ is the number of space dimensions.
The canonical basis of $\R^d$ or of $\C^d$ will be denoted by
$(\boldsymbol{e}_1,..., \boldsymbol{e}_d)$. If $\bx,\by \in \C^d$, we write
$\bx \cdot \by \eq x_1 y_1 + \ldots + x_d y_d$ for the usual dot product
without complex conjugation on the second argument, so that
$|\bx| = (\bx \cdot \overline{\bx})^{1/2}$
is the usual Euclidean norm.

When $\epsilonvec\in \{-1, 1\}^d$ and $\bx\in \R^d$, we will write
\begin{equation}\label{eq:NotationChangementSigne}
\epsilonvec \bx := (\epsilon_1 x_1, \dots, \epsilon_d x_d)
\end{equation}
for the component-wise multiplication.

We will write $\N = \{0, 1, ...\}$ for the set of non-negative integers, and $\N^* = \N \setminus \{0\}$.

For a multi-index $\balpha \in \N^d$, $[\balpha] \eq \alpha_1+\dots+\alpha_d$
denotes its usual $\ell_1$ norm. If $v: \R^d \to \C$, the notation
\begin{equation*}
\partial^{\balpha} v
\eq
\frac{\partial^{\alpha_1}}{\partial x_1^{\alpha_1}}
\dots
\frac{\partial^{\alpha_d}}{\partial x_d^{\alpha_d}} v
\end{equation*}
is employed for the partial derivatives in the sense of distributions, whereas
$\bx^{\balpha} \eq x_1^{\alpha_1} \cdot \ldots \cdot x_d^{\alpha_d}$.
Finally, if $\bbeta \in \N^d$ is another multi-index, we will sometimes
need the notation
\begin{equation*}
\left (
\begin{array}{c}
\balpha \\ \bbeta
\end{array}
\right )
=
\prod_{j=1}^d
\left (
\begin{array}{c}
\alpha_j \\ \beta_j
\end{array}
\right ),
\end{equation*}
and the notation $\balpha \leq \bbeta$ means that $\alpha_j \leq \beta_j$
for all $j \in \{1,\dots,d\}$. 

If $\bn \in \Z^d$, we employ the notation $|\bn|^2 \eq n_1^2 + \dots + n_d^2$
for its $\ell_2$ norm. Finally, if $\Lambda \subset \Z^{2d}$, $\ell^2(\Lambda)$
has its usual definition, and we denote by $\|\cdot\|_{\ell^2(\Lambda)}$ its usual norm.

\subsubsection{Superalgebraically small quantities}

If $a(\h)$ is a quantity depending on $\h\in (0,1]$, we will write $a(\h)= O(\h^\infty)$ if,
for any $m\in \N$, we may find a constant $C_m \geq 0$ such that, for any $\h\in (0,1]$,
we have
\begin{equation}\label{eq:NotationOhInfini}
|a(\h)| \leq C_m \h^m.
\end{equation}
If the quantity $a(\h)$ depends on some other parameters $\bullet$ and $\circ$,
we will write $a(\h ; \bullet,\circ) = O_{\bullet}(\h^\infty)$ to underline
the fact that the constants $C_m$ in \eqref{eq:NotationOhInfini} also depend on
the parameters $\bullet$, but not on the parameters $\circ$. Finally, if $\LH'\subset (0,1]$,
we will say that ``$a(\h)= O(\h^\infty)$ for $\h\in \LH'$'' if the estimates~\eqref{eq:NotationOhInfini}
only hold for $\h\in \LH'$.

For other quantities $b(\h)$ and $c(\h)$, we will also respectively employ the notation
\begin{equation*}
c(\h) = O_{\bullet}(\h^\infty) \times b(\h),
\qquad
c(\h) = b(\h)+O_{\bullet}(\h^\infty),
\end{equation*}
whenever $|c(\h)| \leq |a(\h)| |b(\h)|$ and
$c(\h) = b(\h) + a(\h)$, with $a(\h)$ as above. 

Given $m_0\in \R$, we will adopt similar notations,
writing $a(\h) = O_\bullet (\h^{m_0})$ if~\eqref{eq:NotationOhInfini}
holds for $m=m_0$ with a constant depending only on $\bullet$.

\subsubsection{Key functional spaces}

In what follows, $L^2(\R^d)$ is the usual Lebesgue space of complex-valued square integrable
functions over $\R^d$. The usual norm and inner products of $L^2(\R^d)$ are respectively
denoted by $\|\cdot\|_{L^2}$ and $\langle \cdot,\cdot\rangle$.

Since we are dealing with the (unbounded) $\R^d$ space, following
\cite{chaumontfrelet_ingremeau_2022a}, our analysis will require the
weighted Sobolev spaces
\begin{equation*}
\HH_{\h}^p(\R^d) \eq \left \{
v \in L^2(\R^d) \; | \;
\bx^{\balpha} \partial^{\bbeta} v \in L^2(\R^d)
\quad \forall \balpha,\bbeta \in \N^d; \; [\balpha],[\bbeta] \leq p
\right \},
\end{equation*}
that we equip with the family of $\hbar$-weighted norms given by
\begin{equation*}
\|v\|_{\HH^p_\hbar}^2
\eq
\sum_{[\balpha] \leq p} \sum_{q \leq p- [\balpha]}
\hbar^{2[\balpha]} \||\bx|^q \partial^{\balpha} v\|_{L^2}^2
\end{equation*}
for all $p \in \N$.
If $\cO\subset \R^d$ is an open set, we will sometimes use the notation
\begin{equation*}
\|v\|_{H^p_\hbar(\cO)}^2
\eq
\sum_{[\balpha] \leq p}
\hbar^{2[\balpha]} \|\partial^{\balpha} v\|_{L^2(\cO)}^2
\end{equation*}
for all $p \in \N$, where $\|{\cdot}\|_{L^2(\cO)}$ has its usual meaning.

We also denote by $C^0(\R^d)$ the set of complex-valued continuous functions
defined over $\R^d$, and by $C^\ell(\R^d)$ the set of functions $v: \R^d \to \C$ such that
$\partial^{\balpha} v \in C^0(\R^d)$ for all $\balpha \in \N^d$ with $[\balpha] \leq \ell$.
We introduce the notation
\begin{equation*}
\|v\|_{C^\ell}
\eq
\max_{\substack{\balpha \in \N^d \\ [\balpha] \leq \ell}}
\sup_{\bx \in \R^d} |(\partial^{\balpha} v)(\bx)|
\qquad \forall v \in C^\ell(\R^d)
\end{equation*}
and $C^\ell_{\rm b}(\R^d)$ is the subset of functions $v \in C^\ell(\R^d)$
such that $\|v\|_{C^\ell(\R^d)} < +\infty$. We also set
\begin{equation*}
C^\infty_{\rm b}(\R^d) \eq \bigcap_{\ell \in \N} C^\ell_{\rm b}(\R^d).
\end{equation*}

\subsection{Gaussian coherent states, Gabor frames and Wilson bases}
\label{section_frame}

\subsubsection{Gaussian coherent states}

For $[\bm,\bn] \in \Z^{2d}$, we define the Gaussian state
\begin{equation}\label{eq:DefGaborState}
\gshmn(\bx)
\eq
(\pi\hbar)^{-d/4}
e^{-\frac{1}{2\hbar}|\bx-\xxhm|^2} e^{\frac{i}{\hbar}\xihn \cdot (\bx-\xxhm)},
\end{equation}
where $\xxhm \eq \sqrt{\pi\hbar} \bm$ and $\xihn \eq \sqrt{\pi\hbar} \bn$.
These functions are $L^2$-normalised, and it was proved in \cite{daubechies_grossman_meyer_1986a} that the family of Gaussian coherent states $(\gshmn)_{[\bm,\bn] \in \Z^{2d}}$
forms a \emph{frame} for $L^2(\R^d)$, meaning
there exist two constants $0 < \alpha_{\mathcal{G}} \leq  \beta_{\mathcal{G}} < +\infty$ such that 
\begin{equation}\label{eq:FrameBound}
\alpha_{\mathcal{G}} \|v\|_{L^2}^2
\leq
\sum_{[\bm,\bn] \in \Z^{2d}} |\langle v,\gshmn \rangle|^2
\leq
\beta_{\mathcal{G}} \|v\|_{L^2}^2 \qquad \forall v \in L^2(\R^d).
\end{equation}
Here, the constants $\alpha_{\mathcal{G}}$ and $\beta_{\mathcal{G}}$ can be chosen
independent of $\h$: this follows from the fact that, for any $\h>0$, the operator
$I_\hbar : L^2(\R^d) \longrightarrow L^2(\R^d)$ given by
$(I_{\hbar} u)(\xvec) := \hbar^\frac{d}{4} u\left(\sqrt{\hbar} \xvec\right)$
is a bijective isometry such that $\Psi_{1, \bm, \bn} = I_{\hbar}(\Psi_{\h, \bm, \bn})$.

Actually, the frame property implies (see \cite[Chapter 5]{Gro}) that there exists another family of functions
$(\gshmn^\star)_{[\bm,\bn] \in \Z^{2d}}$ called the dual frame such that
\begin{equation}
\label{eq_dual_frame}
v
=
\sum_{[\bm,\bn] \in \Z^{2d}} \langle v,\gshmn^\star \rangle \gshmn.
\end{equation}
Here, the sum is to be understood in the sense of unconditional convergence,
as in~\cite[Chapter 5.3]{Gro}.

However, the family $(\gshmn)_{[\bm,\bn] \in \Z^{2d}}$ is \emph{not} a Riesz basis,
so that the decomposition~\eqref{eq_dual_frame} of $v$ as a sum of $\gshmn$ is not unique,
as discussed in Remark~\ref{rem:IllConditionned} below. The aim of this paper is to overcome
this issue by working with another family of functions, called the \emph{Wilson states}.

\subsubsection{Wilson states}\label{subsec:Wilson}

The Wilson states will not be indexed by $\Z^{2d}$, but by a slightly more complicated set,
which is essentially ``$2^d$ times smaller''. With the notation 
$E := (\Z \times \N^*) \cup (2 \Z \times \{0\})$, this index set is defined
as
\begin{equation*}
E_d
:=
\left \{
[\mvec,\nvec] \in \Z^{2d} \; | \;
[m_j, n_j] \in E, \;
\forall j \in \llbracket 1, d \rrbracket
\right \}.
\end{equation*}

When $d=1$, the Wilson states are then defined, for all $[m, n] \in E$, as
\begin{equation*}
\forall x \in \R, \quad \w{m, n}(x) := c_n \left(\eg{m, n}(x) + (-1)^{m + n} \eg{m, -n}(x)\right),
\end{equation*}
where $c_0 := \frac{1}{2}$ and $c_n := \frac{1}{\sqrt{2}}$ for all $n \in \N^*$.
When $d\geq 2$,  the Wilson states are defined analogously,  for all $[\mvec, \nvec ] \in E_d$, as
\begin{equation}\label{eq:DefWilsonState}
\begin{aligned}
\forall \xvec \in \R^d, \quad \w{\mvec, \nvec}(\xvec)
&:=
(\w{m_1, n_1} \otimes \dots \otimes \w{m_d, n_d})(\xvec)
\\
&=
c_\nvec \sum_{\epsilonvec \in \{-1, 1\}^d}
\epsilonvec^{\mvec + \nvec} \eg{\mvec, \epsilonvec \nvec}(\xvec),
\end{aligned}
\end{equation}
where we used notation \eqref{eq:NotationChangementSigne}, and we set
$c_\nvec := \prod\limits_{j = 1}^d c_{n_j}$,
and $\epsilonvec^{\mvec + \nvec} := \prod\limits_{j = 1}^d \epsilon_j^{m_j + n_j}$.

The family $(\w{\mvec, \nvec})_{[\mvec, \nvec]
 \in E_d}$ then forms a \emph{Riesz basis},
which means that it satisfies bounds analogous to \eqref{eq:FrameBound}, but that furthermore,
the expansion \eqref{eq_dual_frame} is unique. These functions were introduced
in~\cite{wilsongeneralized}, but the Riesz basis property was only proved later
in~\cite{daubechies1991simple}.

Since the family of Wilson states is a frame, it has an associated dual frame,
and we recall here the relation between the dual frames
$\Psi_{\h,\bm,\bn}^\star$ and $\Phi_{\h,\bm,\bn}^\star$. Namely, for any
$[\mvec, \nvec]\in E_d$, it follows from the definition of the dual frames
in terms of frame operators \cite[Corollary 5.1.3]{Gro} along with the
relation between frame operators of Gabor and Wilson systems~\cite[Theorem 3.1]{bownik2017wilson}
that we have
\begin{equation}
\label{eq:ExpressionDualBasis}
\Phi_{\h, \mvec, \nvec}^\star
=
2^d c_{\nvec}
\sum_{\epsilonvec \in \{-1, 1\}^d}
\epsilonvec^{\mvec + \nvec} \eg{\mvec, \epsilonvec \nvec}^\star.
\end{equation}

\subsection{Settings and key assumptions}
\label{ssec:general}

Throughout this work, we consider a second order differential operator on
$\R^d$ depending on $\hbar$, and taking the form
\begin{equation}
\label{eq:FormeGeneraleOperateur}
(P_{\hbar} v)(\bx)
=
-\hbar^2
\sum_{j,\ell=1}^d
a_{j\ell}^{\hbar}(\bx)
\frac{\partial^2 v}{\partial x_j \partial x_\ell}(\bx)
+
i\hbar\sum_{j=1}^d b_{j}^\hbar(\bx) \frac{\partial v}{\partial  x_j }(\bx)
+
c^\hbar(\bx) v(\bx),
\end{equation}
where $a_{j\ell}^\hbar,b_j^\hbar,c^\hbar \in \CIB(\R^d)$ for $1 \leq j,\ell \leq d$.
Here, the coefficients are allowed to depend on $h$, though in many applications, they don't.

Throughout this work, we consider smooth and bounded coefficients.
Specifically, we demand the following.

\begin{assumption}\label{Hyp1:BoundCoefs}
For every $p\in \N$, the following quantity is finite:
\begin{equation*}
C_{{\rm coef},p}
\eq
\sup_{\hbar \in (0,1]}
\left (
\sum_{j,\ell=1}^d \|a_{j,\ell}^\hbar\|_{C^p}
+
\sum_{j=1}^d \|b_j^\hbar\|_{C^p}
+
\|c^\hbar\|_{C^p}
\right ) < +\infty.
\end{equation*}
\end{assumption}

We note that under the above assumption, the operator
$P_\h : \widehat{H}_\hbar^{p+2}(\R^d) \longrightarrow \widehat{H}_\hbar^{p}(\R^d)$
is bounded for all $p \in \N$. The proof is completely standard in non-weighted
Sobolev spaces, and for completeness, we go through the case of weighted spaces
in Appendix~\ref{app:ProofContinuity} below.

\begin{lemma}
\label{lemme: P est continu}
Under Assumption~\ref{Hyp1:BoundCoefs}, for any $p \in \N$,
there exists $C_p>0$ such that for all $\h \in (0,1]$ and all 
$w \in \widehat{H}_\h^{p+2}(\R^d)$, we have
\begin{align*}
\|P_\h w\|_{\widehat{H}_\hbar^p} \leq C_p \|w\|_{\widehat{H}_\h^{p+2}}.
\end{align*}
\end{lemma}

Along with the smoothness and boundedness of the coefficients, we make three key assumptions.

\begin{assumption}\label{Hyp:PolynResolv}
We assume that there exist $N_0\geq 1$ and a set $\LH \subset (0,1]$
such that the following holds. For every $p\in \N$, there exists $C_{{\rm sol},p}>0$
such that, for every $\hbar\in \LH$, the operator $P_\hbar: \HH_{\h}^{p+2}(\R^d) \to \HH_{\h}^p(\R^d)$
 is invertible, and we have 
\begin{equation*}
\|P_{\hbar}^{-1}\|_{\HH^p_{\h} \to \HH^p_{\h}} \leq C_{{\rm sol},p} \hbar^{-N_0}.
\end{equation*}
\end{assumption}

This assumption is satisfied for instance when the underlying classical dynamics
has no trapped set (with $N_0 = 1$ and $\LH = (0,1]$), or when we exclude a
set of exceptional frequencies: this is the reason why we consider a subset
$\LH$ of $(0,1]$. We refer the reader to \cite[Remark 4.1]{CFDI} for more details
and references.

The symbol of $P_\hbar$ is the function $p_\hbar \in C^\infty(\R^{2d})$ defined by
\begin{equation}
\label{eq_symbol}
p_\hbar(\bx,\bxi)
\eq
\sum_{j,\ell=1}^d
a_{j,\ell}^\hbar(\bx) \xi_j \xi_\ell
+
\sum_{j=1}^d b_j^\hbar(\bx) \xi_j
+
c^\hbar(\bx).
\end{equation}

An essential role will be played by sets of the form
\begin{equation*}
\{ (\bx,\bxi) \in \R^{2d} \; | \; |p_{\h}(\bx,\bxi)| \leq C\sqrt{\h} \}.
\end{equation*}
Such sets will sometimes be called \emph{characteristic sets} (though the expression
will be used only for the discussion, and we will not need a precise definition),
and our next assumption ensures us that they are compact for $\h$ small enough.

\begin{assumption}\label{Hyp:BoundedLayers}
There exist $\delta_0, D_0 > 0$ such that
\begin{equation}
\label{eq_assumption_symbol_bounded}
\forall \hbar \in (0,1], ~~ 
\{(\bx,\bxi)\in \R^{2d} ; \; |p_\hbar(\bx,\bxi)| < \delta_0\} \subset \B(\bzero, D_0).
\end{equation}
\end{assumption}

Finally,  for some of our results, we will need a last assumption of quasi-symmetry of the symbol. 
\begin{assumption}\label{Hyp4:SymetrieSymbole}
There exists $\Csym \geq 0$ such that, recalling notation \eqref{eq:NotationChangementSigne},
we have for all $\h±\in (0,1]$
\begin{equation}
\label{eq: hypothèse sur le symbole}
\forall (\xvec, \xivec) \in (\R^d)^2,
\quad
\forall \epsilonvec \in \{-1, 1\}^d,
\quad
|p_\hbar(\xvec, \xivec) - p_\hbar(\xvec, \epsilonvec \xivec)| \leq \Csym \hbar |\xivec|.
\end{equation}
\end{assumption}

We emphasise that all these abstract assumptions are satisfied for the
Helmholtz model problem mentioned above in~\eqref{eq_helmholtz_intro_PML}.
This is detailed in our previous work~\cite{CFDI}, and we do not
reproduce the arguments here for shortness.

In the remainder of this work, we allow generic constants $C$ (including the ones
in the $O(\h^\infty)$ notations) to depend on $\Csym$,
$N_0$ and $D_0$, and on a finite number of coefficients $\{C_{{\rm coef},p}\}_{p \in \N}$,
$\{C_{{\rm sol},p}\}_{p \in \N}$. They may also depend on the parameters
$\alpha, \lambda, \lambda',\lambda''$ and $\varepsilon$ appearing
in the statement of the Theorems below. We also employ the notation $C_{\bullet}$
if the constant $C$ is additionally allowed to depend on other previously
introduced quantities $\bullet$.

\subsection{Statement of the approximability result}
\label{ssec:statement}

Let $0 < \varepsilon < 1/2$, let $\lambda>0$ and $0 \leq \alpha < \delta_0/2$.
We define the index set
\begin{equation}\label{eq:EnergyLayers}
\Lambda_\h(\lambda) = \Lambda_\h(\lambda ; \varepsilon,  \alpha) :=  \left \{
[\bm,\bn] \in E_{d}
\; | \;
\min_{\epsilonvec\in \{-1,1\}^d}|p_\hbar(\xxhm,\epsilonvec\xihn)| < \alpha + \lambda \hbar^{1/2-\varepsilon}
\right \},
\end{equation}
and the corresponding space
\begin{equation*}
W_\h(\lambda) := \mathrm{Vect} \left\{ \w{\mvec, \nvec} ; [\mvec, \nvec]\in \Lambda_\h(\lambda) \right\}.
\end{equation*}

Our first result shows that, if one considers the equation $P_\h u_\h = f_\h$
with $f_\h\in W_\h(\lambda)$, then for any $\lambda'> \lambda$,  the solution $u_\h$
may be very well approximated in $W_\h(\lambda')$.

\begin{remark}
Note that, given any such triplet $(\lambda,\varepsilon,\alpha)$,
Assumption~\ref{Hyp:BoundedLayers} implies that for all $\h$ small enough,
the set $\Lambda_\h(\lambda;\varepsilon,\alpha)$ is finite, with a cardinal growing
at most as $O_{\h\to 0}(\h^{-d})$.  Actually, in many situations, this cardinal is much
smaller. For instance, suppose that $p_\h =p_0 +O(\h)$, where $p_0$ does not depend on $\h$,
and is non-degenerate, in the sense that
\begin{equation*}
\inf_{(\xvec, \xivec)\in p_0^{-1}(\{0\})} \min_{\epsilonvec\in \{-1,1\}^d}
|\nabla p_{0}(\xvec, \epsilonvec\xivec)| >0.
\end{equation*}
 We then have
\begin{equation*}
\mathrm{Card} \left(  \Lambda_\h(\lambda ; \varepsilon,  0)  \right) = O_{\h \to 0} \left(\h^{-d + \frac{1}{2}- \varepsilon}\right).
\end{equation*}
This is in particular the case for the scattering problem in~\eqref{eq_helmholtz_intro_PML},
so that the discretization spaces we consider are much smaller than in standard
volumic methods.
\end{remark}

\begin{remark}
\label{remark_symetric_symbol}
We also point out that when Assumption~\ref{Hyp4:SymetrieSymbole} holds,
then the set 
\begin{equation*}
\Lambda'_\h(\lambda) := \left \{
[\bm,\bn] \in E_{d}
\; | \;
|p_\hbar(\xxhm,\xihn)| < \alpha + \lambda\hbar^{1/2-\varepsilon}
\right \}
\end{equation*}
differs from $\Lambda_\h(\lambda)$ by few elements, and we could work with $\Lambda'_\h(\lambda)$ instead of $\Lambda_\h(\lambda)$.
\end{remark}

\begin{theorem}
\label{theorem_approximability}
Suppose that Assumptions~\ref{Hyp1:BoundCoefs},~\ref{Hyp:PolynResolv}
and~\ref{Hyp:BoundedLayers} hold. Let $0 < \varepsilon < 1/2$, $0 \leq \alpha < \delta_0/2$
and $\lambda'> \lambda>0$. For every $\h\in \LH$, consider a function $f_h\in W_\h(\lambda)$.
Then, for every $p\in \N$, the solution of $P_\h u_\h = f_\h$ satisfies
\begin{equation}
\label{eq_approximability_estimate}
\left \|
u_\hbar
-
\sum_{[\bm, \bn] \in \Lambda_{\hbar}(\lambda')} \langle u_\hbar,\w{\mvec, \nvec}^\star\rangle \w{\mvec, \nvec}
\right \|_{\HH_\hbar^p}
= O_{p}(\h^\infty) \|f_\h\|_{L^2}
\end{equation}
for all $\h\in \LH$ small enough that $\Lambda_{\hbar}(\lambda')$ is finite.
\end{theorem}

An analogue of Theorem~\ref{theorem_approximability} was proven in~\cite{CFDI},
for the Gaussian states $\gshmn$ rather than for the Wilson states $\w{\mvec, \nvec}$.
We also have the following estimate, which immediately follows
from the proof of Theorem~\ref{theorem_approximability} (see
Section~\ref{subsec:ApproxInGoodSpace}), a triangle inequality,
Assumption~\ref{Hyp:PolynResolv}, and Corollary~\ref{corollary_rhs_uinc}
below.

\begin{corollary}
\label{corollary_uinc}
Under the assumption of Theorem~\ref{theorem_approximability},
consider a bounded open set $\cO\subset \R^d$ such that
\begin{equation*}
\delta \eq 
\dist(\cO,\Sigma^{\rm c})
>
0,
\quad
\text{with}
\quad
\Sigma
\eq
\{
\bx \in \R^d \; | \;
\forall \h \in \LH, \forall \bxi \in \R^d, ~~
p_\h(\bx,\bxi) = |\bxi|^2-1
\}
\end{equation*}
and a cutoff $\chi \in C^\infty_{\rm c}(\cO)$.
Consider an incident field $u_\h^{\rm inc} \in H^1(\cO)$ with
$\h^2\Delta u_\h^{\rm inc}+u_\h^{\rm inc} = 0$ in $\cO$.
Then, if $f_\h = (1+\h^2\Delta)(\chi u_\h^{\rm inc})$
and $P_\h u_\h = f_\h$, whenever $\h\in \LH$ is small
enough so that $\Lambda_{\hbar}(\lambda)$ is finite,
we have
\begin{equation}
\label{eq_approximability_estimate_uinc}
\left \|
u_\hbar
-
\sum_{[\bm, \bn] \in \Lambda_{\hbar}(\lambda)} \langle u_\hbar,\w{\mvec, \nvec}^\star\rangle \w{\mvec, \nvec}
\right \|_{\HH_\hbar^p}
=
O(\h^\infty) \|u_\h^{\rm inc}\|_{H^1_\h(\cO)}
\end{equation}
for all $p \in \N$ and $\lambda,\varepsilon > 0$, where the implied constants in the
$O(\h^\infty)$ notation only depends on $\cO$, $\varepsilon$, $\lambda$, $p$, the norms
$\|\chi\|_{W^{\ell,\infty}(\cO)}$ for $\ell \in \N$, $\delta$, and the smallest $R$
such that $|\bx| \leq R$ for all $\bx \in \cO$.
\end{corollary}

\subsection{Convergence of the Galerkin method}

Given a finite-dimensional subspace $W_\h\subset H^2(\R^d)\subset L^2(\R^d)$,
and given $f_\h\in L^2(\R^d)$, we will be interested in the Galerkin problem:
\begin{equation}\label{eq:DefProblemGalerkin}
\text{Find } v_\h\in W_\h \text{ such that } \forall w_\hbar \in W_\h, \quad \langle P_\hbar v_\hbar, w_\hbar\rangle = \langle f_\hbar, w_\hbar\rangle.
\end{equation}

Note that, if we denote by $\Pi_{W_\h}$ the $L^2$-orthogonal projection on $W_\h$,
and by $T_{W_\h} : W_\h \longrightarrow W_\h$ the operator given by
\begin{equation}
T_{W_\h}= \Pi_{W_\h} \circ P_\h,
\end{equation}
then the problem~\eqref{eq:DefProblemGalerkin} may be rewritten as:
Find $v_\h\in W_\h$ such that
\begin{equation*}
T_{W_\h} v_\h = \Pi_{W_\h} f_\h.
\end{equation*}

Our main theorem ensures that this problem is well-posed,
and that its solution is close to the real one.

\begin{theorem}
\label{theorem_accuracy}
Suppose that
Assumptions~\ref{Hyp1:BoundCoefs},~\ref{Hyp:PolynResolv},~\ref{Hyp:BoundedLayers}
and~\ref{Hyp4:SymetrieSymbole} hold. Let $0 < \varepsilon < 1/2$, $0 \leq \alpha < \delta_0/2$
and $\lambda'> \lambda>0$. There exists $\h_0>0$ such that the following
holds for every $\h \in \LH \cap (0, \h_0]$.

For every $\h \in \LH \cap (0, \h_0]$, consider a function $f_h\in L^2(\R^d)$
and a set of indices $\Lambda_\h$ such that
\begin{equation*}
\Lambda_\h(\lambda')
\subset
\Lambda_\h
\subset
\left \{
[\bm,\bn] \in E_{d}
\; | \;
|p_\hbar(\xxhm,\xihn)| < \delta_0
\right \}. 
\end{equation*}
  
Then there exists a unique
$v_\h \in W_\h := \mathrm{Vect} \left\{ \w{\mvec,\nvec} ; [\mvec,\nvec] \in \Lambda_\h \right\}$
solution to~\eqref{eq:DefProblemGalerkin}.
Furthermore, if $f_\h\in W_\h(\lambda)$ and if $u_\h$ is the solution of $P_\h u_\h = f_\h$, then for any $p\in \N$, we have
\begin{equation}
\label{eq: convergence asymptotique pour Galerkine Wilson}
\|u_\hbar-v_\hbar\|_{\widehat{H}_\hbar^p} = O_{p} (\h^\infty) \|f_\hbar\|_{L^2}.
\end{equation}
\end{theorem}

The above result remains true if the right-hand side corresponds
to an incident field, as in Corollary~\ref{corollary_uinc}. This is a
direct consequence of Assumption~\ref{Hyp:PolynResolv} and
Corollary~\ref{corollary_rhs_uinc}, together
with~\eqref{eq: convergence asymptotique pour Galerkine Wilson}
and a triangle inequality.

In addition to the accuracy properties stated above, the proposed method
also enjoys favorable conditioning properties, as we establish below.
This is in sharp contrast with the earlier iteration of the method
proposed by the authors in~\cite{CFDI}. We further note that for the
case $\alpha = 0$, the condition number is essentially $\h^{1/2}$ times better than
what should naturally be expected.
This is because employing only Wilson states close to the characteristic sets reduces the continuity
constant of $T_{W_h}$.

\begin{theorem}
\label{theorem_conditioning}
We make the same assumptions as in  Theorem~\ref{theorem_accuracy}.
We let $\lambda''>0$ (possibly depending on $\h$), be such that for all $\h\in (0,1]$,
\begin{equation}\label{eq:CouchePasTropGrande}
\Lambda_\h \subset \Lambda_{\h}(\lambda'')\subset
\left \{
[\bm,\bn] \in E_{d}
\; | \;
|p_\hbar(\xxhm,\xihn)| < \delta_0
\right \}.
\end{equation}
We then have
\begin{equation}
\label{eq:BorneTh}
\left \| T_{W_\h} \right \|_{L^2 \to L^2}
\leq
M \left ( \alpha + (1+\lambda'') h^{(1-\varepsilon)/2} \right ),
\end{equation}
where the constant $M > 0$ does not depend on $\alpha, \lambda, \lambda', \lambda''$.
Besides if $N_0$ is as in Assumption~\ref{Hyp:PolynResolv},
there exists $\h_1 > 0$ such that, for all $\h \in \LH \cap (0, \h_1]$,
we have
\begin{equation}
\label{eq:BorneConditionnementTh}
\left \|T_{W_\h}^{-1} \right \|_{L^2 \to L^2} \leq C \hbar^{-N_0}.
\end{equation}
In particular, if $\alpha = 0$ and if $\lambda''$ is bounded independently of $\h$, we have
\begin{equation*}
\left \|T_{W_\h}      \right \|_{L^2 \to L^2}
\left \|T_{W_\h}^{-1} \right \|_{L^2 \to L^2}
\leq
C \h^{1/2-N_0-\varepsilon/2},
\end{equation*}
and in this case, we also have the condition upper bound
\begin{equation*}
\|\textup{A}     \|_{\ell^2 \to \ell^2}
\|\textup{A}^{-1}\|_{\ell^2 \to \ell^2}
\leq
C \h^{1/2-N_0-\varepsilon/2},
\end{equation*}
for the Galerkin matrix $\textup{A}$ with entries given by
\begin{equation*}
\textup{A}_{[\bm, \bn], [\bm', \bn']}
\eq
\langle P_\h \w{\bm,\bn},\w{\bm',\bn'} \rangle
\end{equation*}
for all $[\bm, \bn], [\bm', \bn'] \in \Lambda_\h$.
\end{theorem}

\section{Properties of Gabor frames and Wilson bases}
\label{section_gabor_wilson}

The aim of this section is to prove properties of Gaussian and Wilson states.
Many of these properties have been stated for Gaussian states in~\cite{CFDI}
and~\cite{chaumontfrelet_ingremeau_2022a}, and we adapt these results to the
Wilson states here.

\subsubsection*{Notations for indices}

In the remainder of this work, it will often be convenient to denote by
$\qvec$ the elements of $\Z^{2d}$ or of $E_d$, rather than by $[\mvec, \nvec]$.
We will sometimes write $\mvec_\qvec$ and $\nvec_{\qvec}$ for the first $d$
(resp. the last $d$) coordinates of $\qvec$. Similarly, we will often write
\begin{equation*}
\bz^{\h, \bq}= (\xm{\mvec_\qvec}, \xin{\nvec_\qvec}) \in \R^{2d}.
\end{equation*}

It will be useful to work with the following pseudometric on $\Z^{2d}$:
for every $\qvec \in (\Z^d)^2$ and $\svec = [\mvec_\svec,\nvec_\svec] \in (\Z^d)^2$,
we set
\begin{equation*}
\delta(\qvec, \svec) := \min_{\epsilonvec \in \{-1, 1\}^d} |\qvec - [\mvec_\svec, \epsilonvec \nvec_\svec]|.
\end{equation*}
Note that, if $\qvec, \svec \in E_d$, then we simply have $\delta(\qvec, \svec) = |\qvec - \svec|$.

\subsubsection*{Norms of Wilson states}

Our first lemma is a direct consequence of~\cite[Lemma C.1]{CFDI} and of~\eqref{eq:DefWilsonState}.

\begin{lemma}
\label{lemme: majoration de la norme H de phi}
For every $p \in \N$,  every $\qvec \in E_d$ and every $\hbar \in (0, 1]$, we have
\begin{equation}
\label{eq: majoration de la norme H de phi}
\|\w{\qvec}\|_{\widehat{H}_\hbar^p}
\leq
C_p \left(1 + |\bz^{\h,\bq}|^2\right)^{\frac{p}{2}}
\leq
C_p (1 + |\qvec|)^p.
\end{equation}
\end{lemma}

\subsubsection*{Frame and Riesz basis properties}

We recall that the family $(\Phi_{\h,\qvec})_{\qvec\in E_d}$ forms a frame,
meaning that there exist $0< \alpha_\mathcal{W} < \beta_\mathcal{W}$ such that,
for any $u\in L^2(\R^d)$, we have
\begin{equation}
\label{eq: inégalité de trame pour Wilson}
\alpha_\mathcal{W} \|u\|_{L^2}^2
\leq
\sum_{\qvec \in E_d} |\langle u, \w{\qvec}\rangle |^2
\leq
\beta_\mathcal{W} \|u\|_{L^2}^2.
\end{equation}
Here, the constants $\alpha_\mathcal{W}$ and $\beta_\mathcal{W}$ can be chosen
independent of $\h$, by the argument following \eqref{eq:FrameBound}. 

We recall that the frame property~\eqref{eq: inégalité de trame pour Wilson}
implies the existence of a dual frame $(\w[\hbar]{\qvec}^\star)_{\qvec\in E_d}$,
such that, for any $u \in L^2(\R^d)$, we have
\begin{equation}
\label{eq: inégalité de trame pour Wilson étoile}
\frac{1}{\beta_\mathcal{W}} \|u\|_{L^2}^2
\leq
\sum_{\qvec \in E_d} |\langle u, \w{\qvec}^\star\rangle |^2
\leq
\frac{1}{\alpha_\mathcal{W}} \|u\|_{L^2}^2,
\end{equation}
and we may write
\begin{equation}
\begin{aligned}
\label{eq_wilson_decomposition}
u
=
\sum_{\qvec \in E_d} \langle u, \w{\qvec}^\star\rangle  \w{\qvec}
=
\sum_{\qvec \in E_d} \langle u, \w{\qvec}\rangle  \w{\qvec}^\star.
\end{aligned}
\end{equation}

The functions $\w{\qvec}^\star$ are bounded independently of $\h$ and $\bq$:
there exists a constant $C$ depending solely on the dimension $d$ such that
\begin{equation}\label{eq:BornePhiStar}
\|\w{\qvec}^\star\|_{L^2} \leq C.
\end{equation}

As already mentioned, the family $(\Phi_{\h, \qvec})_{\qvec\in E_d}$ has the additional
property of forming a \emph{Riesz basis}, which means that, in~\eqref{eq_wilson_decomposition},
the decompositions are unique.
The Riesz basis property may be better understood by introducing the \emph{reconstruction operator} $\cD_\h : \ell^2(E_d) \longrightarrow L^2(\R^d)$, defined as follows: for any $U_\h\in \ell^2(E_d)$,  we set
\begin{equation}\label{eq:DefOperateurReconstruction}
\cD_\h U_\h := \sum_{\qvec\in E_d} U_\h(\qvec) \Phi_{\h, \qvec},
\end{equation}
and we then have, for all $U_\hbar\in \ell^2(E_d)$
\begin{equation}
\label{eq: relation entre la norme de U et celle de u pour Wilson}
\sqrt{\alpha_\mathcal{W}} \big\|U_\hbar\big\|_{\ell^2(E_d)}
\leq
\|\cD_\h U_\h\|_{L^2}
\leq
\sqrt{\beta_\mathcal{W}} \big\|U_\hbar\big\|_{\ell^2(E_d)}.
\end{equation}
Indeed, by uniqueness of the decomposition~\eqref{eq_wilson_decomposition},
we must have $U_h(\qvec)= \langle \cD_\h U_\h, \w{\qvec}^\star\rangle$, so
that~\eqref{eq: relation entre la norme de U et celle de u pour Wilson}
follows from~\eqref{eq: inégalité de trame pour Wilson étoile}.

\begin{remark}
\label{rem:IllConditionned}
Note that, if we define the reconstruction operator for the Gaussian Gabor
frame $\wit{\cD}_\h : \ell^2(\Z^{2d}) \longrightarrow L^2(\R^d)$ as
\begin{equation}\label{eq:DefOperateurReconstructionGaussian}
\forall U_\h\in \ell^2(\Z^{2d}),~~~~~~
\wit{\cD}_\h U_\h := \sum_{\qvec\in \Z^{2d}} U_\h(\qvec) \Psi_{\h, \qvec},
\end{equation}
then we have 
\begin{equation}
\label{eq: relation entre la norme de U et celle de u pour Gabor}
\|\wit{\cD}_\h U_\h\|_{L^2} \leq \sqrt{\beta} \big\|U_\hbar\big\|_{\ell^2(\Z^{2d})}
\end{equation}
as in~\eqref{eq: relation entre la norme de U et celle de u pour Wilson}.
However, the analogue of the lower bound
in~\eqref{eq: relation entre la norme de U et celle de u pour Wilson}
does not hold: the operator $\wit{\cD}_\h$ has a non-trivial kernel. 

When $U_\h \in \ell^2(\Z^{2d})$ is a sequence with a finite support
and $\|U_\h\|_{\ell^2} = 1$, we cannot have
$\wit{\cD}_\h(U_\h)=0$ in $L^2(\R^d)$. However, the function
$\wit{\cD}_\h(U_\h)$ can have an extremely
small $L^2$ norm when the support of $U_\h$ is wide. Indeed, it was shown
in \cite{grochenig2015linear} that, for any $\tau > 0$,  we have
\begin{equation*}
\inf_{U_\h\in \ell^2 ([-N,N]^{2d})}
\frac{\|\wit{\cD}_\h(U_\h)\|_{L^2}}{\|U_\h\|_{\ell^2}}
\leq
C_\tau e^{-c_\tau N^{1-\tau}}.
\end{equation*}
These comments are also of importance in numerical applications,
because they imply that Galerkin matrices resulting from (truncated)
Gaussian Gabor frames are fundamentally ill-conditioned. In contrast,
the Wilson states considered in this work avoid this pitfall.
\end{remark}

The following elementary lemma will be used many times in the sequel.

\begin{lemma}
\label{lem:LePtiLemKiSauveToutLTemps}
Let $K\subset \R^{2d}$ be a bounded set, $\h_{max}>0$, and let $\Lambda_\h \subset E_d$
be a family of sets depending on $\h$ such that, for all $ \h\leq \h_{max}$ and all
$\bq \in \Lambda_\h$, we have $\bz^{\h, \bq} \in K$.
Then, for all $\h \leq \h_{max}$, we have $\mathrm{Card}(\Lambda_h) \leq C_K \h^{-d}$. Furthermore,
for any $p\in \N$, there exists $C_{K,p}>0$ such that, for any $\h\in (0, \h_{max}]$, we have
\begin{equation*}
\forall U_h\in \ell^2(\Lambda_\h), ~~~~
\|\cD_\h U_\h\|_{\widehat{H}_\hbar^p}
\leq
C_{K,p} \h^{-\frac{d}{2}} \big\|U_\hbar\big\|_{\ell^2(\Lambda_\h)}
\leq
C'_{K,p}  \h^{-\frac{d}{2}} \|\cD_\h U_\h\|_{L^2}.
\end{equation*}
\end{lemma}

\begin{proof}
The estimate on $\mathrm{Card}(\Lambda_h)$ simply comes from the definition of
$\bz^{\h,\bq}$. If $U_h\in \ell^2(\Lambda_\h)$, we have
\begin{equation*}
\|\cD_\h U_\h\|_{\widehat{H}_\hbar^p}
\leq
\sum_{\bq\in \Lambda_\h} |U_\h(\bq)|    \|\w{\qvec}\|_{\widehat{H}_\hbar^p}
\stackrel{\eqref{eq: majoration de la norme H de phi}}{\leq}
C_{K,p} \sum_{\bq\in \Lambda_\h} |U_\h(\bq)|
\leq
C'_{K,p} \h^{-\frac{d}{2}} \left(\sum_{\bq\in \Lambda_\h} |U_\h(\bq)|^2\right)^{1/2}
\end{equation*}
by the Cauchy-Schwarz inequality. The last inequality in the lemma follows from \eqref{eq: relation entre la norme de U et celle de u pour Wilson}.
\end{proof}

\subsubsection*{Scalar products between states}

\begin{lemma}
\label{lemme: majoration de psi scalaire phi étoile}
Let $\h \in (0,1]$. For all $\qvec \in \Z^{2d}$ and $\svec \in E_d$, we have
\begin{equation}
\label{eq: majoration de psi scalaire phi étoile}
\begin{aligned}
|\langle \eg{\qvec}, \w{\svec}^\star\rangle| \leq C e^{-\frac{\pi}{8}(\delta(\qvec, \svec))^\frac{1}{2}}.
\end{aligned}
\end{equation}
Furthermore, for all $\qvec, \svec \in E_d$, the estimates
\begin{equation}
\label{eq: majoration de phi scalaire phi}
|\langle\w{\qvec}, \w{\svec}\rangle| \leq C e^{-\frac{\pi}{4}|\qvec - \svec|}
\end{equation}
and
\begin{equation}
\label{eq: majoration de phi étoile scalaire phi étoile}
|\langle \w{\qvec}^\star, \w{\svec}^\star\rangle| \leq C e^{-|\qvec - \svec|^\frac{1}{2}}
\end{equation}
hold true. The constants $C$ above only depend on $d$.
\end{lemma}

\begin{proof}
It is established in~\cite[Proposition A.3 and Lemma C.2]{chaumontfrelet_ingremeau_2022a}
and~\cite[Proposition 5.2]{CFDI}, that
\begin{equation*}
|\langle \eg{\rvec}, \eg{\rvec'}\rangle|
=
e^{-\frac{\pi}{4}|\qvec - \svec|},
\quad
|\langle \eg{\rvec}^\star, \eg{\rvec'}^\star\rangle|
\leq
C e^{-|\rvec - \rvec'|^\frac{1}{2}},
\quad
|\langle\eg{\rvec}^\star, \eg{\rvec'}\rangle|
\leq
C e^{-\frac{\pi}{8}|\rvec - \rvec'|^\frac{1}{2}},
\end{equation*}
for all $\rvec, \rvec' \in \Z^{2d}$.
As a result, the estimates above follow from
the expressions for $\w{\qvec}$ and $\w{\qvec}^\star$
in~\eqref{eq:DefWilsonState} and~\eqref{eq:ExpressionDualBasis}.
\end{proof}

The previous lemma will be especially useful combined with the following proposition,
which can be seen as a generalisation of \cite[Proposition 3.1]{CFDI}.

\begin{proposition}
\label{prop: approximation de g dans le frame des varphi}
Let $\Lambda,\Lambda'$ be two (possibly $h$-dependent) subsets of $\Z^{2d}$.
Let $(g_{\hbar, \qvec})_{\qvec \in \Lambda}\subset L^2(\R^d)$ be an ($\h$-dependent)
family of functions, and let $(\varphi_{\hbar, \svec})_{\svec \in \Lambda'}\subset L^2(\R^d)$
be an ($\h$-dependent) frame with dual frame $(\varphi^\star_{\hbar, \svec})_{\svec \in \Lambda'}$.
We suppose that for all $m\in \N$, there exists $c_m >0$ such that 
\begin{equation*}
\forall \qvec \in \Lambda,
\forall \svec \in \Lambda',
\quad
|\langle g_{\hbar, \qvec}, \varphi_{\hbar, \svec}^\star\rangle|
\leq
\frac{c_m}{1 + (\delta(\qvec, \svec))^m}.
\end{equation*}
and for all $p \in \N$, there exists $C_p>0$ such that 
\begin{equation*}
\forall \svec \in \Lambda',
\quad
\|\varphi_{\h,\svec}\|_{\widehat{H}_\h^p}
\leq
C_p (1 + |\svec|^p).
\end{equation*}

Then, for any $p\in \N$, there exists $C'_p>0$ such that
$\|g_{\hbar, \qvec}\|_{\widehat{H}_\hbar^p} \leq C'_p (1 + |\qvec|^p)$.
Furthermore, for any $\varepsilon > 0$, $p \in \N$, we have
\begin{equation}\label{eq:Approxg}
\forall \qvec \in \Lambda,
\quad
\left \|
g_{\hbar,\qvec}
-
\sum_{\substack{\svec \in \Lambda'\\\delta(\qvec, \svec) \leq \hbar^{-\varepsilon}}}
\langle g_{\hbar, \qvec}, \varphi_{\hbar, \svec}^\star\rangle \varphi_{\hbar, \svec}
\right \|_{\widehat{H}_\hbar^p}
=
(1 + |\qvec|^p) \times O_{\varepsilon,p}(\h^\infty)  .
\end{equation}
\end{proposition}

\begin{proof}
By the frame property, we have  $g_{\hbar, \qvec} = \sum\limits_{\svec \in \Lambda'} \langle g_{\hbar, \qvec}, \varphi_{\hbar, \svec}^\star\rangle \varphi_{\hbar, \svec}$, so that
\begin{align*}
\tau
:=
\left \|
g_{\hbar, \qvec}
-
\sum_{\substack{\svec \in \Lambda'\\\delta(\qvec, \svec) \leq \hbar^{-\varepsilon}}}
\langle g_{\hbar, \qvec}, \varphi_{\hbar, \svec}^\star\rangle \varphi_{\hbar, \svec}
\right \|_{\widehat{H}_\hbar^p}
&=
\left \|
\sum_{\substack{\svec \in \Lambda'\\\delta(\qvec, \svec) > \hbar^{-\varepsilon}}}
\langle g_{\hbar, \qvec}, \varphi_{\hbar, \svec}^\star\rangle \varphi_{\hbar, \svec}
\right \|_{\widehat{H}_\hbar^p}
\\
&\leq
\sum_{\substack{\svec \in \Lambda'\\\delta(\qvec, \svec) > \hbar^{-\varepsilon}}}
|\langle g_{\hbar, \qvec}, \varphi_{\hbar, \svec}^\star\rangle|
\|\varphi_{\hbar, \svec}\|_{\widehat{H}_\hbar^p}.
\end{align*}
Using the two assumptions, we can further write that
\begin{equation*}
\tau
\leq
C_p c_m
\sum_{\substack{\svec \in \Lambda'\\\delta(\qvec, \svec) > \hbar^{-\varepsilon}}}
\frac{1 + |\svec|^p}{1 + (\delta(\qvec, \svec))^m}
\leq
C_p c_m
\hbar^\frac{\varepsilon m}{2}
\sum_{\substack{\svec \in \Lambda'\\\delta(\qvec, \svec) > \hbar^{-\varepsilon}}}
\frac{1 + |\svec|^p}{(\delta(\qvec, \svec))^\frac{m}{2}},
\end{equation*}
for all $m \geq m_0$, where $m_0 \in \N$ will be fixed large enough below.
The estimate in~\eqref{eq:Approxg} then follows from
\begin{align*}
\tau
&\leq C_p c_m \hbar^\frac{\varepsilon m}{2}  \sum_{\substack{\svec \in \Lambda'\\\delta(\qvec, \svec) > \hbar^{-\varepsilon}}} \sum_{\epsilonvec \in \{-1, 1\}^d} \frac{1 + |\svec|^p}{|\qvec - [\mvec_\svec, \epsilonvec \nvec_\svec]|^\frac{m}{2}}\\
&\leq C_p c_m \hbar^\frac{\varepsilon m}{2}  \sum_{\epsilonvec \in \{-1, 1\}^d} \sum_{\svec' \in \Z^{2d}\setminus\{\zerovec\}} \frac{1 + |\qvec - \svec'|^p}{|\svec'|^\frac{m}{2}}\\
&\leq C'_p c_m \hbar^\frac{\varepsilon m}{2}  \sum_{\svec' \in \Z^{2d}\setminus\{\zerovec\}} \frac{1 + |\qvec|^p + |\svec'|^p}{|\svec'|^\frac{m}{2}}\\
&\leq C'_p c_m \hbar^\frac{\varepsilon m}{2}  \left(\sum_{\svec' \in \Z^{2d}\setminus\{\zerovec\}} \frac{1 + |\svec'|^p}{|\svec'|^\frac{m}{2}} + |\qvec|^p \sum_{\svec' \in \Z^{2d}\setminus\{\zerovec\}} \frac{1}{|\svec'|^\frac{m}{2}}\right)\\
&\leq C'_p c_m \hbar^\frac{\varepsilon m}{2} \left(1+ |\qvec|^p\right),
\end{align*}
provided $m_0$ is large enough so that the two sums converge.

To prove the estimate on $\|g_{\hbar, \qvec}\|_{\widehat{H}_\hbar^p}$, we simply note that, for any $m\in \N$, we have
\begin{align*}
\|g_{\hbar, \qvec}\|_{\widehat{H}_\hbar^p}
=
\left \|
\sum_{\svec \in \Lambda'}
\langle g_{\hbar, \qvec}, \varphi_{\hbar, \svec}^\star \rangle \varphi_{\hbar, \svec}
\right \|_{\widehat{H}_\hbar^p}
\leq
\sum_{\svec \in \Lambda'}
|\langle g_{\hbar, \qvec}, \varphi_{\hbar, \svec}^\star \rangle|
\|\varphi_{\hbar, \svec}\|_{\widehat{H}_\hbar^p}
\leq
C_p c_m
\sum_{\svec \in \Lambda'}
\frac{1 + |\svec|^p}{1 + (\delta(\qvec, \svec))^m},
\end{align*}
and we conclude as above.
\end{proof}

As an immediate consequence of the previous proposition and of Lemmas \ref{lemme: majoration de la norme H de phi} and \ref{lemme: majoration de psi scalaire phi étoile},  we obtain the following result.

\begin{corollary}
\label{cor: formules d'approximation 1}
Consider arbitrary $p\in  \N$ and $\varepsilon > 0$. We have
\begin{equation}
\label{eq_approx_Psi_star}
\left \|
\eg{\qvec} -
\sum_{\substack{\svec \in E_d\\\delta(\qvec, \svec) \leq \hbar^{-\varepsilon}}}
\langle\eg{\qvec}, \w{\svec}^\star\rangle  \w{\svec}
\right\|_{\widehat{H}_\hbar^p}
=
(1 + |\qvec|^p) \times O_{p,\varepsilon}(\h^\infty)
\end{equation}
for all $\qvec \in \Z^{2d}$. Besides, for $\qvec \in E_d$, the estimates
\begin{equation}
\label{eq_approx_Phi_star}
\left \|
\w{\qvec}^\star
-
\sum\limits_{\substack{\svec \in E_d\\|\qvec - \svec| \leq \hbar^{-\varepsilon}}}
\langle\w{\qvec}^\star, \w{\svec}^\star\rangle \w{\svec}
\right \|_{\widehat{H}_\hbar^p}
=
(1 + |\qvec|^p) \times O_{p,\varepsilon}(\h^\infty)
\end{equation}
and
\begin{equation}
\label{eq_norm_Phi_star}
\|\w{\qvec}^\star\|_{\widehat{H}_\hbar^p}
\leq
C_p (1 + |\qvec|^p)
\end{equation}
hold true.
\end{corollary}

\subsubsection*{Interplay of Wilson states with differential operators}

We will also need a variant of~\eqref{eq: majoration de phi scalaire phi},
where we apply a differential operator to one of the Wilson states.
The following lemma directly follows from \cite[Lemma A.3]{CFDI},
and from the definition of the Wilson states.

\begin{lemma}
\label{lemme: P w scalaire w}
Let $n\in \N$,
$A
:=
\left (
A^\h_{\alphavec}
\right )_{\alphavec \in \N^d, \hbar \in ]0, 1]}
\subset
\mathcal{C}^\infty_b(\R^d)$
be such that, for every $p \in \N$,
\begin{equation*}
|A|_p := \sup_{\hbar \in ]0, 1]} \sum_{\substack{\alphavec \in \N^d\\{[\alphavec]} \leq n}} \left\|A^\hbar_{\alphavec}\right\|_{C^p} < + \infty,
\end{equation*}
and consider the operator
\begin{equation*}
\mathcal{P}_{\h,A}
:=
\sum_{\substack{\alphavec \in \N^d\\{[\alphavec]} \leq n}}
\h^{[\alphavec]} A_{\alphavec}^\h\partial^{\alphavec}.
\end{equation*}
Then, for every $m\in \N$, there exists $C_{m}>0$
such that,
for every $\qvec, \svec \in E_d$,  we have
\begin{equation*}
|\langle \mathcal{P}_{\h,A} \w{\qvec}, \w{\svec}\rangle |
\leq
C_{m} |A|_m
\frac{1 + |\xin{\nvec_\qvec}|^{n}}{1 + |\qvec - \svec|^m}.
\end{equation*}
\end{lemma}

Using the previous Lemma along with~\eqref{eq_norm_Phi_star} in
Corollary~\ref{cor: formules d'approximation 1}
and with Proposition~\ref{prop: approximation de g dans le frame des varphi}
(for $g_{\hbar, \qvec} = \mathcal{P}_{\hbar, A} \w{\qvec}/(1 + (\pi \hbar)^d |\qvec|^{2d})$),
we deduce the following result.

\begin{corollary}
\label{cor: formules d'approximation 2}
Let $\mathcal{P}_{\h,A}$ be as in~Lemma \ref{lemme: P w scalaire w}.
Let $N,s\in \N$ and let $\varepsilon>0$. There exists $p= p(N,s, \varepsilon)\in \N$
such that for all $\qvec \in E_d$,
\begin{equation}
\left \|
\mathcal{P}_{\h,A} \w{\qvec}
-
\sum\limits_{\substack{\svec \in E_d\\|\qvec - \svec| \leq \hbar^{-\varepsilon}}}
\langle \mathcal{P}_{\h,A} \w{\qvec}, \w{\svec}\rangle \w{\svec}^\star
\right \|_{\widehat{H}_\hbar^s}
\leq C_{\varepsilon,s,N,A_p}
(1 + |\qvec|^{n + s}) \h^N.
\end{equation}
\end{corollary}

\begin{remark}
In the sequel,  many of the results will be of the form
\begin{equation}\label{eq:GeneralBound}
\forall \h \in (0,1], ~~ \|f_\h\|_{\HH_\h^s} \leq C_{s,N, \bullet} \h^N,
\end{equation}
where $f_\h$ is some function in $\HH_\h^s$, and the constant $C_{s,N, \bullet}$ is allowed to depend on some extra parameters $\bullet$.

To prove such results, we will always show that there exists $\h_0= \h_0(s,N, \bullet)$
such that $\forall \h \in (0,\h_0], ~~ \|f_\h\|_{\HH_\h^s} \leq C'_{s,N, \bullet} \h^N$.
We will leave it to the reader to readily check that, in all the situations considered below,
we have $\forall \h \in (\h_0,1], ~~ \|f_\h\|_{\HH_\h^s} \leq C''_{s,N, \bullet}$, so
that~\eqref{eq:GeneralBound} holds with $C_{s,N, \bullet} = \max \left( C_{s,N, \bullet}', \frac{C''_{s,N, \bullet}}{\h_0^N}\right)$.
\end{remark}

\section{A discrete parametrix away from the characteristic sets}
\label{section_parametrix}

In all this section, we suppose that Assumption~\ref{Hyp4:SymetrieSymbole} holds,
along with Assumptions~\ref{Hyp1:BoundCoefs}
and~\ref{Hyp:BoundedLayers}.

The aim of this section is to prove the following proposition.
Since, in the rest of the proof, we will need results that apply
both to the operator $P_\h$ and to its adjoint $P_\h^*$, we will
denote by $\cP_\h$ an operator that is either $P_\h$ or $P_\h^*$.
Correspondingly, $\pi_\h$ will either refer to $p_\h$
or to $\overline{p_\h}$. All the quantities introduced below in
this section may implicitly depend on our choice of $P_\h$ or
of $P_\h^*$.

\begin{proposition}
\label{prop: proposition-clé}
Let $\varepsilon, \varepsilon' \in \left]0, \frac{1}{2}\right[$, $D>0$, $N\in \N$.
There exists $C=C_{\varepsilon, \varepsilon',N, D}>0$ such that the
following holds for all $0< \h<1$. For any $\qvec_0 \in E_d\cap \B(\bzero, \h^{-1/2} D)$ with
$2 \hbar^{\frac{1}{2} - \varepsilon} \leq \left|p_\h(\bz^{\h, \bq_0})\right|$, there exist coefficients $U_{\hbar, \qvec_0, \svec}= U_{\hbar, \qvec_0, \svec}^{\varepsilon, \varepsilon', N} \in \C$ such that $\left|U_{\hbar, \qvec_0, \svec}\right| \leq C\hbar^{-\frac{1}{2}}$,  and the following holds.  For any $p\in \N$,  there exists a constant $C_{p, \varepsilon,  \varepsilon', N,D}>0$ such that
\begin{equation*}
\left\|\w{\qvec_0} - \sum_{\substack{\svec \in E_d\\|\svec - \qvec_0| \leq \hbar^{-\varepsilon'}}} U_{\hbar, \qvec_0, \svec}\mathcal{P}_\hbar \w{\svec}\right\|_{\widehat{H}_\hbar^p} \leq C_{p, \varepsilon,  \varepsilon', N,D}\h^N.
\end{equation*}
\end{proposition}

Note that the coefficients $U_{\hbar, \qvec_0, \svec}$ will a priori depend on $N$, so that we don't use the $O(\h^\infty)$ notation in the statement of Proposition \ref{prop: proposition-clé}.

\begin{remark}
The proof of the previous proposition does not use Assumption \ref{Hyp:PolynResolv}. However, under this assumption, Proposition \ref{prop: proposition-clé} implies that, when $\pi_\h(\bz^{\h, \bq_0})$ is away from the characteristic sets, we have a simple approximate expression (i.e., a parametrix) for $\cP_\h^{-1} \Phi_{\h, \bq_0}$, as a linear combination of $\Phi_{\h, \bs}$ with $\bs$ close to $\bq_0$. This result can be seen as an analogue of Theorem \ref{theorem_approximability} when the source term is away from the characteristic sets, and explains why, in most of the paper, we focus on source terms that are microlocalised on characteristic sets.
\end{remark}

\begin{remark}\label{Rem:GeneralisationPropClef}
In the proof of Proposition \ref{prop: proposition-clé}, we will never make use of the Riesz basis property satisfied by the Wilson states. We will only use the frame property,  along with Lemma \ref{lemme: majoration de psi scalaire phi étoile} and Corollary \ref{cor: formules d'approximation 2}. The Gaussian states $(\Psi_{\h, \bq})_{\bq \in \Z^{2d}}$ satisfy such properties, and hence all the proofs in this section would work in the same way if we replace $E_d$ with $\Z^{2d}$ and the $\Phi_{\h,\bq}$ with the  $\Psi_{\h,\bq}$.
Furthermore, in this section, we only use Assumption \ref{Hyp4:SymetrieSymbole} in Lemma \ref{lemme: (P - p) Phi}, which holds for Gaussian states without Assumption~\ref{Hyp4:SymetrieSymbole} (see \eqref{eq:OperatorCLoseToSymbolGaussians}).

 We would obtain the following result, which has interest on its own, holds without Assumption \ref{Hyp4:SymetrieSymbole},  and which we will use in Section~\ref{subsec:ProofApproxRealCase}. 
 
\smallskip

Let $\varepsilon, \varepsilon' \in \left]0, \frac{1}{2}\right[$, $D>0$ and $N\in \N$.
There exists $C=C_{\varepsilon, \varepsilon',N, D}>0$ such that the following holds
for any $0<\h < 1$. For any $\qvec_0 \in \Z^{2d}\cap \B(\bzero, \h^{-1/2} D)$ such that
$2 \hbar^{\frac{1}{2} - \varepsilon} \leq \left|p_\h(\bz^{\h, \bq_0})\right|$, there exist
coefficients $U_{\hbar, \qvec_0, \svec}= U_{\hbar, \qvec_0, \svec}^{\varepsilon, \varepsilon', N} \in \C$ such that $\left|U_{\hbar, \qvec_0, \svec}\right| \leq C\hbar^{-\frac{1}{2}}$,  and the following holds.  For any $p\in \N$,  there exists a constant $C_{p, \varepsilon,  \varepsilon', N,D}>0$ such that
\begin{equation}\label{eq:GeneralPropClef}
\left\|\Psi_{\h,\qvec_0} - \sum_{\substack{\svec \in \Z^{2d}\\|\svec - \qvec_0| \leq \hbar^{-\varepsilon'}}} U_{\hbar, \qvec_0, \svec}\mathcal{P}_\hbar \Psi_{\h,\svec}\right\|_{\widehat{H}_\hbar^p} \leq C_{p, \varepsilon,  \varepsilon', N,D}\h^N.
\end{equation}
\end{remark}

Before proving this proposition,  we need to state a few preliminary results. First of all,  \cite[Proposition 3.3]{CFDI} implies that there exists $C>0$ such that, for any $[\bm,\bn]\in \Z^{2d}$, we have
\begin{equation}\label{eq:OperatorCLoseToSymbolGaussians}
\left\|\left(\cP_\hbar - \pi_\h(\xm{\mvec}, \xin{\nvec})\right) \eg{\mvec,  \nvec}\right\|_{L^2(\R^d)}      \leq C (1 + |\xin{\nvec}|) \hbar^\frac{1}{2}.
\end{equation}
Recalling the definition of the Wilson states~\eqref{eq:DefWilsonState}
and Assumption~\ref{Hyp4:SymetrieSymbole}, we deduce the following result.

\begin{lemma}
\label{lemme: (P - p) Phi}
There exists $C>0$ such that,  for any $[\mvec, \nvec] \in E_d$, we have
\begin{equation*}
\left\|\left(\mathcal{P}_\hbar - \pi_\hbar(\xm{\mvec}, \xin{\nvec})\right) \w{\mvec, \nvec}\right\|_{L^2(\R^d)} \leq C (1 + |\xin{\nvec}|) \hbar^\frac{1}{2}.
\end{equation*}
\end{lemma}

The following lemma is a first step towards Proposition \ref{prop: proposition-clé}.

\begin{lemma}
\label{lemme: 1er lemme de la proposition clé}
Let $\varepsilon, \varepsilon_1 \in \left]0, \frac{1}{2}\right[$, $D>0$.
We may find $C>0$ such that the following holds for all $\h\in (0,1]$.
Let  $\qvec  \in E_d\cap \B(\bzero,  h^{-1/2}D)$ be such that
$\hbar^{\frac{1}{2} - \varepsilon} \leq |p_\hbar(\bz^{\h, \bq})| $.
We may then find a family of coefficients
$\left(\wit{V}_{\hbar, \qvec, \rvec,\varepsilon_1}\right)_{ \rvec \in E_d \cap \overline{\B}(\qvec, 2 \hbar^{-\varepsilon_1})}$ such that $|\wit{V}_{\hbar, \qvec, \rvec, \varepsilon_1}| \leq C_D \hbar^{\varepsilon - 2d \varepsilon_1}$,
and such that
\begin{equation*}
\w{\qvec} = \frac{1}{\pi_\hbar(\bz^{\h, \bq})} \mathcal{P}_\hbar \w{\qvec} + \sum_{\substack{\rvec \in E_d\\|\qvec - \rvec| \leq 2 \hbar^{-\varepsilon_1}}} \wit{V}_{\hbar, \qvec, \rvec, \varepsilon_1} \w{\rvec} + \gamma_{\hbar, \qvec, \varepsilon_1},
\end{equation*}
where for any $p\in \N$, we have $\|\gamma_{\hbar, \qvec, \varepsilon_1}\|_{\widehat{H}_\hbar^p}  = O_{\varepsilon_1, p,D}(\h^\infty)$.
\end{lemma}

\begin{proof}
Let $R_{\hbar, \qvec} := \left(\mathcal{P}_\hbar - \pi_\hbar(\bz^{\h, \bq})\right) \w{\qvec}$.
Thanks to Lemma \ref{lemme: (P - p) Phi}, we have
\begin{equation*}
\|R_{\hbar, \qvec}\|_{L^2} \leq C_D \hbar^\frac{1}{2}.
\end{equation*}
Applying Corollary~\ref{cor: formules d'approximation 2}, we obtain
\begin{equation*}
R_{\hbar, \qvec} = \sum_{\substack{\svec \in E_d\\|\qvec - \svec| \leq \hbar^{-\varepsilon_1}}} \langle R_{\hbar, \qvec}, \w{\svec}\rangle \w{\svec}^\star + \rho_{\hbar, \qvec, \varepsilon_1},
\end{equation*}
where, for every $p\in \N$, we have
\begin{equation*}
\|\rho_{\hbar, \qvec, \varepsilon_1}\|_{\widehat{H}_\hbar^p}
=
(1 + |\qvec|^{2d + p}) \times O_{\varepsilon_1, p,D}(\h^\infty)
=
O_{\varepsilon_1, p,D}(\h^\infty).
\end{equation*}

Now, recall that, by Corollary \ref{cor: formules d'approximation 1}, we may write for any $\svec \in E_d$,
    \begin{equation*}
        \w{\svec}^\star = \sum_{\substack{\rvec \in E_d\\|\svec - \rvec| \leq \hbar^{-\varepsilon_1}}} \langle \w{\svec}^\star, \w{\rvec}^\star\rangle \w{\rvec} + \eta_{\hbar, \svec, \varepsilon_1}
    \end{equation*}
    where, for any $p\in \N$, we have $\|\eta_{\hbar, \svec, \varepsilon_1}\|_{\widehat{H}_\hbar^p(\R^d)} = (1 + |\svec|^p) \times O_{\varepsilon_1, p}(\h^\infty)$. We thus have
    \begin{align*}
        R_{\hbar, \qvec} &= \sum_{\substack{\svec \in E_d\\|\qvec - \svec| \leq \hbar^{-\varepsilon_1}}} \sum_{\substack{\rvec \in E_d\\|\svec - \rvec| \leq \hbar^{-\varepsilon_1}}} \langle R_{\hbar, \qvec}, \w{\svec}\rangle  \langle \w{\svec}^\star, \w{\rvec}^\star\rangle  \w{\rvec} + \sum_{\substack{\svec \in E_d\\|\qvec - \svec| \leq \hbar^{-\varepsilon_1}}} \langle R_{\hbar, \qvec}, \w{\svec}\rangle  \eta_{\hbar, \svec, \varepsilon_1} + \rho_{\hbar, \qvec, \varepsilon_1}\\
        &= \sum_{\substack{\rvec \in E_d\\|\qvec - \rvec| \leq 2\hbar^{-\varepsilon_1}}} V_{\hbar, \qvec, \rvec,\varepsilon_1} \w{\rvec} + \Tilde{\rho}_{\hbar, \qvec, \varepsilon_1}
    \end{align*}
with
    \begin{equation*}
        V_{\hbar, \qvec, \rvec, \varepsilon_1} := \sum\limits_{\substack{\svec \in E_d\\|\qvec - \svec| \leq \hbar^{-\varepsilon_1}\\|\svec - \rvec| \leq \hbar^{-\varepsilon_1}}} \langle R_{\hbar, \qvec}, \w{\svec}\rangle  \langle \w{\svec}^\star, \w{\rvec}^\star\rangle,
    \end{equation*}
which satisfies $|V_{\hbar, \qvec, \rvec, \varepsilon_1}| \leq C_D \hbar^{\frac{1}{2} - 2d\varepsilon_1}$, and 
    \begin{equation*}
        \Tilde{\rho}_{\hbar, \qvec, \varepsilon_1} := \sum\limits_{\substack{\svec \in E_d\\|\qvec - \svec| \leq \hbar^{-\varepsilon_1}}} \langle R_{\hbar, \qvec}, \w{\svec}\rangle  \eta_{\hbar, \svec, \varepsilon_1} + \rho_{\hbar, \qvec, \varepsilon_1}.
    \end{equation*}

In the sum above,  we always have $|\bs|= O_D(\h^{-1/2})$ and the sum contains a number of terms that is bounded polynomially in $\h^{-1}$.  We deduce that $\|\Tilde{\rho}_{\hbar, \qvec, \varepsilon_1}\|_{\widehat{H}_\hbar^p}= O_{\varepsilon_1, p,D}(\h^\infty)$ for any $p\in \N$.

    To conclude,  we note that, by definition of $R_{\hbar, \qvec}$,  we have
    
    \begin{equation*}
        \w{\qvec} = \frac{1}{\pi_\hbar(\bz^{\h, \bq})} \mathcal{P}_\hbar \w{\qvec} - \frac{1}{\pi_\hbar(\bz^{\h, \bq})} R_{\hbar, \qvec},
    \end{equation*}
so that
    \begin{equation*}
        \w{\qvec} = \frac{1}{\pi_\hbar(\bz^{\h, \bq})} \mathcal{P}_\hbar \w{\qvec} + \sum_{\substack{\rvec \in E_d\\|\qvec - \rvec| \leq 2\hbar^{-\varepsilon_1}}} \wit{V}_{\hbar, \qvec, \rvec, \varepsilon_1} \w{\rvec} + \gamma_{\hbar, \qvec, \varepsilon_1}
    \end{equation*}
with
$
\wit{V}_{\hbar, \qvec, \rvec, \varepsilon_1}
:=
-\frac{1}{\pi_\hbar(\bz^{\h, \bq})} V_{\hbar, \qvec, \rvec, \varepsilon_1}
$
and
$
\gamma_{\hbar, \qvec, \varepsilon_1}
:=
-\frac{1}{\pi_\hbar(\bz^{\h, \bq})} \Tilde{\rho}_{\hbar, \qvec, \varepsilon_1}
$.
Since
$|\pi_\hbar(\bz^{\h, \bq})| = |p_\hbar(\bz^{\h, \bq})| \geq \hbar^{\frac{1}{2} - \varepsilon}$,
we have
$|\wit{V}_{\hbar, \qvec, \rvec, \varepsilon_1}| \leq C_D \hbar^{\varepsilon - 2d \varepsilon_1}$
and, for any $p\in \N$,
$
\|\gamma_{\hbar, \qvec, \varepsilon_1}\|_{\widehat{H}_\hbar^p}
=
O_{\varepsilon_1, p,D}(\h^\infty)$.
\end{proof}

We will want to apply Lemma \ref{lemme: 1er lemme de la proposition clé} iteratively, and to this end, we will consider sequences of points of $E_d$.
For any $\qvec_0 \in E_d$, $L \in \N^*$ and $\tau > 0$,  we thus set
\begin{equation*}
F^L_{\hbar,\tau}(\bq_0)
:=
\left \{
(\qvec_1, \dots, \qvec_L) \in (E_d)^L
\ | \
\forall \ell \in\{1, ..., L  \}, \ 
|\qvec_{\ell-1} - \qvec_\ell| \leq 2 \hbar^{-\tau}
\right \}.
\end{equation*}

We have
$
F^L_{\hbar,\tau}(\bq_0)
\subset
\bigg(\overline{\B}(\qvec_0, 2 L \hbar^{-\tau}) \cap \Z^{2d}\bigg)^L
$
and hence
$\card\left(F^L_{\hbar,\tau}(\bq_0) \right) \leq C_L \hbar^{-2d \tau L}$.

\begin{lemma}
    \label{lemme: les qj ne bavent pas trop}
   Let $0< \varepsilon_1 < \varepsilon < \frac{1}{2}$, $D>0$, and let $L\in \N^*$.  There exists $\h_0 = \h_0(\varepsilon, \varepsilon_1, D, L)>0$ such that, for all $\h<\h_0$, the following holds.  If 
    $\qvec_0\in E_d\cap \B(\bzero,  D \h^{-1/2})$ is such that $2 \hbar^{\frac{1}{2} - \varepsilon} \leq \left|p_\hbar(\bz^{\h,\bq_0})\right| $,
    then for all $(\qvec_1, \dots, \qvec_L) \in F^L_{\hbar, \varepsilon_1}(\bq_0)$, for all $\ell\in \{1,..., L\}$, we have $|\bq_\ell|< 2D \h^{-\frac{1}{2}}$ and 
        \begin{equation*}
    \hbar^{\frac{1}{2} - \varepsilon} \leq \left|p_\hbar(\bz^{\h,\bq_\ell})\right|.
    \end{equation*}
\end{lemma}

\begin{proof} 
All the $\bq_\ell$ belong to $\overline{\B}(\qvec_0, 2 L \hbar^{-\varepsilon_1})\subset \overline{\B}(\bzero, D \h^{-1/2} + 2L \h^{-\varepsilon_1})$,  which is included in $\B(\bzero,  2 D \h^{-1/2})$ provided $\h$ is small enough.  Then, all the $\bz^{\h, \bq_\ell}$ belong to $\B(\bzero, 2D \sqrt{\pi})$, and we may use the fact that $p_\h$ is Lipschitz on this ball to see that $\left| \left|p_\hbar(\bz^{\h,\bq_\ell})\right| -  \left|p_\hbar(\bz^{\h,\bq_0})\right|\right| \leq C_D L \h^{\frac{1}{2}- \varepsilon_1}$, from which we may conclude.
\end{proof}

We may now state an iterated version of Lemma \ref{lemme: 1er lemme de la proposition clé}.
In the following statement, the coefficients $\wit{V}_{\h, \bq, \br, \varepsilon_1}$ will be
as in Lemma \ref{lemme: 1er lemme de la proposition clé}.

\begin{lemma}
\label{lemme: récurrence pour la proposition-clé}
Let $0< \varepsilon_1 < \varepsilon < \frac{1}{2}$,  $D>0$ and let $L\in \N$.
There exists $\h_0 = \h_0(\varepsilon, \varepsilon_1, L, D)$ such that the following
holds for all $\h\in (0, \h_0)$. Let $\qvec_0 \in E_d\cap \B(\bzero,  D \h^{-1/2})$
be such that $2 \hbar^{\frac{1}{2} - \varepsilon} \leq |p_\hbar(\bz^{\h, \bq_0})|$.
We then have
\begin{equation}
\label{eq: récurrence pour la proposition-clé}
\begin{split}
\w{\qvec_0}
=
\frac{1}{\pi_\hbar(\bz^{\h, \bq_0})} \mathcal{P}_\hbar \w{\qvec_0}
&+
\sum_{\ell = 1}^L
\sum_{(\qvec_1, \dots, \qvec_\ell) \in  F^\ell_{\hbar, \varepsilon_1}(\bq_0)}
\frac{1}{\pi_\hbar(\bz^{\h, \bq_\ell})}
\left[\prod_{i = 1}^\ell \wit{V}_{\hbar, \qvec_{i - 1}, \qvec_i \varepsilon_1}\right]
\mathcal{P}_\hbar \w{\qvec_j}
\\
&+
\sum_{(\qvec_1, \dots, \qvec_{L + 1}) \in   F^{L+1}_{\hbar, \varepsilon_1}(\bq_0) }
\left[\prod_{i = 1}^{L + 1} \wit{V}_{\hbar, \qvec_{i - 1}, \qvec_i \varepsilon_1}\right]
\w{\qvec_{L + 1}} + \rho_{L, \hbar, \qvec_0, \varepsilon_1},
\end{split}
\end{equation}
where, for any $p\in \N$,  $\|\rho_{L, \hbar, \qvec_0, \varepsilon_1}\|_{\widehat{H}_\hbar^p} =O_{\varepsilon_1, p, L,D}(\h^\infty)$.
\end{lemma}

\begin{proof} 
The proof goes by induction. For $L=0$, the result is just
Lemma~\ref{lemme: 1er lemme de la proposition clé}.

Suppose that~\eqref{eq: récurrence pour la proposition-clé} holds at rank $L\in \N$.
Thanks to Lemma~\ref{lemme: les qj ne bavent pas trop}, up to taking $\h_0$ smaller,
we know that all the $\bq_{L+1}$ in the last sum are in $\B(\bzero, 2D \h^{-1/2})$ and satisfy 
\begin{equation*}
\hbar^{\frac{1}{2} - \varepsilon} \leq \left|p_\hbar(\bz^{\h, \bq_{L+1}}) \right|.
\end{equation*}
We may thus apply Lemma \ref{lemme: 1er lemme de la proposition clé} to each of the terms in
the last sum of~\eqref{eq: récurrence pour la proposition-clé} to write 
\begin{equation}
\label{eq: application du 1er lemme de la proposition-clé dans l'hérédité}
\w{\qvec_{L + 1}}
=
\frac{1}{\pi_\hbar(\bz^{\h, \bq_{L+1}})} \mathcal{P}_\hbar \w{\qvec_{L + 1}}
+
\sum_{\substack{\qvec_{L + 2} \in E_d\\|\qvec_{L + 1} - \qvec_{L + 2}|
\leq
2 \hbar^{-\varepsilon_1}}}
\wit{V}_{\hbar, \qvec_{L + 1}, \qvec_{L + 2}, \varepsilon_1} \w{\qvec_{L + 2}}
+
\gamma_{\hbar, \qvec_{L + 1}, \varepsilon_1},
\end{equation}
where for every $p\in \N$,
$
\|\gamma_{\hbar, \qvec_{L + 1}, \varepsilon_1}\|_{\widehat{H}_\hbar^p}
=
O_{\varepsilon_1, p,D}(\h^\infty)
$.
    
We may then inject \eqref{eq: application du 1er lemme de la proposition-clé dans l'hérédité} in the last sum of \eqref{eq: récurrence pour la proposition-clé} to obtain
    \begin{equation*}
        \begin{split}
            \w{\qvec_0} = \frac{1}{\pi_\hbar(\bz^{\h, \bq_0})} \mathcal{P_\hbar} \w{\qvec_0} &+ \sum_{\ell = 1}^{L + 1} \sum_{(\qvec_1, \dots, \qvec_\ell) \in F^\ell_{\hbar, \varepsilon_1}(\bq_0) } \frac{1}{\pi_\hbar(\bz^{\h, \bq_\ell})} \left[\prod_{i = 1}^\ell \wit{V}_{\hbar, \qvec_{i - 1}, \qvec_i, \varepsilon_1}\right] \mathcal{P}_\hbar \w{\qvec_j}\\
            &+ \sum_{(\qvec_1, \dots, \qvec_{L + 2}) \in F^{L+2}_{\hbar, \varepsilon_1}(\bq_0) } \left[\prod_{\ell = 1}^{L + 2} \wit{V}_{\hbar, \qvec_{l - 1}, \qvec_l, \varepsilon_1}\right] \w{\qvec_{L + 2}} + \rho_{L + 1, \hbar, \qvec_0, \varepsilon_1},
        \end{split}
    \end{equation*}
with
  
    \begin{equation*}
        \rho_{L + 1, \hbar, \qvec_0, \varepsilon_1} := \sum_{(\qvec_1, \dots, \qvec_{L + 1}) \in F^{L+1}_{\hbar, \varepsilon_1}(\bq_0) } \left[\prod_{\ell = 1}^{L + 1} \wit{V}_{\hbar, \qvec_{\ell - 1}, \qvec_\ell, \varepsilon_1}\right] \gamma_{\hbar, \qvec_{L + 1}, \varepsilon_1} + \rho_{L, \hbar, \qvec_0, \varepsilon_1}.
    \end{equation*}
    
Since the number of terms in the sum and the coefficients multiplying the $\gamma_{\hbar, \qvec_{L + 1}, \varepsilon_1}$ are all bounded by some negative power of $\h$, we deduce that for any $p\in \N$,    $\|\rho_{L + 1, \hbar, \qvec_0, \varepsilon_1}\|_{\widehat{H}_\hbar^p} =O_{\varepsilon_1, p, L,D}(\h^\infty)$,  which proves the result at rank $L+1$.
\end{proof}

We may now proceed with the proof of Proposition \ref{prop: proposition-clé}.

\begin{proof}[Proof of Proposition \ref{prop: proposition-clé}]
\textbf{Step 1: Applying Lemma  \ref{lemme: récurrence pour la proposition-clé}.}
Let $N\in \N$.  We set $\varepsilon_1 := \frac{1}{2} \min\left(\frac{\varepsilon^2}{2dN + 4d \varepsilon}, \varepsilon'\right)$, and $L := \left\lfloor \frac{\varepsilon}{2d \varepsilon_1} \right\rfloor$.  Note that $\varepsilon_1 < \frac{\varepsilon^2}{2dN + 4d \varepsilon} \leq \frac{\varepsilon}{4d}$, so that $L\geq 2$.  Furthermore,  since $\varepsilon_1 \leq \frac{\varepsilon^2}{2dN + 4d \varepsilon}$, we have $\frac{\varepsilon}{2d \varepsilon_1} - \frac{N}{\varepsilon - 4d \varepsilon_1} = \frac{\varepsilon^2- \varepsilon_1(2dN + 4d\varepsilon)}{2d\varepsilon_1(\varepsilon-4d\varepsilon_1)} \geq 0$, and thus
    \begin{equation*}
        \frac{N}{\varepsilon - 4d \varepsilon_1} - 1 \leq L \leq \frac{\varepsilon}{2d \varepsilon_1}.
    \end{equation*}
    
We may then apply Lemma \ref{lemme: récurrence pour la proposition-clé} to obtain
\begin{multline}
\label{eq: résultat de la récurrence utilisé dans la preuve de la proposition-clé}
\w{\qvec_0}
=
\frac{1}{\pi_\hbar(\bz^{\h, \bq_0})} \mathcal{P}_\hbar \w{\qvec_0}
\\
+
\sum_{\ell = 1}^L
\sum_{(\qvec_1, \dots, \qvec_\ell) \in F^{\ell}_{\hbar, \varepsilon_1}(\bq_0)}
\frac{1}{\pi_\hbar(\bz^{\h, \bq_\ell})}
\left[\prod_{i = 1}^\ell \wit{V}_{\hbar, \qvec_{i - 1}, \qvec_i, \varepsilon_1}\right]
\mathcal{P}_\hbar \w{\qvec_\ell}
+
\eta_{\hbar, \qvec_0, \varepsilon, \varepsilon', N}
\end{multline}
with
    \begin{equation*}
        \eta_{\hbar, \qvec_0, \varepsilon, \varepsilon', N} := \sum_{(\qvec_1, \dots, \qvec_{L + 1}) \in F^{L+1}_{\hbar, \varepsilon_1}(\bq_0)} \left[\prod_{\ell = 1}^{L + 1} \wit{V}_{\hbar, \qvec_{\ell - 1}, \qvec_\ell, \varepsilon_1}\right] \w{\qvec_{L + 1}} + \rho_{L, \hbar, \qvec_0, \varepsilon_1},
    \end{equation*}
    where, for any $p\in \N$,  $\| \rho_{L, \hbar, \qvec_0, \varepsilon_1}\|_{\widehat{H}_\hbar^p}  = O_{\varepsilon, \varepsilon', p, N,D} (\h^\infty)$.

    \textbf{Step 2: Bounding $\|\eta_{\hbar, \qvec_0, \varepsilon, \varepsilon', N}\|_{\widehat{H}_\hbar^p}$.}
    We have
    \begin{align*}
        \left\|\sum_{(\qvec_1, \dots, \qvec_{L + 1}) \in F^{L+1}_{\hbar, \varepsilon_1}(\bq_0)} \left[\prod_{\ell = 1}^{L + 1} \wit{V}_{\hbar, \qvec_{\ell - 1}, \qvec_\ell, \varepsilon_1}\right] \w{\qvec_{L + 1}}\right\|_{\widehat{H}_\hbar^p} 
        &\leq \sum_{(\qvec_1, \dots, \qvec_{L+ 1}) \in F^{L+1}_{\hbar, \varepsilon_1}(\bq_0)} \left[\prod_{\ell = 1}^{L + 1} \left|\wit{V}_{\hbar, \qvec_{\ell - 1}, \qvec_\ell, \varepsilon_1}\right|\right] \left\|\w{\qvec_{L + 1}}\right\|_{\widehat{H}_\hbar^p}\\
    &\leq C_{p,D} \sum_{(\qvec_1, \dots, \qvec_{L+ 1}) \in F^{L+1}_{\hbar, \varepsilon_1}(\bq_0)} \left[\prod_{\ell = 1}^{L + 1} \left|\wit{V}_{\hbar, \qvec_{\ell - 1}, \qvec_\ell, \varepsilon_1}\right|\right]\\
        &\leq C_{p, N,D, \varepsilon, \varepsilon'} \hbar^{(\varepsilon - 4d \varepsilon_1)(L+ 1)}.
    \end{align*}
    In the second inequality, we used Lemma \ref{lemme: majoration de la norme H de phi} along with the fact that all the $\bz^{\h,\bq_{L+1}}$ corresponding to the $\bq_{L+1}$ in the sum belong to a bounded set, depending only on $D$. In the last inequality, we used the fact that the $\left|\wit{V}_{\hbar, \qvec_{\ell - 1}, \qvec_\ell, \varepsilon_1}\right|$ are all $O_D(\hbar^{\varepsilon - 2d \varepsilon_1})$, while the number of terms in the sum is $ \leq C_L \hbar^{-2d \varepsilon_1 (L + 1)}$.
    
    Now, recalling that $(\varepsilon - 4d \varepsilon_1)(L+ 1)\geq N$, we deduce that, for any $p\in \N$, we may find $C_{\varepsilon, \varepsilon', p, N,D}>0$ such that $\|\eta_{\hbar, \qvec_0, \varepsilon, \varepsilon', N}\|_{\widehat{H}_\hbar^p} \leq C_{\varepsilon, \varepsilon', p, N,D} \hbar^N$.
    
    \textbf{Step 3: Studying the double sum in (\ref{eq: résultat de la récurrence utilisé dans la preuve de la proposition-clé}).}
    To rewrite the  double sum in (\ref{eq: résultat de la récurrence utilisé dans la preuve de la proposition-clé}), we start by noting that all the $\bq_\ell$ appearing there satisfy $|\qvec_\ell - \qvec_0| \leq 2\ell \hbar^{-\varepsilon_1} \leq 2L \hbar^{-\varepsilon_1} \leq \hbar^{-\varepsilon'}$ for $\hbar$ small enough (since $\varepsilon_1 < \varepsilon'$).  The formula  (\ref{eq: résultat de la récurrence utilisé dans la preuve de la proposition-clé}) may thus be rewritten, for $\h$ smaller than some $\h_0(\varepsilon, \varepsilon', N)$, as
    \begin{equation}
        \label{eq: formule de la proposition-clé dans la preuve}
        \w{\qvec_0} = \sum_{\substack{\svec \in E_d\\|\svec - \qvec_0| \leq \hbar^{-\varepsilon'}}} U_{\hbar, \qvec_0, \svec}^{\varepsilon, \varepsilon', N} \mathcal{P}_\hbar \w{\svec} + \eta_{\hbar, \qvec_0, \varepsilon, \varepsilon', N},
    \end{equation}
    for some family of coefficients $U_{\hbar, \qvec_0, \svec}^{\varepsilon, \varepsilon', N}$ which we now need to bound.
    
    Let $\svec \in E_d$ be such that $|\svec - \qvec_0| \leq \hbar^{-\varepsilon'}$.  In the double sum of \eqref{eq: résultat de la récurrence utilisé dans la preuve de la proposition-clé},  the number of terms involving $\mathcal{P}_\hbar \w{\svec}$ is smaller than the total number of terms, which is 
    \begin{equation*}
        \sum\limits_{\ell = 1}^L \card\left( F^{\ell}_{\hbar, \varepsilon_1}(\bq_0) \right) \leq \sum\limits_{\ell = 1}^L C_\ell \hbar^{-2d \varepsilon_1 \ell} \leq C_L \hbar^{-2d \varepsilon_1 L}.
    \end{equation*}

Furthermore,  in the double sum in \eqref{eq: résultat de la récurrence utilisé dans la preuve de la proposition-clé}, each term involving $\mathcal{P}_\hbar \w{\svec}$ comes with a coefficient of the form $\frac{1}{\pi_\hbar(\bz^{\h, \bq_\ell})} \prod\limits_{i = 1}^\ell \wit{V}_{\hbar, \qvec_{i - 1}, \qvec_li \varepsilon_1}$, which may be bounded as
    \begin{equation*}
        \left|\frac{1}{\pi_\hbar(\bz^{\h, \bq_\ell})} \prod_{i = 1}^l \wit{V}_{\hbar, \qvec_{i - 1}, \qvec_i, \varepsilon_1}\right| \leq \hbar^{-\frac{1}{2} + \varepsilon} \times C^\ell_D \hbar^{\ell(\varepsilon - 2d \varepsilon_1)} \leq C_{L, D} \hbar^{-\frac{1}{2} + \varepsilon}.
    \end{equation*}
Here, we used the fact that,  by Lemma \ref{lemme: les qj ne bavent pas trop},  for every $1\leq \ell \leq L$,  we have $\left|\pi(\bz^{\h, \bq_\ell})\right| = \left|p_\hbar(\bz^{\h, \bq_\ell})\right| \geq \hbar^{\frac{1}{2} - \varepsilon}$.

All in all, we obtain that,  if $\svec \neq \qvec_0$, then $\left|U_{\hbar, \qvec_0, \svec}^{\varepsilon, \varepsilon', N}\right| \leq C_{L, D} \hbar^{-\frac{1}{2} + \varepsilon - 2d \varepsilon_1 L} \leq C_{L, D} \hbar^{-\frac{1}{2}}$ (since $\varepsilon - 2d \varepsilon_1 L \geq 0$).  Since $L=L(\varepsilon, \varepsilon', N)$, we get $\left|U_{\hbar, \qvec_0, \svec}^{\varepsilon, \varepsilon', N}\right| \leq C_{\varepsilon, \varepsilon', N, D} \hbar^{-\frac{1}{2}}$, as announced.
When $\svec = \qvec_0$, we must also take into account the term $\frac{1}{\pi_\hbar(\bz^{\h, \bq_0})} \mathcal{P}_\hbar \w{\qvec_0}$,  and we obtain the same result since $\left|\frac{1}{\pi_\hbar(\bz^{\h, \bq_0})}\right| \leq  \frac{1}{2} \hbar^{-\frac{1}{2} + \varepsilon}$. 
\end{proof}

\section{Proof of the approximation theorem}
\label{section_approximation}

The aim of this section is to prove Theorem~\ref{theorem_approximability}.
To this end, we fix $0 < \varepsilon < 1/2$,  $0 \leq \alpha < \delta_0/2$
and $\lambda'> \lambda>0$. All the quantities introduced in this section
 may implicitly depend on $\varepsilon$, $\alpha$, $\lambda$,
$\lambda'$ and on the integer $J$ introduced below, but we will not always
write this dependence explicitly. In this section, we will always suppose that
$\h$ is small enough so that $\Lambda_\h(\lambda')$ is finite (which is possible
by Assumption \ref{Hyp:BoundedLayers}).

\subsection{Introducing intermediate energy layers}
\label{subsec:EnergyLayers}
In the proof of Theorems~\ref{theorem_approximability}
and~\ref{theorem_accuracy},
we will need to introduce several intermediate energy layers.
We thus take an integer $J\in \N^*$, and define a sequence $(\lambda_j)_{j=0,..., J}$ as
\begin{equation}
\lambda_j = \lambda + \frac{j}{J} (\lambda'-\lambda),
\end{equation}
so that $\lambda_0 = \lambda$ and $\lambda_J=\lambda'$. In the proof of
Theorem~\ref{theorem_approximability}, we will take $J = 3$, while in
the proof of Theorem~\ref{theorem_accuracy},
we will need $J = 6$. 

We then introduce, for $1\leq j \leq J$,  the sets
\begin{equation*}
\Lambda_{\h, j} :=  \left\{[\mvec, \nvec] \in E_d \;\Big|\; \alpha + \lambda_{j-1}  \hbar^{\frac{1}{2} - \varepsilon} \leq \min_{\epsilonvec \in \{-1, 1\}^d}\left|p_\hbar(\xm{\mvec}, \epsilonvec \xin{\nvec})\right| < \alpha + \lambda_j \hbar^{\frac{1}{2} - \varepsilon}\right\},
\end{equation*}
and write $\Lambda_{\h,0} := \Lambda_\h(\lambda)$, defined as in (\ref{eq:EnergyLayers}), so that 
\begin{equation*}
\Lambda_\h(\lambda') =  \bigsqcup_{j=0}^J \Lambda_{\h,j}.
\end{equation*}

If $0\leq  j_1\leq j_2 \leq J$,  we will also write
\begin{equation*}
\Lambda_{\h, j_1, j_2} := \bigsqcup_{j=j_1}^{j_2} \Lambda_{\h,j},
\end{equation*}
and $W_{\h, j_1, j_2} := \cD (\ell^2(\Lambda_{\h, \lambda_1, \lambda_2}))$ with the operator $\cD$ as in Section~\ref{section_gabor_wilson}. Note that, if $w\in W_{\h, j_1, j_2}$, the uniqueness of the decomposition \eqref{eq_wilson_decomposition} implies that we have
\begin{equation}\label{eq:UniquenessDecomposition}
w= \sum_{\bq\in \Lambda_{\h, j_1, j_2}} \langle w, \w{\qvec}^\star\rangle  \w{\qvec}.
\end{equation}

Finally, we introduce analogous sets, but in $\Z^{2d}$, to work with Gabor frames rather than Riesz bases. We define
\begin{equation}\label{eq:DefLambdaTilde}
\begin{aligned}
\wit{\Lambda}_{\h,0} &:=  \left\{[\mvec, \nvec] \in \Z^{2d} \;\Big|\; \left|p_\hbar(\xm{\mvec},   \xin{\nvec})\right| < \alpha + \lambda \hbar^{\frac{1}{2} - \varepsilon}\right\}\\
\wit{\Lambda}_{\h, j} &:=  \left\{[\mvec, \nvec] \in \Z^{2d} \;\Big|\; \alpha + \lambda_{j-1}  \hbar^{\frac{1}{2} - \varepsilon} \leq \left|p_\hbar(\xm{\mvec},   \xin{\nvec})\right| < \alpha + \lambda_j \hbar^{\frac{1}{2} - \varepsilon}\right\} ~~~~\text{ for } 1\leq j \leq J,\\
\wit{\Lambda}_{\h, j_1, j_2} &:= \bigsqcup_{j=j_1}^{j_2} \wit{\Lambda}_{\h, j}.
\end{aligned}
\end{equation}

We will often use the following lemma, which tells us that non-consecutive energy layers are well-separated from  each other.

\begin{lemma}
\label{lemme: épaisseur des niveaux d'énergie}
There exist $C, \h_0>0$ such that, for all $j\in \{0,\ldots, J-1\}$,
$\h\in (0,\h_0]$, $\qvec \in \wit{\Lambda}_{\h, j} \cup \Lambda_{\h,j}$
and $\svec \in E_d \setminus \Lambda_{\h,0,j+1}$, we have
\begin{equation}\label{eq:QSLoin}
\delta(\qvec, \svec) > C \h^{-\varepsilon}.
\end{equation}
In particular, for any $\varepsilon'\in (0, \varepsilon)$,
there exists $\h_1(\varepsilon')$ such that, for all $\h \in (0,\h_1(\varepsilon')]$,
$j\in \{0,..., J-1\}$, and $\qvec \in \wit{\Lambda}_{\h, j} \cup \Lambda_{\h,j}$,
we have
\begin{equation}
\label{eq: les s proches de q sont dans lambda 1,j+1}
\left \{
\svec \in E_d, \delta(\qvec, \svec) \leq \hbar^{-\varepsilon'}
\right \}
\subset
\Lambda_{\h,0,j+1},
\end{equation}
and, if $\bq\in \Lambda_{\h, j}$ with $j\in \{1,..., J-1\}$,  we have
\begin{equation}
\label{eq:s proches de q dans lambda j-1,j+1}
\left \{
\svec \in E_d, |\qvec - \svec| \leq \h^{-\varepsilon'}
\right \}
\subset
\Lambda_{\h,j-1,j+1}.
\end{equation}
\end{lemma}

\begin{proof}
First of all, we note that, thanks to Assumption~\ref{Hyp:BoundedLayers},
we have $\bz^{\h,\qvec}\in \B(\bzero, D_0)$ for all $\h$ smaller than some $\h_0(\varepsilon, \lambda')$. When $\bz^{\h,\bs}\notin \B(\bzero,2D_0)$,
we have, for each $\epsilonvec \in \{-1, 1\}^d$,
\begin{equation*}
|\bq - [\bm_{\bs}, \epsilonvec \bn_{\bs}]| = \frac{1}{\sqrt{\pi \h}}|\bz^{\h, \bq} - [\xm{\bm_{\bs}}, \epsilonvec \xin{\bn_{\bs}}]| > \frac{D_0}{\sqrt{\pi \h}}
\end{equation*}
so that \eqref{eq:QSLoin} holds. We may thus suppose that $\bz^{\h,\bs}\in \B(\bzero, 2D_0)$.

 Now, thanks to Assumption \ref{Hyp1:BoundCoefs} and to (\ref{eq_symbol}),  $|\nabla p_\h|$ is bounded (independently of $\h$) on $\B(\bzero, 2D_0)$, and thus, $|p_h|$ is Lipschitz on $\B(\bzero, 2D_0)$.  We deduce that there exists $C>0$ such that, for any $\epsilonvec, \epsilonvec'\in \{-1, 1\}^d$
 we have
\begin{equation}\label{eq:LipschitzP}
 \left| p_\hbar(\xm{\mvec_\svec}, \epsilonvec \xin{\nvec_\svec})- p_\hbar(\xm{\mvec_\qvec},  \epsilonvec' \xin{\nvec_\qvec})\right| \leq C \left|  (\xm{\mvec_\svec}, \epsilonvec \xin{\nvec_\svec}) - (\xm{\mvec_\qvec},  \epsilonvec' \xin{\nvec_\qvec}) \right|. 
 \end{equation} 
First, if $\bq \in \wit{\Lambda}_{\h,j}$, we take $\epsilonvec'=1$, then $\epsilonvec = \epsilonvec_0$ realising the minimum in the right-hand side of (\ref{eq:LipschitzP}),  to obtain 
\begin{align*}
J^{-1} (\lambda'-\lambda)h^{\frac{1}{2}-\varepsilon} &< \min_{\epsilonvec \in \{-1, 1\}^d} \left| p_\hbar(\xm{\mvec_\svec}, \epsilonvec \xin{\nvec_\svec})\right| -   \left|p_\hbar(\xm{\mvec_\qvec}, \xin{\nvec_\qvec}) \right| \\
&\leq  \left| p_\hbar(\xm{\mvec_\svec}, \epsilonvec_0 \xin{\nvec_\svec})\right| -   \left|p_\hbar(\xm{\mvec_\qvec}, \xin{\nvec_\qvec}) \right|\\
&\leq  C    \left|(\xm{\mvec_\qvec}, \xin{\nvec_\qvec}) - (\xm{\mvec_\svec}, \epsilonvec_0 \xin{\nvec_\svec})\right|\\
&= C  \min_{\epsilonvec \in \{-1, 1\}^d}  \left|(\xm{\mvec_\qvec}, \xin{\nvec_\qvec}) - (\xm{\mvec_\svec}, \epsilonvec \xin{\nvec_\svec})\right|\\
&= C \sqrt{\pi\h}  \min_{\epsilonvec \in \{-1, 1\}^d}  \left|[\mvec_\qvec- \mvec_{\svec} , \nvec_\qvec -\epsilonvec\nvec_{\svec} ]\right|= C \sqrt{\pi\h}   \delta(\qvec, \svec),
\end{align*}
and we deduce (\ref{eq:QSLoin}).  When $\bq\in \Lambda_{\h,j}$,  we take $\epsilonvec'$ which minimises  $\left|p_\hbar(\xm{\mvec_\qvec},  \epsilonvec' \xin{\nvec_\qvec}) \right|$,  then $\epsilonvec = \epsilonvec_0$ realising the minimum in the right-hand side of (\ref{eq:LipschitzP}). We may then argue in the same way, starting from $J^{-1} (\lambda'-\lambda)h^{\frac{1}{2}-\varepsilon} < \min\limits_{\epsilonvec \in \{-1, 1\}^d} \left| p_\hbar(\xm{\mvec_\svec}, \epsilonvec \xin{\nvec_\svec})\right| -   \left|p_\hbar(\xm{\mvec_\qvec},  \epsilonvec' \xin{\nvec_\qvec}) \right|$.

Equation~\eqref{eq: les s proches de q sont dans lambda 1,j+1} directly follows from the contrapositive of \eqref{eq:QSLoin}.

Finally,~\eqref{eq:s proches de q dans lambda j-1,j+1} is proved using~\eqref{eq: les s proches de q sont dans lambda 1,j+1} twice: if $\svec \in E_d$ with $|\qvec - \svec| \leq \hbar^{-\varepsilon'}$, then by \eqref{eq: les s proches de q sont dans lambda 1,j+1}, we must have $\bs\in \Lambda_{\h, 0, j+1}$. But, if $\bs \in \Lambda_{\h, j'}$ for some $j' < j-1$, then using \eqref{eq: les s proches de q sont dans lambda 1,j+1} again, we would have $\bq\in \Lambda_{\h, 0, j'+1}$, which is a contradiction.
\end{proof}

The following lemma reflects the fact that the operator $P_\h$ is pseudodifferential,
so that it ``does not move things in phase space''.

\begin{lemma}
    \label{lemme: si v est micro-localisé alors Pv aussi}
 Let $0\leq j \leq J-2$, and let $w_\hbar \in W_{\h, 0,j}$.  There exists $\wzero \in  W_{\h, 0,j+2}$ such that, for all $p\in \N$, we have
    \begin{equation*}
        \left\|P_\hbar w_\hbar - \wzero \right\|_{\widehat{H}_\hbar^p} = O_{p} (\h^\infty)\times  \|w_\hbar\|_{L^2}.
    \end{equation*}
\end{lemma}

\begin{proof}
Recall that by Corollary \ref{cor: formules d'approximation 2}, for any $\bq\in \Lambda_{\h, 0, j}$, we may write  
\begin{align*}
        P_\hbar \w{\qvec} = \sum_{\substack{\svec \in E_d\\|\qvec - \svec| \leq \hbar^{-\frac{\varepsilon}{2}}}} \langle P_\hbar \w{\qvec}, \w{\svec}\rangle  \w{\svec}^\star + R_{\hbar, \qvec}
    \end{align*}
    with $\|R_{\hbar, \qvec}\|_{\widehat{H}_\hbar^p} = O_{p} (\h^\infty)$ for any $p\in\N$. We thus have
    \begin{align*}
        P_\hbar w_\hbar &= \sum_{\qvec \in \Lambda_{\h,0, j}} \langle w_\hbar, \w{\qvec}^\star\rangle  P_\hbar \w{\qvec}\\
        &= \sum_{\qvec \in \Lambda_{\h,0, j}} \sum_{\substack{\svec \in E_d\\|\qvec - \svec| \leq \hbar^{-\frac{\varepsilon}{2}}}} \langle w_\hbar, \w{\qvec}^\star\rangle  \langle P_\hbar \w{\qvec}, \w{\svec}\rangle  \w{\svec}^\star + \sum_{\qvec \in \Lambda_{\h,0, j}} \langle w_\hbar, \w{\qvec}^\star\rangle  R_{\hbar, \qvec}.
    \end{align*}
    
    Let us denote by $R_{\hbar}(w_\h)$ the second term. Using Lemma \ref{lem:LePtiLemKiSauveToutLTemps} and \eqref{eq:BornePhiStar}, we have $\|R_{\hbar}(w_\h)\|_{\widehat{H}_\hbar^p} = O_{p} (\h^\infty) \times \|w_\h\|_{L^2}$ for any $p\in\N$. Thanks to \eqref{eq: les s proches de q sont dans lambda 1,j+1}, we have for $\h$ small enough and for all $\qvec \in  \Lambda_{\h, 0, j}$, $\left\{\svec \in E_d, |\qvec - \svec| \leq \hbar^{-\frac{\varepsilon}{2}}\right\} = \left\{\svec \in \Lambda_{\h, 0, j+1}, |\qvec - \svec| \leq \hbar^{-\frac{\varepsilon}{2}}\right\}$. We deduce that
    \begin{align*}
        P_\hbar w_\hbar &= \sum_{\qvec \in \Lambda_{\h, 0, j}} \sum_{\substack{\svec \in \Lambda_{\h, 0, j+1}\\|\qvec - \svec| \leq \hbar^{-\frac{\varepsilon}{2}}}} \langle w_\hbar, \w{\qvec}^\star\rangle  \langle P_\hbar \w{\qvec}, \w{\svec}\rangle  \w{\svec}^\star + R_{\hbar}(w_\h)\\
        &= \sum_{\svec \in \Lambda_{\h,0,j+1}} \gamma_{\hbar,  \svec} \w{\svec}^\star + R_{\hbar}(w_\h),
    \end{align*}
with $\gamma_{\hbar, \svec} := \sum\limits_{\substack{\qvec \in \Lambda_{\h,0,j}\\|\qvec - \svec| \leq \hbar^{-\frac{\varepsilon}{2}}}} \langle w_\hbar, \w{\qvec}^\star\rangle  \langle P_\hbar \w{\qvec}, \w{\svec}\rangle $. 

Next, we recall from Corollary \ref{cor: formules d'approximation 1},  that we have, for any $\bs\in \Lambda_{\h,0,j+1}$
    \begin{align*}
        \w{\svec}^\star = \sum_{\substack{\rvec \in E_d\\|\svec - \rvec| \leq \hbar^{-\frac{\varepsilon}{2}}}} \langle \w{\svec}^\star, \w{\rvec}^\star\rangle  \w{\rvec} + \rho_{\hbar, \svec},
    \end{align*}
with $\|\rho_{\hbar, \svec}\|_{\widehat{H}_\hbar^p} =O_{p} (\h^\infty)$ for any $p\in\N$.  We therefore obtain, using \eqref{eq: les s proches de q sont dans lambda 1,j+1} again,
    \begin{align*}
        P_\hbar w_\hbar &= \sum_{\svec \in \Lambda_{\h,0,j+1}} \sum_{\substack{\rvec \in \Lambda_{\h,0,j+2}\\|\svec - \rvec| \leq \hbar^{-\frac{\varepsilon}{2}}}} \gamma_{\hbar, \svec} \langle \w{\svec}^\star, \w{\rvec}^\star\rangle \w{\rvec} + \sum_{\svec \in \Lambda_{\h,0,j+1}} \gamma_{\hbar,  \svec} \rho_{\hbar, \varepsilon, \svec} + R_{\hbar}(w_\h)\\
        &= \sum_{\rvec \in \Lambda_{\h,0,j+2}} \gamma_{\hbar, \rvec}' \w{\rvec} + \sum_{\svec \in \Lambda_{\h,0,j+1}} \gamma_{\hbar, \svec} \rho_{\hbar, \svec} + R_{\hbar}(w_\h),
    \end{align*}
with $\gamma_{\hbar, \rvec}' := \sum\limits_{\substack{\svec \in \Lambda_{\h,0,j+1}\\|\svec - \rvec| \leq \hbar^{-\frac{\varepsilon}{2}}}} \gamma_{\hbar, \svec} \langle \w{\svec}^\star, \w{\rvec}^\star \rangle $. 

 We then set $\wzero := \sum\limits_{\rvec \in\Lambda_{\h,0,j+2}} \gamma_{\hbar, \rvec}' \w{\rvec} \in W_{\hbar, 0, j+2}$. 
We thus have, for any $p\in \N$,
    \begin{align*}
        \|P_\hbar w_\hbar - \wzero\|_{\widehat{H}_\hbar^p} \leq \sum_{\svec \in \Lambda_{\h,0, j+1}} |\gamma_{\hbar, \svec}| \times O_{ p} (\h^\infty) + O_{ p} (\h^\infty) \times \|w_\h\|_{L^2}.
    \end{align*}
    Now,  the sum above contains $O(\h^{-d})$ terms, and each of the $|\gamma_{\hbar, \svec}|$ is bounded as $O(\h^{-d})\times \|w_\h\|_{L^2}$, thanks to Lemma  \ref{lemme: P w scalaire w}. We thus have $ \|P_\hbar w_\hbar - \wzero\|_{\widehat{H}_\hbar^p} = O_{p} (\h^\infty) \times \|w_\h\|_{L^2}$ for every $p\in \N$, as announced.
\end{proof}

Recall that the reconstruction operator $\wit{\cD}_\h$ was introduced in \eqref{eq:DefOperateurReconstructionGaussian}.  The proof of Theorem \ref{theorem_approximability} will heavily rely on\footnote{Strictly speaking,  \cite[Theorem 2.4]{CFDI} only treats the case where $\lambda_1=2$ and $\lambda_2=4$. However,  the proof would work in exactly the same way for more general $\lambda_1<\lambda_2$. Alternatively,  we may take $\varepsilon'= \varepsilon/2$, $\alpha'= \alpha+ \lambda_1 \h^{1/2-\varepsilon} -2 \h^{1/2-\varepsilon'}$, and apply the result of \cite[Theorem 2.4]{CFDI} with $\varepsilon'$ and $\alpha'$ instead of $\varepsilon$ and $\alpha$, to deduce \eqref{eq: approximation de u par des psi}.} \cite[Theorem 2.4]{CFDI},  where the authors prove that, if $F'_\h\in \ell^2(\wit{\Lambda}_{\h, 0,1})$ and if $u_\h$ is the solution of $P_\h u_\h = \wit{\cD}_\h F'_\h$, then we have for any $p\in \N$
    \begin{equation}
        \label{eq: approximation de u par des psi}
       \left\|u_\hbar - \sum\limits_{\qvec \in \wit{\Lambda}_{\h, 0,2}} \langle u_\hbar, \eg{\qvec}^\star\rangle  \eg{\qvec}\right\|_{\widehat{H}_\hbar^p} = O_{p} (\h^\infty)\times  \|F'_\hbar\|_{\ell^2\left(\wit{\Lambda}_{\h, 0,1}\right)}.
    \end{equation}
    
Since~\eqref{eq: approximation de u par des psi} holds when
$f_\h = \wit{\cD}_\h(F_\h') \in \wit{\cD}_\h (\ell^2(\wit{\Lambda}_{\h, 0,1}) )$,
we will start by proving Theorem~\ref{theorem_approximability} in this case,
even though it is not exactly the assumption we made on $f_\h$ in the theorem.

\subsection{Proof of Theorem \ref{theorem_approximability} when $f_\h \in \wit{\cD}_\h (\ell^2(\wit{\Lambda}_{\h, 0,1}) )$}\label{subsec:ApproxInGoodSpace}

Let us take $f_\h = \wit{\cD}_\h(F_\h')$ with $F_\h'\in \ell^2(\wit{\Lambda}_{\h, 0,1})$.
The proof of the statement will go in two steps: we will first show that $u_\h$ may be well approximated in $\cD_\h (\ell^2(\Lambda_{\h, 0,3}))$
, and we will then identify the coefficients of the approximation in $\cD_\h (\ell^2(\Lambda_{\h}(\lambda')))$. 

\textbf{Step 1: Approaching $u_\h$ in $\cD_\h (\ell^2(\Lambda_{\h, 0,3}))$.}
    By Corollary \ref{cor: formules d'approximation 1},  we have, for any $\bq\in \wit{\Lambda}_{\h,0,2}$,
    \begin{align*}
        \eg{\qvec} = \sum_{\substack{\svec \in E_d\\\delta(\qvec, \svec) \leq \hbar^{-\frac{\varepsilon}{2}}}} \langle\eg{\qvec}, \w{\svec}^\star\rangle  \w{\svec} + \rho_{\hbar, \qvec}
    \end{align*}
where $\|\rho_{\hbar, \qvec}\|_{\widehat{H}_\hbar^p} = O_{p} (\h^\infty)$ for any $p\in \N$.  We may thus rewrite 
    \begin{align*}
        \sum_{\qvec \in \wit{\Lambda}_{\h,0,2}} \langle u_\hbar, \eg{\qvec}^\star\rangle  \eg{\qvec} &= \sum_{\qvec \in \wit{\Lambda}_{\h,0,2}} \langle u_\hbar, \eg{\qvec}^\star\rangle  \sum_{\substack{\svec \in E_d\\\delta(\qvec, \svec) \leq \hbar^{-\frac{\varepsilon}{2}}}} \langle \eg{\qvec}, \w{\svec}^\star\rangle  \w{\svec} + \sum_{\qvec \in \wit{\Lambda}_{\h,0,2}} \langle u_\hbar, \eg{\qvec}^\star\rangle  \rho_{\hbar, \qvec}\\
        &= \sum_{\qvec \in \wit{\Lambda}_{\h,0,2}} \sum_{\substack{\svec \in \Lambda_{\h,0,3}\\\delta(\qvec, \svec) \leq \hbar^{-\frac{\varepsilon}{2}}}} \langle u_\hbar, \eg{\qvec}^\star\rangle  \langle \eg{\qvec}, \w{\svec}^\star\rangle  \w{\svec} + \sum_{\qvec \in \wit{\Lambda}_{\h,0,2}} \langle u_\hbar, \eg{\qvec}^\star\rangle  \rho_{\hbar, \qvec} \quad \mbox{(by \eqref{eq: les s proches de q sont dans lambda 1,j+1})}\\
        &= \sum_{\svec \in \Lambda_{\h,0,3}} \alpha_{\hbar, \svec} \w{\svec} + \sum_{\qvec \in \wit{\Lambda}_{\h,0,2}} \langle u_\hbar, \eg{\qvec}^\star\rangle \rho_{\hbar, \qvec}
    \end{align*}
with  $\alpha_{\hbar, \svec} := \sum\limits_{\substack{\qvec \in \wit{\Lambda}_{\h,0,2}\\\delta(\qvec, \svec) \leq \hbar^{-\frac{\varepsilon}{2}}}} \langle u_\hbar, \eg{\qvec}^\star\rangle  \langle \eg{\qvec}, \w{\svec}^\star\rangle $.
    
We therefore have, for any $p\in \N$,
    \begin{align*}
        \left\|u_\hbar - \sum_{\svec \in \Lambda_{\h,0,3}} \alpha_{\hbar, \svec} \w{\svec}\right\|_{\widehat{H}_\hbar^p} &\leq \left\|u_\hbar - \sum_{\qvec \in\wit{\Lambda}_{\h,0,2}} \langle u_\hbar, \eg{\qvec}^\star\rangle  \eg{\qvec}\right\|_{\widehat{H}_\hbar^p} + \left\|\sum_{\qvec \in \wit{\Lambda}_{\h,0,2}} \langle u_\hbar, \eg{\qvec}^\star\rangle  \rho_{\hbar, \qvec}\right\|_{\widehat{H}_\hbar^p}\\
        &\leq O_{p} (\h^\infty)\times  \|F'_\hbar\|_{\ell^2\left(\wit{\Lambda}_{\h, 0,1}\right)} +\sum_{\qvec \in \wit{\Lambda}_{\h,0,2}} \left|\langle u_\hbar, \eg{\qvec}^\star\rangle  \right| \times O_{p} (\h^\infty)\\
      &=  O_{p} (\h^\infty)\times  \left( \|F'_\hbar\|_{\ell^2\left(\wit{\Lambda}_{\h, 0,1}\right)}+ \|u_\h\|_{L^2} \right),
    \end{align*}
    thanks to \eqref{eq: approximation de u par des psi}, the Cauchy-Schwarz inequality and Lemma~\ref{lem:LePtiLemKiSauveToutLTemps}.
    Now, by Assumption~\ref{Hyp:PolynResolv} and~\eqref{eq: relation entre la norme de U et celle de u pour Gabor}, we have $\|u_\hbar\|_{L^2} \leq C \hbar^{-N_0} \|f_\hbar\|_{L^2(\R^d)} \leq C \hbar^{-N_0}\|F'_\hbar\|_{\ell^2\left(\wit{\Lambda}_{\h, 0,1}\right)}$, so that we have
    \begin{align}
        \label{eq: approximation de la solution avec des alpha}
         \left\|u_\hbar - \sum_{\svec \in \Lambda_{\h,0,3}} \alpha_{\hbar, \svec} \w{\svec}\right\|_{\widehat{H}_\hbar^p} =  O_{p} (\h^\infty)\times  \|F'_\hbar\|_{\ell^2\left(\wit{\Lambda}_{\h,0,1}\right)}.
    \end{align}
    
    \textbf{Step 2: Identifying the coefficients of the approximation in $\cD_\h (\ell^2(\Lambda_{\h}(\lambda')))$.}
    Now,  remember that $u_\hbar = \sum\limits_{\svec \in \Lambda_{\h, 0,3}} \langle u_\hbar, \w{\svec}^\star \rangle  \w{\svec} + \sum\limits_{\svec \in E_d \setminus\Lambda_{\h, 0,3}} \langle u_\hbar, \w{\svec}^\star \rangle  \w{\svec}$.  We shall write $\alpha'_{\hbar, \svec} := \langle u_\hbar, \w{\svec}^\star\rangle  - \alpha_{\hbar, \svec}$ for $\svec \in \Lambda_{h,0,3}$ and  $\alpha'_{\hbar, \svec} := \langle u_\hbar, \w{\svec}^\star\rangle $ for all $\svec \in E_d \setminus \Lambda_{h,0,3}$, so that
        \begin{align}
        \label{eq: inégalité triangulaire pour remplacer les alpha par les u scalaire phi étoile}
        \left\|u_\hbar - \sum_{\svec \in \Lambda_{h}(\lambda')} \langle u_\hbar, \w{\svec}^\star\rangle  \w{\svec}\right\|_{\widehat{H}_\hbar^p} \leq \left\|u_\hbar - \sum_{\svec \in \Lambda_{h,0,3}} \alpha_{\hbar, \svec} \w{\svec}\right\|_{\widehat{H}_\hbar^p} + \left\|\sum_{\svec \in \Lambda_{h}(\lambda')} \alpha'_{\hbar, \svec} \w{\svec}\right\|_{\widehat{H}_\hbar^p}.
    \end{align}
    
We already know by \eqref{eq: approximation de la solution avec des alpha} that the first term is $O_{p} (\h^\infty)\times  \|F'_\hbar\|_{\ell^2\left(\wit{\Lambda}_{\h,0,1}\right)}$. To bound the second term, we may use Lemma \ref{lem:LePtiLemKiSauveToutLTemps} to get
\begin{align*}
\left \|
\sum_{\svec \in  \Lambda_{h}(\lambda')} \alpha_{\hbar, \svec}' \w{\svec}
\right \|_{\widehat{H}_\hbar^p}
&\leq
C_p \h^{-\frac{d}{2}} \left \| \alpha'_{\h } \right\|_{\ell^2(\Lambda_\h(\lambda'))}\\
&\leq
C_p \h^{-\frac{d}{2}} \left\| \alpha'_{\h }\right\|_{\ell^2(E_d)}
\\
&\leq
C'_p \h^{-\frac{d}{2}}
\left \|
\sum_{\svec \in E_d} \alpha_{\hbar,  \svec}' \w{\svec}
\right \|_{L^2}~~\text{by \eqref{eq: relation entre la norme de U et celle de u pour Wilson}} \\
&=
C_p \h^{-\frac{d}{2}}
\left \|
u_\hbar - \sum_{\svec \in \Lambda_{\h,0,3}} \alpha_{\hbar,  \svec} \w{\svec}
\right \|_{L^2}
\\
&=
O_{p} (\h^\infty) \times \|F'_\hbar\|_{\ell^2\left(\wit{\Lambda}_{\h,0,1}\right)}
\end{align*}
using~\eqref{eq: approximation de la solution avec des alpha} again. We therefore obtain 
\begin{equation*}
\left\|u_\hbar - \sum_{\svec \in  \Lambda_{h}(\lambda')} \langle u_\hbar, \w{\svec}^\star\rangle  \w{\svec}\right\|_{\widehat{H}_\hbar^p} = O_{ p} (\h^\infty)\times  \|F'_\hbar\|_{\ell^2\left(\wit{\Lambda}_{\h,0,1}\right)}
\end{equation*}
whenever $F'_\h \in \ell^2\left(\wit{\Lambda}_{\h,0,1}\right)$.
    
\subsection{Proof of Theorem \ref{theorem_approximability} when $f_\h \in \cD_\h (\ell^2(\Lambda_{\h}(\lambda)) )$}\label{subsec:ProofApproxRealCase}
    
In the  actual statement of Theorem \ref{theorem_approximability}, we have $f_\h= \cD_\h(F_\h)$ with $F_\h\in \ell^2(\Lambda_\h(\lambda)) =\ell^2(\Lambda_{\h,0})$, and we will now consider this situation.
     
\textbf{Step 1: Decomposing $u_\h$.}
Thanks to (\ref{eq:DefWilsonState}), we have
\begin{equation}\label{eq:DecompoPhiEnPsi}
f_\h
=
\sum_{\bq \in \Lambda_{\h,0}}
F_{\h,\bq} \Phi_{\h, \bq}
=
\sum_{\bq \in \Lambda_{\h,0}}
F_{\h,\bq}  c_{\nvec_{\bq}}
\sum_{\epsilonvec \in \{-1, 1\}^d}
\epsilonvec^{\mvec_{\bq} + \nvec_{\bq}} \eg{\mvec_{\bq}, \epsilonvec \nvec_{\bq}}
=:
\sum_{\bs \in  \Z^{2d}}
F'_{\h,\bs} \Psi_{\h, \bs}.
\end{equation}
Note that, in the sum above,  the non-vanishing terms only correspond to indices $\bs\in \Z^{2d}$ such that there exists $\bq\in \Lambda_{\h,0}$ with $\delta(\bq, \bs) = 0$. In particular, we have $\|F'_{\h}\|_{\ell^2} \leq C \|F_{\h}\|_{\ell^2}\leq C' \|f_{\h}\|_{L^2}$, and for $\h$ small enough, all the corresponding $\bz^{\h, \bs}$ belong to a bounded set.
     
We then define
\begin{equation*}
f_{\h,0} := \sum\limits_{\svec \in \wit{\Lambda}_{\h,0,1}} F'_{\h,\bs} \Psi_{\h, \bs},
\qquad
f_{\h,1} := \sum\limits_{\svec \in\Z^{2d}\setminus  \wit{\Lambda}_{\h,0,1}} F'_{\h,\bs} \Psi_{\h, \bs},
\end{equation*}
so that $f_{\h} = f_{\h,0} + f_{\h,1}$. 
Similarly, we decompose $u_\h= P_\h^{-1} f_\h$ as
\begin{equation*}
u_\h = P_\h^{-1} f_{\h,0} +  P_\h^{-1} f_{\h,1} =: u_{\h,0} + u_{\h,1}.
\end{equation*}

\textbf{Step 2: Approaching $u_{\h,0}$.}
We may apply the result of Section~\ref{subsec:ApproxInGoodSpace} to $u_{\h,0}$, to obtain
    \begin{equation}\label{eq:ApproxUh0}
        \left\|u_{\hbar,0} - \sum_{\svec \in \Lambda_{h,0,3}} \langle u_{\hbar,0}, \w{\svec}^\star\rangle  \w{\svec}\right\|_{\widehat{H}_\hbar^p} = O_{p} (\h^\infty)\times  \|F'_\hbar\|_{\ell^2\left(\wit{\Lambda}_{\h,0,1}\right)}.
    \end{equation}
    
\textbf{Step 3: Dealing with $u_{\h,1}$ when Assumption~\ref{Hyp4:SymetrieSymbole} holds.}
Now, we note that, when Assumption~\ref{Hyp4:SymetrieSymbole} holds (i.e. the symbol is nearly symmetric),
then for all $\h$ small enough, we have $f_{\h,1}= 0$.
Indeed, if $\bq\in \Lambda_{\h,0}$ and if $\bs \in \Z^{2d}$ is such that $\delta(\bq, \bs) =0$, 
then there exists $\epsilonvec, \epsilonvec'\in \{-1, 1\}^d$ 
such that $[\bm_{\bq}, \epsilonvec \bn_{\bq}]= [\bm_{\bs},  \bn_{\bs}]$, and such that $|p_\h(\xm{\bm_\qvec}, \epsilonvec' \xin{\bn_\qvec})| < \alpha + \lambda_0 \h^{\frac{1}{2}-\varepsilon}$. Thanks to  Assumption \ref{Hyp4:SymetrieSymbole}, we obtain
\begin{align*}
\left|p_\hbar(\xm{\mvec_\svec},  \xin{\nvec_\svec})\right| &\leq \left|  p_\hbar(\xm{\mvec_\qvec}, \epsilonvec'\xin{\nvec_\qvec})\right| 
+ \left|p_\hbar(\xm{\mvec_\qvec},  \epsilonvec'\xin{\nvec_\qvec}) - p_\hbar(\xm{\mvec_\qvec},  \epsilonvec \xin{\nvec_\qvec})  \right| \\
& \leq  \alpha+ \lambda_0 \h^{\frac{1}{2}-\varepsilon} + C' \h < \alpha+ \lambda_1 \h^{\frac{1}{2}-\varepsilon},
\end{align*}
for $\h$ small enough,  so $\bs \in \wit{\Lambda}_{\h,0,1}$ for $\h$ small enough.

We therefore have $f_{\h,1} = 0$, so that $u_\h = u_{\h,0}$, and the result follows from (\ref{eq:ApproxUh0}).

\textbf{Step 3': Approaching $u_{\h,1}$ when Assumption~\ref{Hyp4:SymetrieSymbole} does not hold.}
When Assumption~\ref{Hyp4:SymetrieSymbole} does not hold,
the term $f_{\h,1}$, and thus $u_{\h,1}$ is a priori not trivial.
The proof of Theorem~\ref{theorem_approximability} is then a bit
more tedious, and we will be more sketchy here.
    
First,  we note that, for $\h$ small enough, all the indices $\bs$ in $f_{\h,1}$ belong to a ball $\B(\bzero, D \h^{-1/2})\subset \R^{2d}$, with $D = \frac{D_0}{\sqrt{\pi}}$ which does not depend on $\h$.
We may use Remark~\ref{Rem:GeneralisationPropClef} to see that, for any $N\in \N$, there exist coefficients $U_{\h,\bs,\bs'} = U^{N,\varepsilon}_{\h,\bs,\bs'}$ such that for any $\svec \in \left( \Z^{2d} \setminus  \wit{\Lambda}_{\h,0,1}\right) \cap \B(\bzero, D \h^{-1/2})$ and any $p\in \N$, we have
    \begin{equation*}
        \left\|\Psi_{\h,\bs} - \sum_{\substack{\svec' \in \Z^{2d}\\|\svec' - \svec| \leq \hbar^{-\varepsilon/2}}} U_{\hbar, \svec, \svec'}P_\hbar \Psi_{\h,\svec'}\right\|_{\widehat{H}_\hbar^p} \leq C_{p,  N}\h^N
    \end{equation*}
    for some $C_{p, N}>0$.
Composing this equation with $P_\h^{-1}$ and using Assumption~\ref{Hyp:PolynResolv},
we obtain that 
 \begin{equation*}
        \left\|P_\h^{-1} \Psi_{\h,\bs} - \sum_{\substack{\svec' \in \Z^{2d}\\|\svec' - \svec| \leq \hbar^{-\varepsilon/2}}} U_{\hbar, \svec, \svec'} \Psi_{\h,\svec'}\right\|_{\widehat{H}_\hbar^p} \leq C'_{p, N}\h^{N-N_0}.
    \end{equation*}
    
    Furthermore,  by Corollary \ref{cor: formules d'approximation 1}, we see that this may be rewritten as
     \begin{equation*}
        \left\|P_\h^{-1} \Psi_{\h,\bs} - \sum_{\substack{\svec'' \in E_d\\ \delta( \svec'',  \svec) \leq 2 \hbar^{-\varepsilon/2}}} U'_{\hbar, \svec, \svec''} \Phi_{\h,\svec''}\right\|_{\widehat{H}_\hbar^p} \leq C'_{p,  N}\h^{N-N_0-\frac{p}{2}}.
    \end{equation*}
    for some choice of coefficients $U'_{\hbar, \svec, \svec''}$.
   
    We therefore get
    \begin{align*}
    &  \left\| u_{\h,1}- \sum_{\svec'' \in E_d}  \sum\limits_{\substack{\svec \in\Z^{2d}\setminus  \wit{\Lambda}_{\h,0,1}\\ \delta(\svec'', \svec) \leq 2\hbar^{-\varepsilon/2}}}  F'_{\h,\bs}   U'_{\hbar, \svec, \svec''} \Phi_{\h,\svec''}\right\|_{\widehat{H}_\hbar^p}\\    
    &=  \left\| \sum\limits_{\svec \in\Z^{2d}\setminus  \wit{\Lambda}_{\h,0,1}} F'_{\h,\bs} \left( P_\h^{-1} \Psi_{\h, \bs} - \sum_{\substack{\svec'' \in E_d\\ \delta (\svec'',  \svec) \leq 2 \hbar^{-\varepsilon/2}}} U'_{\hbar, \svec, \svec''} \Phi_{\h,\svec''}\right)  \right\|_{\widehat{H}_\hbar^p}\\
    &\leq C'_{p, N}\h^{N-N_0-\frac{p}{2}} \sum\limits_{\svec \in\Z^{2d}\setminus  \wit{\Lambda}_{\h,0,1}} |F'_{\h,\bs}|.
    \end{align*}
Since the number of non-zero terms in the sum above is $O(\h^{-d})$ by Lemma \ref{lem:LePtiLemKiSauveToutLTemps} and the discussion after \eqref{eq:DecompoPhiEnPsi},  we get
    \begin{equation}\label{eq:approxU1}
       \left\| u_{\h,1}- \sum_{\svec'' \in E_d}  \sum\limits_{\substack{\svec \in\Z^{2d}\setminus  \wit{\Lambda}_{\h,0,1}\\ \delta (\svec'',  \svec) \leq 2 \hbar^{-\varepsilon/2}}}   F'_{\h,\bs}  U_{\hbar, \svec, \svec''} \Phi_{\h,\svec''}\right\|_{\widehat{H}_\hbar^p}
       \leq C_{p, N}\h^{N-N_0-\frac{p}{2}-d} \|f_\h\|_{L^2}.
       \end{equation}
       
       Now,  by the discussion after (\ref{eq:DecompoPhiEnPsi}),  for every index $\bs''$ in the sum such that $F'_{\h, \bs} U_{\hbar, \svec, \svec''}\neq 0$, there exists $\bq\in \Lambda_{h,0}$ such that $\delta(\bq, \bs'') \leq 2 \h^{-\varepsilon/2}$.   By (\ref{eq: les s proches de q sont dans lambda 1,j+1}),  this implies that, provided $\h$ is small enough, we always have $\bs''\in \Lambda_{\h, 0,1}$.  Combining (\ref{eq:approxU1}) with (\ref{eq:ApproxUh0}), we have thus shown that, for any $p, N\in \N$, there exist coefficients $\alpha_{\h, N}(\bs)$ such that
           \begin{equation*}
         \left\|u_\hbar - \sum_{\svec \in \Lambda_{\h,0,3}} \alpha_{\hbar, N}(\svec) \w{\svec}\right\|_{\widehat{H}_\hbar^p} \leq C_{ p,N} \h^N \times  \|f_\h\|_{L^2}.
    \end{equation*}
    
    Arguing as in Step 2 in section \ref{subsec:ApproxInGoodSpace}, we obtain, for any $p, N\in \N$
\begin{equation*}   
      \left\|u_\hbar - \sum_{\svec \in \Lambda_{h}(\lambda')} \langle u_\hbar, \w{\svec}^\star\rangle  \w{\svec}\right\|_{\widehat{H}_\hbar^p} \leq C_{p,N} \h^{N-\frac{d}{2}} \times  \|f_\h\|_{L^2}.
      \end{equation*}
       Since this is true for any $N\in \N$, the result follows.

\subsection{Right-hand sides corresponding to incident fields: proof of Corollary \ref{corollary_uinc}}

So far, we have dealt with cases where $f_\h$ is a sum of coherent
states, which is not a practical assumption. As we show in Theorem \ref{thm:IncidentField}
below, however, any right-hand side corresponding to an incident field can be well-approximated
by such a sum. This result can be seen as an extension of \cite[Lemma 4.4]{CFDI}.
The main tool in the proof of Theorem \ref{thm:IncidentField} is the following lemma.
Recall that the set $\wit{\Lambda}_{\h,0}$ (which depends on $\lambda$, $\varepsilon$ and $\alpha$)
has been defined in~\eqref{eq:DefLambdaTilde}.

\begin{lemma}
\label{lemma_dot_uinc}
We consider a bounded open set $\cO\subset \R^d$ such that
\begin{equation*}
\delta \eq 
\dist(\cO,\Sigma^{\rm c})
>
0,
\quad
\text{with}
\quad
\Sigma
\eq
\{
\bx \in \R^d \; | \;
\forall \h \in \LH, \forall \bxi \in \R^d, ~~
p_\h(\bx,\bxi) = |\bxi|^2-1
\}
\end{equation*}
and a cutoff $\chi \in C^\infty_{\rm c}(\cO)$. Then, for all
incident fields $u_\h^{\rm inc} \in H^1(\cO)$ with
$\h^2\Delta u_\h^{\rm inc}+u_\h^{\rm inc} = 0$ in $\cO$,
and for all $\varepsilon,\lambda > 0$, we have
\begin{equation*}
|\langle f_\h,\eg{\bq} \rangle|
=
O(\h^\infty) \|u_\h^{\rm inc}\|_{H^1_\h(\cO)}
\qquad
\forall \bq \in \Z^{2d} \setminus  \wit{\Lambda}_{\h,0}
\end{equation*}
for $f_\h \eq \chi u_\h^{inc}$, where the implied constant in the $O(\h^\infty)$
notation only depends on $\lambda$, $\varepsilon$, the norms $\|\chi\|_{C^\ell}$
for $\ell \in \N$, $\delta$, and the smallest $R$ such that $|\bx| \leq R$
for all $\bx \in \cO$. In particular, the implied constant does not depend on $\bq$.
\end{lemma}

\begin{proof}
Let us first note that since $f_\h$ is compactly supported in $\cO$, we can
use elliptic regularity for $-\Delta$ together with the product rule
to show that
\begin{equation}
\label{tmp_elliptic_regularity}
\|f_\h\|_{\HH^m_\h}
\leq
R^m \|f_\h\|_{H^m_\h(\cO)}
\leq
C_m R^m \|\chi\|_{C^m} \|u_\h^{\rm inc}\|_{H^1_\h(\cO)}
\leq
C_{m,\chi} \|u_\h^{\rm inc}\|_{H^1_\h(\cO)}.
\end{equation}
for all $m \in \N$.

We then distinguish two cases, depending on whether $\bx^{\h,\bm_{\bq}} \in \Sigma$ or not.

{\bf Case 1.}
We first focus on indices $\bq$ that are far away
from the support of $f_\h$. More specifically, we consider
indices $\bq \in \Z^{2d} \setminus \wit{\Lambda}_{\h,0}$
with $\bx^{\h,\bm_{\bq}} \in \Sigma^{\rm c}$. We then know that
$\dist(\bx^{\h,\bm_{\bq}},\operatorname{supp} f_\h) \geq \delta$.
It follows that
\begin{equation*}
|\langle f_\h,\eg{\bq} \rangle|
\leq
C \h^{-\frac{d}{4}} \|f_\h\|_{L^1} e^{-\frac{\delta^2}{2\h}}
\leq
C \h^{-\frac{d}{4}} R^{d/2} \|f_\h\|_{L^2} e^{-\frac{\delta^2}{2\h}},
\end{equation*}
which concludes the proof for this case.

{\bf Case 2.}
In the second case, we consider $\bq \in \Z^{2d} \setminus \wit{\Lambda}_{\h,0}$
with $\bx^{\h,\bm_{\bq}} \in \Sigma$. For later references, we note that by definition we have
\begin{equation}
\label{tmp_lb_lambda}
||\bxi^{\h,\bn_{\bq}}|^2-1| \geq \lambda \h^{1/2-\varepsilon}.
\end{equation}

Next, we observe that~\cite[Proposition A.6]{CFDI} implies that
\begin{equation}
\label{tmp_hQ_Psi}
\|(\h^2\Delta+|\bxi^{\h,\bn_{\bq}}|^2)^\ell \eg{\bq}\|_{L^2}
\leq
C_\ell (1+|\bxi^{\h,\bn_{\bq}}|^2)^{\ell} \h^{\ell/2}
\end{equation}
for all $\ell \in \N$.
To see this, we apply the proposition to the operator $P_\h := -\h^2\Delta$,
and note that then we have $p_\h(\bz^{\h,\bq}) = |\bxi^{\h,\bn_{\bq}}|^2$ for the associated
symbol, so that $\h^2\Delta+|\bxi^{\h,\bn_{\bq}}|^2 = p_\h(\bz^{\h,\bq})-P_\h$.
For shortness, we introduce the operator
\begin{equation*}
Q \eq \frac{1}{|\bxi^{\h,\bn_{\bq}}|^2-1} (\h^2\Delta+|\bxi^{\h,\bn_{\bq}}|^2)
\end{equation*}
which, as a direct consequence of~\eqref{tmp_lb_lambda} and~\eqref{tmp_hQ_Psi}, satisfies
\begin{equation}
\label{tmp_Q_Psi}
\|Q^{\ell} \eg{\bq}\|_{L^2}
\leq
C_\ell \left |\frac{|\bxi^{\h,\bn_{\bq}}|^2+1}{|\bxi^{\h,\bn_{\bq}}|^2-1}\right |^\ell \h^{\ell/2}
\leq
C_{\ell,\lambda}\h^{\ell\varepsilon}
\quad
\forall \ell \in \N,
\end{equation}
uniformly for all the values of $\bq$ as in the statement of the Lemma. Indeed, we always have
\begin{equation*}
\left |\frac{|\bxi^{\h,\bn_{\bq}}|^2+1}{|\bxi^{\h,\bn_{\bq}}|^2-1}\right |
\leq
\frac{C}{\lambda}
\h^{-1/2+\varepsilon}.
\end{equation*}
for such values of $\bq$ due to~\eqref{tmp_lb_lambda}.

We now remark that if $\cR$ is a differential operator, we have
\begin{equation*}
\cR Q u_\h^{\rm inc}
=
\frac{1}{|\bxi^{\h,\bn_{\bq}}|^2-1}
\cR (\h^2 \Delta u_\h^{\rm inc}+|\bxi^{\h,\bn_{\bq}}|^2 u_\h^{\rm inc})
=
\frac{1}{|\bxi^{\h,\bn_{\bq}}|^2-1}
\cR (-u_\h^{\rm inc}+|\bxi^{\h,\bn_{\bq}}|^2 u_\h^{\rm inc})
=
\cR u_\h^{\rm inc}.
\end{equation*}
It follows that, for such an operator $\cR$, we have
\begin{equation*}
\cR u_\h^{\rm inc}
=
Q \cR u_\h^{\rm inc}
+
[\cR,Q] u_\h^{\rm inc}.
\end{equation*}
We now iterate this relation, starting with $\cR_0 \eq \chi$,
and then letting $\cR_{\ell+1} \eq [\cR_{\ell},Q]$
for $\ell \geq 0$. This gives
\begin{equation}
\label{tmp_QR_induction}
f_\h
=
\chi u_\h^{\rm inc}
=
\sum_{\ell=0}^m \left (\begin{array}{c} m \\ \ell\end{array}\right )
Q^\ell \cR_{m-\ell} u_\h^{\rm inc}
\end{equation}
for all $m \in \N$.

We are going to show that
\begin{equation}
\label{tmp_bound_Ru}
\|\cR_{\ell} u_\h^{\rm inc}\|_{L^2(\R^d)}
\leq
C_{\ell,\lambda,\chi} \h^{\ell\varepsilon} \|u_\h^{\rm inc}\|_{H^1_\h(\cO)}.
\end{equation}
This together with~\eqref{tmp_QR_induction} concludes the proof,
since from~\eqref{tmp_Q_Psi}, we have
\begin{equation*}
|\langle f_\h, \eg{\bq} \rangle|
=
\left |
\sum_{\ell=0}^m \left (\begin{array}{c} m \\ \ell\end{array}\right )
\langle \cR_{m-\ell} u_\h^{\rm inc},Q^\ell \eg{\bq} \rangle
\right |
\leq
C_{m,\lambda,\chi} \h^{m\varepsilon} \|u_\h^{\rm inc}\|_{H^1_\h(\cO)}
\end{equation*}
for all $m \in \N$.

We now establish~\eqref{tmp_bound_Ru}, by showing by induction that
for each $\ell \in \N$, there exist constants $c_\ell^{(\balpha,\bbeta)} \in \C$
such that
\begin{equation}
\label{tmp_induction_Rell}
\cR_\ell
=
\frac{\h^{2\ell}}{(|\bxi^{\h,\bn_{\bq}}|^2-1)^\ell}
\sum_{\substack{[\balpha] \leq 2\ell \\ [\beta] \leq \ell}}
c_\ell^{(\balpha,\bbeta)} \partial^{\balpha} \chi \partial^{\bbeta}.
\end{equation}
We first note that~\eqref{tmp_induction_Rell} clearly holds true for $\ell = 0$.
Assuming the relation holds for some $\ell \geq 0$, we have
\begin{equation*}
\cR_{\ell+1}
=
[\cR_\ell,Q]
=
\frac{\h^{2\ell}}{(|\bxi^{\h,\bn_{\bq}}|^2-1)^\ell}
\sum_{\substack{[\balpha] \leq 2\ell \\ [\bbeta] \leq \ell}}
c_\ell^{(\balpha,\bbeta)}
[\partial^{\balpha} \chi \partial^{\bbeta},Q]
=
\frac{\h^{2(\ell+1)}}{(|\bxi^{\h,\bn_{\bq}}|^2-1)^{\ell+1}}
\sum_{\substack{[\balpha] \leq 2\ell \\ [\bbeta] \leq \ell}}
c_\ell^{(\balpha,\bbeta)}
[\partial^{\balpha} \chi \partial^{\bbeta},\Delta].
\end{equation*}
since multiplication by the constant $|\bxi^{\h,\bn_{\bq}}|^2$
commutes with the differentiation operators.
Classically, we have for $1 \leq j \leq d$
\begin{equation*}
[\partial^{\balpha} \chi \partial^{\bbeta},\partial_j^2]
=
-\partial^{\balpha+2\be_j} \chi \partial^{\bbeta}
-
2
\partial^{\balpha+\be_j} \chi \partial^{\bbeta+\be_j},
\end{equation*}
from which~\eqref{tmp_induction_Rell} follows at rank $\ell+1$, and hence at all ranks.
A direct consequence of~\eqref{tmp_induction_Rell}, \eqref{tmp_lb_lambda} and of (\ref{tmp_elliptic_regularity})
is now that
\begin{equation*}
\|\cR_\ell u_\h^{\rm inc}\|_{L^2(\R^d)}
\leq
C_{\ell,\lambda}
\h^{\ell} \h^{\ell\varepsilon}
\sum_{\substack{[\balpha] \leq 2\ell \\ [\bbeta] \leq \ell}}
\|\partial^{\balpha} \chi\|_{L^\infty}
\|\partial^{\bbeta}u_\h^{\rm inc}\|_{L^2(\cO)}
\leq
C_{\ell,\lambda,\chi}
\h^{\ell} \h^{\ell\varepsilon}
\sum_{\substack{[\balpha] \leq 2\ell \\ [\bbeta] \leq \ell}}
\h^{-[\bbeta]}
\|u_\h^{\rm inc}\|_{H^1_\h(\cO)},
\end{equation*}
from which~\eqref{tmp_bound_Ru} follows. This concludes the proof.
\end{proof}

\begin{theorem}\label{thm:IncidentField}
Consider $f_\h$ as in Lemma~\ref{lemma_dot_uinc} above.
Then, for all $\lambda > 0$, we have
\begin{equation}
\label{eq_approx_uinc}
\left \|
f_\h-\sum_{\bq \in \wit{\Lambda}_{\h,0}} \langle f_\h,\eg{\bq}^\star \rangle \eg{\bq}
\right \|_{\HH^p_\h(\R^d)}
=
O(\h^\infty) \|u_\h^{\rm inc}\|_{H^1_\h(\cO)}
\end{equation}
where the implied constant in the $O(\h^\infty)$ notation only depends on
$\varepsilon$, $\lambda$, $p$, the norms $\|\chi\|_{C^\ell(\cO)}$
for $\ell \in \N$, $\delta$, and the smallest $R$ such that $|\bx| \leq R$
for all $\bx \in \cO$.
\end{theorem}

\begin{proof}
Due to Assumption~\ref{Hyp:BoundedLayers}, for all $\h$ small enough,
$\wit{\Lambda}_{\h,0} \subset B_{\h^{-1}}$, where
$B_{\h^{-1}} \eq \{ \br \in \Z^{2d} \; | \; |\bz^{\h,\br}| \leq \h^{-1} \}$.
For future reference, we note that $\mathrm{Card} (B_{\h^{-1}}) \leq C\h^{-3d}$.

Simple triangular inequalities along with \eqref{eq_dual_frame} now reveals that
\begin{align*}
\left \|
f_\h-\sum_{\bq \in \wit{\Lambda}_{\h,0}} \langle f_\h,\eg{\bq}^\star \rangle \eg{\bq}
\right \|_{\HH_{\h}^p}
&\leq
\left \|
f_\h-\sum_{\bq \in \B_{\h^{-1}}} \langle f_\h,\eg{\bq}^\star \rangle \eg{\bq}
\right \|_{\HH_{\h}^p}
\\
&+
\sum_{\bq \in \B_{\h^{-1}} \setminus \wit{\Lambda}_{\h,0}}
|\langle f_\h,\eg{\bq}^\star \rangle| \|\eg{\bq}\|_{\HH_{\h}^p}.
\end{align*}
For the first term, we can immediately invoke~\cite[Theorem 3.1]{chaumontfrelet_ingremeau_2022a}
together with~\eqref{tmp_elliptic_regularity}, leading to
\begin{equation*}
\left \|
f_\h-\sum_{\bq \in \B_{\h^{-1}}} \langle f_\h,\eg{\bq}^\star \rangle \eg{\bq}
\right \|_{\HH_{\h}^p}
\leq
C_{p, m} \h^{m} \|f_\h\|_{\HH_{\h}^{p+m}}
\leq
C_{p, m,\chi} \h^m \|u_\h^{\rm inc}\|_{H^1_\h}
\end{equation*}
for all $m \geq 0$. For the second term, we use the finite cardinality of $B_{\h^{-1}}$
together with the estimate for the $\HH_{\h}^p$ norms of $\eg{\bq}$ in~\cite[Lemma C.1]{chaumontfrelet_ingremeau_2022a}. This gives
\begin{equation*}
\sum_{\bq \in \B_{\h^{-1}} \setminus \wit{\Lambda}_{\h,0}}
|\langle f_\h,\eg{\bq}^\star \rangle| \|\eg{\bq}\|_{\HH_{\h}^p}
\leq
C \h^{-d}
\max_{\bq \in \B_{\h^{-1}} \setminus \wit{\Lambda}_{\h,0}}
|\langle f_\h,\eg{\bq}^\star \rangle| \|\eg{\bq}\|_{\HH_{\h}^p}
\leq
C_{p} \h^{- d-p}
\max_{\bq \in \B_{\h^{-1}} \setminus \wit{\Lambda}_{\h,0}} |\langle f_\h,\eg{\bq}^\star \rangle|.
\end{equation*}
As a result, it remains to show that
\begin{equation*}
|\langle f_\h,\eg{\bq}^\star \rangle|
=
O_{\chi}(\h^\infty) \|u_\h^{\rm inc}\|_{H^1_\h(\cO)}
\end{equation*}
uniformly for $\bq \in \Z^{2d} \setminus \wit{\Lambda}_{\h,0}$.
To do so, we invoke \eqref{eq_dual_frame} which shows that for all
$\bq \in \Z^{2d}$, we have
\begin{align*}
\langle f_\h,\eg{\bq}^\star \rangle
&=
\sum_{\bq' \in \Z^{2d}}
\langle \eg{\bq'}^\star,\eg{\bq}^\star \rangle
\langle f_\h,\eg{\bq'} \rangle
\\
&=
\sum_{\substack{\bq' \in \Z^{2d} \\ |\bq-\bq'| \leq \tau \lambda \h^{-\varepsilon}}}
\langle \eg{\bq'}^\star,\eg{\bq}^\star \rangle
\langle f_\h,\eg{\bq'} \rangle
+
\sum_{\substack{\bq' \in \Z^{2d} \\ |\bq-\bq'| > \tau \lambda \h^{-\varepsilon}}}
\langle \eg{\bq'}^\star,\eg{\bq}^\star \rangle
\langle f_\h,\eg{\bq'} \rangle
\end{align*}
where $0 < \tau \leq 1$ is a (small) parameter to be fixed later.
The value of $\tau$ will ultimately only depend on constants appearing in
Assumptions~\ref{Hyp1:BoundCoefs} and~\ref{Hyp:BoundedLayers}, so we don't
list it in the $O(\h^\infty)$ notations.

On the one hand,~\cite[Proposition 5.2]{chaumontfrelet_ingremeau_2022a} ensures that
$|\langle \eg{\bq'}^\star,\eg{\bq}^\star \rangle| \leq C e^{-|\bq-\bq'|^{1/2}}$,
from which we deduce that
\begin{equation*}
\left |
\sum_{\substack{\bq' \in \Z^{2d} \\ |\bq-\bq'| > \tau \lambda \h^{-\varepsilon}}}
\langle \eg{\bq'}^\star,\eg{\bq}^\star \rangle
\langle f_\h,\eg{\bq'} \rangle
\right |
=
O(\h^\infty) \|f_\h\|_{L^2}
=
O_{\chi}(\h^\infty) \|u_\h^{\rm inc}\|_{H^1_\h(\cO)}
\end{equation*}
for the first term. On the other hand, we have
\begin{equation}\label{eq:SecondTerm}
\left |\sum_{\substack{\bq' \in \Z^{2d} \\ |\bq-\bq'| \leq \tau \lambda \h^{-\varepsilon}}}
\langle \eg{\bq'}^\star,\eg{\bq}^\star \rangle
\langle f_\h,\eg{\bq'} \rangle
\right |
\leq
C\lambda^{2d}\h^{-2d\varepsilon}
\max_{\substack{\bq' \in \Z^{2d} \\ |\bq-\bq'| \leq \tau \lambda \h^{-\varepsilon}}}
|\langle f_\h,\eg{\bq'} \rangle|.
\end{equation}

Let us suppose for contradiction that the previous expression contains an index
$\bq'$ such that $|p_\h(\bz^{\h,\bq'})|\leq \frac{\lambda}{2}\h^{1/2-\varepsilon}$.
Since  $p_\h$ is Lispchitz continuous on the ball introduced in
Assumption~\ref{Hyp:BoundedLayers}, we have
\begin{equation*}
|p_\h(\bz^{\h,\bq})|
\leq
|p_\h(\bz^{\h,\bq'})|
+
C |\bz^{\h,\bq}-\bz^{\h,\bq'}|
\leq
\frac{\lambda}{2} \h^{1/2-\varepsilon}
+
C \tau \lambda \h^{1/2-\varepsilon}
<
\lambda \h^{1/2-\varepsilon}
\end{equation*}
upon selecting $\tau$ small enough, which contradicts the fact that $\bq \in \Z^{2d} \setminus \wit{\Lambda}_{\h,0}$.\newline
Therefore, the expression  in (\ref{eq:SecondTerm}) contains only indices $\bq'$ such that $|p_\h(\bz^{\h,\bq'})|> \frac{\lambda}{2}\h^{1/2-\varepsilon}$, and we may apply Lemma~\ref{lemma_dot_uinc} (with $\lambda$ replaced with $\frac{\lambda}{2}$ and with $\alpha = 0$) to show that (\ref{eq:SecondTerm}) is $O(\h^\infty)$.
\end{proof}

We deduce the following corollary, which, along with the result of
Section~\ref{subsec:ApproxInGoodSpace}, directly implies Corollary~\ref{corollary_uinc}.

\begin{corollary}
\label{corollary_rhs_uinc}
Let $\chi$ and $u_\h^{\rm inc}$ as in Lemma~\ref{lemma_dot_uinc} above.
Then,~\eqref{eq_approx_uinc} still holds for $f_\h \eq (\h^2\Delta+1)(\chi u_\h^{\rm inc})$.
\end{corollary}

\begin{proof}
We simply observe that
\begin{equation*}
f_\h
=
\h^2(\Delta \chi)u_\h^{\rm inc}
+
\h^2 \chi \Delta u_\h^{\rm inc}
+
2 \h^2 \nabla \chi \cdot \nabla u_\h ^{\rm inc}
+
\chi u_\h^{\rm inc} 
=
\h^2(\Delta \chi)u_\h^{\rm inc}
+
2 \h^2 \nabla \chi \cdot \nabla u_\h ^{\rm inc},
\end{equation*}
and apply~\eqref{eq_approx_uinc} to each of the terms in the sum. Indeed, this is a linear combination of terms
of the form ``$\widetilde{\chi} \widetilde{u}_h^{\rm inc}$'' with $\widetilde{\chi}$
and $\widetilde{u}_\h^{\rm inc}$ satisfying the required assumptions.
\end{proof}

\section{Convergence of the Galerkin method}
\label{section_convergence_galerkin}

In all this section, we suppose that Assumption~\ref{Hyp4:SymetrieSymbole} holds,
along with Assumptions~\ref{Hyp1:BoundCoefs},~\ref{Hyp:PolynResolv}
and~\ref{Hyp:BoundedLayers}.

\subsection{An approximate right-inverse away from the characteristic sets}

We now proceed with the proof of Theorem \ref{theorem_accuracy}.
To this end, we fix, throughout Section~\ref{section_convergence_galerkin},
$0 < \varepsilon < 1/2$, $0 \leq \alpha < \delta_0/2$ and $\lambda'> \lambda>0$,
and we build the same energy layers as in section \ref{subsec:EnergyLayers}
(taking $J=6$ here). As in the previous section, all the objects and constants
we introduce in this section may depend on the parameters $\varepsilon$, $\alpha$,
$\lambda$ and $\lambda'$, even though we will not write this dependence explicitly.

We also add a last energy layer
\begin{equation*}
\Lambda_{\h, J+1} :=   \Lambda_\h \setminus \Lambda_\h(\lambda'),
\end{equation*}
so that
\begin{equation*}
\Lambda_\h  = \bigsqcup_{j=0}^{J+1} \Lambda_{\h, j}.
\end{equation*}

As previously,  if $0\leq i \leq j \leq J+1$, we write $\Lambda_{\h, i,j} = \bigsqcup\limits_{\ell=i}^j \Lambda_{\h,\ell}$, and
\begin{equation*}
W_{\h, i,j} = \Vect \{\Phi_{\h,\bq} ; \bq\in \Lambda_{\h,i,j} \} = \cD_\h \left(\ell^2(\Lambda_{\h,i,j}) \right),
\end{equation*}
with  $\cD_\h$ as in \eqref{eq:DefOperateurReconstruction}.  We also write $W_{\h,j}$ instead of $W_{\h, j,j}$.

Recall that, if $w_\h\in W_{\h, 1, J+1}$, \eqref{eq:UniquenessDecomposition} implies that
\begin{equation}\label{eq:ExpressionMorceau}
w_{\hbar} = \sum_{\qvec \in \Lambda_{\h,1,J+1}} \langle w_\hbar, \w{\qvec}^\star\rangle  \w{\qvec}.
\end{equation} 

As before,  $(\mathcal{P}_\hbar, \pi_\hbar)$ will denote either $(P_\hbar, p_\hbar)$ or  $(P_\hbar^*, \overline{p_\hbar})$. We define  an operator $\Inv =  \Inv(\cP_\h) : W_{\h, 1, J+1} \longrightarrow W_{\h, 1, J+1}$ by
    \begin{equation}\label{eq:DefApproxInverse}
    \Inv w_\h := \sum_{\qvec \in  \Lambda_{\h,1,J+1}} \langle w_\hbar, \w{\qvec}^\star\rangle  \frac{1}{\pi_\hbar(\bz^{\h, \bq})} \w{\qvec}.
    \end{equation}

Note that, since we only consider  $w_\h\in W_{\h, 1, J+1}$ here, we always have $|\pi_\hbar(\bz^{\h, \bq})|\geq \lambda  \h^{\frac{1}{2}- \varepsilon}$. Therefore,  we deduce from (\ref{eq: relation entre la norme de U et celle de u pour Wilson}) that there exists $C>0$ such that, for all $\h\in (0,1]$, we have
\begin{equation}\label{eq:BoundedAproximateInverse}
\forall w_\h\in W_{\h, 1, J+1}, ~~ \|\Inv w_\h\|_{L^2}\leq C \h^{-\frac{1}{2}+\varepsilon} \|w_\h\|_{L^2}.
\end{equation}
    
    The following lemma tells us that the operator $\Inv$ is an approximate right-inverse for $\cP_\h$.

\begin{lemma}
    \label{lemme: erreur entre P étoile v67tilde et v67}
For any $w_\h\in W_{\h, 1, J+1}$, we have
        \begin{equation}
        \|\mathcal{P}_\hbar   \Inv  w_{\hbar} - w_{\hbar}\|_{L^2} \leq C \hbar^\frac{\varepsilon}{2} \|w_{\hbar}\|_{L^2}.
    \end{equation}
\end{lemma}

\begin{proof}
 We set  $R_\hbar:=\mathcal{P}_\hbar    \Inv w_{\hbar} - w_{\hbar}$, so that
    \begin{align*}
        R_\hbar   &= \sum_{\qvec \in  \Lambda_{\h, 1, J+1}} \langle w_\hbar, \w{\qvec}^\star\rangle  \left[\frac{1}{\pi_\hbar(\bz^{\h, \bq})} \mathcal{P}_\hbar \w{\qvec} - \w{\qvec}\right].
    \end{align*}

Let $0<\varepsilon_1 < \frac{1}{2}$, which we will fix below.
We may apply Lemma~\ref{lemme: 1er lemme de la proposition clé} (possibly with a smaller $\varepsilon$)
to each of the $\w{\qvec}$ in the previous sum, to obtain
\begin{equation*}
R_\hbar
=
-\sum_{\qvec \in \Lambda_{\h, 1, J+1}}
\sum_{\substack{\rvec \in E_d\\|\qvec - \rvec| \leq 2 \hbar^{-\varepsilon_1}}}
\langle w_{\hbar}, \w{\qvec}^\star\rangle
\wit{V}_{\hbar, \qvec, \rvec, \varepsilon_1} \w{\rvec}
-
\wit{\gamma}_{\hbar, \varepsilon_1},
\end{equation*}
with
\begin{equation*}
\wit{\gamma}_{\hbar, \varepsilon_1}
:=
\sum_{\qvec \in \Lambda_{\h,1,J+1}}
\langle w_{\hbar}, \w{\qvec}^\star\rangle  \gamma_{\hbar, \qvec, \varepsilon_1},
\end{equation*}
so that $\|\wit{\gamma}_{\hbar, \varepsilon_1}\|_{L^2}= O_{\varepsilon_1}(\h^\infty)\times \|w_\h\|_{L^2}$.

We deduce that
\begin{align*}
\|R_\hbar\|_{L^2}^2
&\leq
C \left \|
\sum_{\rvec \in E_d}
\sum_{\substack{\qvec \in \Lambda_{\h,1,J+1} \\|\qvec - \rvec| \leq 2 \hbar^{-\varepsilon_1}}}
\langle w_{\hbar}, \w{\qvec}^\star\rangle
\wit{V}_{\hbar, \qvec, \rvec, \varepsilon_1} \w{\rvec}
\right \|_{L^2}^2
+
O_{\varepsilon_1}(\h^\infty)\times \|w_\h\|_{L^2}^2
\\
&\stackrel{\eqref{eq: relation entre la norme de U et celle de u pour Wilson}}{\leq}
C \sum_{\rvec \in E_d}
\left |
\sum_{\substack{\qvec \in \Lambda_{\h,1,J+1}\\|\qvec - \rvec| \leq 2 \hbar^{-\varepsilon_1}}}
\langle w_{\hbar}, \w{\qvec}^\star\rangle
\wit{V}_{\hbar, \qvec, \rvec, \varepsilon_1}
\right |^2
+
O_{\varepsilon_1}(\h^\infty)\times \|w_\h\|_{L^2}^2
\\
&\leq
C \sum_{\rvec \in E_d}
\left (
\sum_{\substack{\qvec \in \Lambda_{\h,1,J+1} \\|\rvec - \qvec| \leq 2 \hbar^{-\varepsilon_1}}}
\left |\left\langle w_{\hbar}, \w{\qvec}^\star\right\rangle \right|^2
\right )
\left (
\sum_{\substack{\svec \in \Lambda_{\h,1,J+1}\\|\rvec - \svec| \leq 2 \hbar^{-\varepsilon_1}}}
\left |\wit{V}_{\hbar, \svec, \rvec, \varepsilon_1}\right |^2
\right )
+
O_{\varepsilon_1}(\h^\infty)\times \|w_\h\|_{L^2}^2,
\end{align*} 
    by the Cauchy-Schwarz inequality.
    We have  $\left|\wit{V}_{\hbar, \svec, \rvec, \varepsilon_1}\right|^2 \leq C \hbar^{2 \varepsilon - 4d \varepsilon_1}$,  so that  $\sum\limits_{\substack{\svec \in\Lambda_{\h,1,J+1}\\|\rvec - \svec| \leq 2 \hbar^{-\varepsilon_1}}} \left|\wit{V}_{\hbar, \svec, \rvec, \varepsilon_1}\right|^2 \leq C \hbar^{2 \varepsilon - 6d \varepsilon_1}$. Therefore
    \begin{align*}
        \|R_\hbar\|_{L^2}^2 &\leq C \hbar^{2 \varepsilon - 6d \varepsilon_1} \sum_{\rvec \in E_d} \sum_{\substack{\qvec \in E_d\\|\rvec - \qvec| \leq 2 \hbar^{-\varepsilon_1}}} \left|\left\langle w_{\hbar}, \w{\qvec}^\star\right\rangle \right|^2 + O_{\varepsilon_1}(\h^\infty)\times \|w_\h\|_{L^2}^2\\
        &\leq C \hbar^{2 \varepsilon - 8d \varepsilon_1} \sum_{\qvec \in E_d} \left|\left\langle w_{\hbar}, \w{\qvec}^\star\right\rangle \right|^2 + O_{\varepsilon_1}(\h^\infty)\times \|w_\h\|_{L^2}^2\\
        &\leq  C \hbar^{2 \varepsilon - 8d \varepsilon_1} \|w_{\hbar}\|_{L^2}^2  + O_{\varepsilon_1}(\h^\infty)\times \|w_\h\|_{L^2}^2,
    \end{align*}
    thanks to  \eqref{eq: inégalité de trame pour Wilson étoile}. Taking $\varepsilon_1= \frac{\varepsilon}{8d}$, we obtain $\|R_\hbar\|_{L^2} \leq C \hbar^\frac{\varepsilon}{2} \|w_{\hbar}\|_{L^2}$,  which concludes the proof.
\end{proof}

\subsection{Microlocal properties of the Galerkin solution}
\label{section: micro-localisation de la solution de Galerkine}
In this subsection and the next, we consider a family of functions $(f_\h)_{\h \in \LH}$ with $f_\h \in W_{\h,0}$ for all $\h \in (0,1]$,
 and we consider the following Galerkin problem: 
\begin{equation}
    \label{eq: Problème de Galerkine}
     \quad \forall w_\hbar \in W_\h, \quad \langle P_\hbar v_\hbar, w_\hbar\rangle = \langle f_\hbar, w_\hbar\rangle.
\end{equation}

This problem does not have to well-posed a priori, but we will show that this is indeed
the case if $\h$ is small enough. In the remainder of this section, we will prove estimates
for any $v_\h$ solution to~\eqref{eq: Problème de Galerkine}. The existence and uniqueness
of such $v_\h$ will be established a posteriori in the proof of Theorem~\ref{theorem_accuracy}
below.

Let therefore $v_\h$ be any solution to~\eqref{eq: Problème de Galerkine}.
For every $0\leq i\leq j \leq J+1$, we shall write
\begin{equation}\label{eq:DecompoVLayers}
v_{\h,i,j} = \sum_{\qvec \in \Lambda_{\h, i,j}} \langle v_\hbar, \w{\qvec}^\star\rangle \w{\qvec},
\end{equation}
and we will simply write $v_{\h,j}$ instead of $v_{\h,j,j}$ when $i=j$.  Note that, by the frame properties, there exists $C>0$ such that, for all $0\leq i \leq j \leq J+1$ and all $\h>0$,
\begin{equation}\label{eq:BornerLesMorceaux}
\|v_{\h, i,j}\|_{L^2}\leq C \|v_h\|_{L^2}.
\end{equation}

In the sequel, we will decompose $v_\h$ as
\begin{equation}\label{eq:DecompoVen3Morceaux}
v_\h= v_{\h, 0,3} +v_{\h,4} + v_{\h,5, J+1}.
\end{equation}

The aim of this subsection is to prove the following proposition, which will be a direct consequence of Lemma  \ref{lemme: majoration de la norme H de v5} and Corollary \ref{lemme: majoration de la norme H de v67} below.  

\begin{proposition}
\label{prop: micro-localisation de la solution de Galerkine}
There exists $\h_0 >0$ such that, for any $p\in \N$, we have for $\h \in \LH \cap (0,\h_0]$
\begin{equation*}
\|v_\hbar - \vA\|_{\widehat{H}_\hbar^p}= O_p(\h^\infty) \times  \left(\|f_\hbar\|_{L^2} + \|v_\hbar\|_{L^2}\right).
\end{equation*}
\end{proposition}

\subsubsection{Bounding $\|\vB\|_{\widehat{H}_\hbar^p(\R^d)}$}

\begin{lemma}
    \label{lemme: majoration de (v, P étoile Phi)}
    For every $\rvec \in \Lambda_{\h,2,J+1}$, we have for $\h\in \LH$
    \begin{equation}
        |\langle v_\hbar, P_\hbar^* \w{\rvec}\rangle | =O(\h^\infty) \times \|f_\hbar\|_{L^2}.
    \end{equation}
\end{lemma}

\begin{proof}
Let $\rvec \in \Lambda_{\h,2,J+1}$.
Thanks to \eqref{eq: Problème de Galerkine}, we have 
    \begin{equation*}
        \langle v_\hbar, P_\hbar^* \w{\rvec}\rangle  = \langle P_\hbar v_\hbar, \w{\rvec}\rangle  = \langle f_\hbar, \w{\rvec}\rangle .
    \end{equation*}
We know from our assumption on $f_\h$ and from (\ref{eq:UniquenessDecomposition}) that $f_\hbar = \sum\limits_{\svec \in \Lambda_{\h,0}} \langle f_\hbar, \w{\svec}^\star\rangle  \w{\svec}$, so that
    \begin{equation*}
        |\langle v_\hbar, P_\hbar^* \w{\rvec}\rangle | \leq  \|f_\hbar\|_{L^2} \sum_{\svec \in \Lambda_{\h,0}} \|\w{\svec}^\star\|_{L^2} |\langle \w{\svec}, \w{\rvec}\rangle |.
    \end{equation*}
Now, thanks to Lemmas \ref{lemme: majoration de psi scalaire phi étoile} and \ref{lemme: épaisseur des niveaux d'énergie},  we have $|\langle \w{\svec}, \w{\rvec}\rangle | \leq C e^{-\frac{\pi}{4}|\svec - \rvec|} \leq C e^{-\frac{\pi}{4} C' \hbar^{-\varepsilon}} = O(\h^\infty)$. 
The result then follows using Lemma \ref{lem:LePtiLemKiSauveToutLTemps} along with \eqref{eq:BornePhiStar}.
\end{proof}

\begin{lemma}
    \label{lemme: majoration de la norme H de v5}
For every $p\in \N$, we have for $\h\in \LH$
    \begin{equation*}
        \|\vB\|_{\widehat{H}_\hbar^p} = O_{p}(\h^\infty)\times  \left(\|f_\hbar\|_{L^2} + \|v_\hbar\|_{L^2}\right).
    \end{equation*}
\end{lemma}

\begin{proof}
    \textbf{Step 1: Bounding the scalar product between $v_\h$ and some of the $\Phi_{\h, \bs}$.}
We set $\varepsilon' := \frac{\varepsilon}{2}$.  Thanks to Proposition \ref{prop: proposition-clé},  we know that for any $N\in \N$ and any $\svec \in\Lambda_{\h, 3,5}$, we may write
    \begin{equation*}
        \w{\svec} = \sum_{\substack{\rvec \in E_d\\|\svec - \rvec| \leq \hbar^{-\varepsilon/2}}} U_{\hbar, \svec, \rvec}^{N} P_\hbar^* \w{\rvec} + \eta_{\hbar, \svec, N}
    \end{equation*}
    with $\left|U_{\hbar, \svec, \rvec}^{N}\right| \leq C_{N} \hbar^{-\frac{1}{2}}$ and $\|\eta_{\hbar, \svec, N}\|_{L^2} \leq C_{N} \hbar^N$.
    We thus have
    \begin{equation}\label{eq:ScalarVPhi}
        \langle v_\hbar, \w{\svec}\rangle  = \sum_{\substack{\rvec \in E_d\\|\svec - \rvec| \leq \hbar^{-\varepsilon/2}}} \overline{U_{\hbar, \svec, \rvec}^{N}} \langle v_\hbar, P_\hbar^* \w{\rvec}\rangle  + \langle v_\hbar, \eta_{\hbar, \svec,  N}\rangle .
    \end{equation}

We may then apply (\ref{eq:s proches de q dans lambda j-1,j+1}) to see that all the $\br$ in the sum  belong to $\Lambda_{\h, 2,6}$ for $\h$ small enough. Thanks to Lemma \ref{lemme: majoration de (v, P étoile Phi)},  they thus satisfy $|\langle v_\hbar, P_\hbar^* \w{\rvec}\rangle | \leq C_{N} \hbar^N \|f_\hbar\|_{L^2}$.  Since the sum in  (\ref{eq:ScalarVPhi}) contains $O(\h^{-\varepsilon d})$ terms,  and since we may take $N$ arbitrarily large, we deduce that:
    \begin{equation}
        \label{eq: majoration de (v, Phi)}
      \forall \svec \in\Lambda_{\h, 3,5},~~~~  |\langle v_\hbar, \w{\svec}\rangle | = O(\h^\infty)\times \left(\|f_\hbar\|_{L^2} + \|v_\hbar\|_{L^2}\right).
    \end{equation}

    \textbf{Step 2: Bounding the scalar product between $v_\h$ and some of the $\Phi^\star_{\h, \bs}$.}
Thanks to Corollary \ref{cor: formules d'approximation 1},  we have, for all $\qvec \in \Lambda_{\h, 4}$,
    \begin{equation*}
        \w{\qvec}^\star = \sum_{\substack{\svec \in E_d\\|\qvec - \svec| \leq \hbar^{-\varepsilon'}}} \langle \w{\qvec}^\star, \w{\svec}^\star\rangle  \w{\svec} + \rho_{\hbar, \qvec}
    \end{equation*}
    with $\|\rho_{\hbar, \qvec}\|_{L^2} = O(\h^\infty)$, and thanks to Lemma \ref{lemme: épaisseur des niveaux d'énergie}, the $\bs$ in this sum do all belong to $\Lambda_{\h, 3,5}$.  We therefore have, for all $\qvec \in \Lambda_{\h,4}$,
    \begin{equation}\label{eq:MajorationPSV PhiStar}
        \langle v_\hbar, \w{\qvec}^\star\rangle  = \sum_{\substack{\svec \in E_d\\|\qvec - \svec| \leq \hbar^{-\varepsilon'}}} \overline{\langle \w{\qvec}^\star, \w{\svec}^\star\rangle } \langle v_\hbar, \w{\svec}\rangle  + \langle v_\hbar, \rho_{\hbar, \qvec}\rangle  = O(\h^\infty)\times \left(\|f_\hbar\|_{L^2} + \|v_\hbar\|_{L^2}\right),
    \end{equation}
    thanks to \eqref{eq: majoration de (v, Phi)}.
We may therefore deduce from Lemma \ref{lem:LePtiLemKiSauveToutLTemps}  and from \eqref{eq_norm_Phi_star} that, 
for any $p\in \N$,
    \begin{equation*}
        \|\vB\|_{\widehat{H}_\hbar^p} \leq \sum_{\qvec \in\Lambda_{\h,4}} \left|\left\langle v_\hbar, \w{\qvec}^\star\right\rangle \right| \times \|\w{\qvec}\|_{\widehat{H}_\hbar^p} = O_{p}(\h^\infty)\times \left(\|f_\hbar\|_{L^2} + \|v_\hbar\|_{L^2}\right).
    \end{equation*}
\end{proof}

\subsubsection{Bounding $\|\vC\|_{\widehat{H}_\hbar^p}$}

Our next goal is to estimate $\vC$.
Remember that $\vC$ has been defined in (\ref{eq:DecompoVLayers}), and that the operator $\Inv$ has been introduced in (\ref{eq:DefApproxInverse}). Here, we only consider the case when $\cP_\h= P_\h^*$, so that we will have $\Inv = \Inv(P_\h^*)$.

\begin{lemma}
    \label{lemme: majoration de (v67, P étoile v67tilde)}
   We have for $\h\in \LH$
    \begin{equation}
        |\langle \vC, P_\hbar^* \Inv \vC \rangle| = O(\h^\infty) \times \left(\|f_\hbar\|_{L^2(\R^d)} \|v_\hbar\|_{L^2(\R^d)} + \|v_\hbar\|_{L^2(\R^d)}^2\right).
    \end{equation}
\end{lemma}

\begin{proof}
Thanks to (\ref{eq:DecompoVen3Morceaux}), we have $P_\hbar \vC = P_\hbar v_\hbar - P_\hbar \vA - P_\hbar \vB$,  so that 
\begin{multline}
\label{eq: majoration de (v67, P étoile v67tilde) par inégalité triangulaire}
|\langle \vC, P_\hbar^*\Inv  \vC \rangle|
\leq
\\
| \langle P_\hbar v_\hbar, \Inv  \vC\rangle |
+
|\langle P_\hbar \vA, \Inv  \vC\rangle |
+
|\langle P_\hbar \vB, \Inv  \vC\rangle |.
\end{multline}
    
    First of all, we have
    \begin{align*}
      |\langle P_\hbar v_\hbar,  \Inv  \vC\rangle | &= \left| \sum_{\qvec \in \Lambda_{\h, 5, J+1}} \overline{\langle v_\hbar, \w{\qvec}^\star\rangle } \frac{1}{p_\hbar(\bz^{\h, \bq})} \langle v_\hbar, P_\hbar^* \w{\qvec}\rangle \right|\\
      &\leq \sum_{\qvec \in \Lambda_{\h, 5, J+1}} \|v_\hbar\|_{L^2} \|\w{\qvec}^\star\|_{L^2} \frac{1}{|p_\hbar(\bz^{\h,\bq})|} |\langle v_\hbar, P_\hbar^* \w{\qvec}\rangle |.
    \end{align*}
    Now,  since the previous sum contains $O(\h^{-d})$ terms, since $\|\w{\qvec}^\star\|_{L^2} \leq C$ (by \eqref{eq:BornePhiStar}),  and since $|p_\hbar(\bz^{\h,\bq})| \geq \lambda_4 \hbar^{\frac{1}{2} - \varepsilon}\geq \lambda \h^{\frac{1}{2}-\varepsilon}$, we deduce from Lemma \ref{lemme: majoration de (v, P étoile Phi)} that
    \begin{equation}
        \label{eq: majoration de (Pv, v67tilde)}
        |\langle P_\hbar v_\hbar,  \Inv \vC\rangle| =O(\h^\infty) \times  \|f_\hbar\|_{L^2} \|v_\hbar\|_{L^2}.
    \end{equation}

We now turn to the second term in \eqref{eq: majoration de (v67, P étoile v67tilde) par inégalité triangulaire}. We have
\begin{equation}\label{eq:MajorationDeuxiemeTerme}
    \begin{aligned}
        |\langle P_\hbar \vA,  \Inv \vC\rangle | &= \left| \sum_{\qvec \in  \Lambda_{\h, 0,3}} \sum_{\svec \in  \Lambda_{\h, 5,J+1}} \langle v_\hbar, \w{\qvec}^\star\rangle  \overline{\langle v_\hbar, \w{\svec}^\star\rangle } \frac{1}{p_\hbar(\bz^{\h, \bs})} \langle P_\hbar \w{\qvec}, \w{\svec}\rangle \right|\\
        &\leq  \sum_{\qvec \in  \Lambda_{\h, 0,3}} \sum_{\svec \in  \Lambda_{\h, 5,J+1}} \|v_\hbar\|_{L^2}^2 \|\w{\qvec}^\star\|_{L^2} \|\w{\svec}^\star\|_{L^2} \frac{1}{|p_\hbar(\bz^{\h, \bs})|} |\langle P_\hbar \w{\qvec}, \w{\svec}\rangle |\\
        &\leq  C \h^{\varepsilon- \frac{1}{2}}  \|v_\hbar\|_{L^2}^2  \sum_{\qvec \in  \Lambda_{\h, 0,3}} \sum_{\svec \in  \Lambda_{\h, 5,J+1}} |\langle P_\hbar \w{\qvec}, \w{\svec}\rangle| ,
    \end{aligned}
    \end{equation}
    by (\ref{eq:BornePhiStar}).  Now, thanks to Lemmas \ref{lemme: P w scalaire w} and  \ref{lemme: épaisseur des niveaux d'énergie} and to Assumption~\ref{Hyp:BoundedLayers}, we have for all $\qvec \in \Lambda_{\h, 0, 3}$,  $\svec \in \Lambda_{\h, 5, J+1}$ and all $m \in \N$,
    \begin{equation*}
        |\langle P_\hbar \w{\qvec}, \w{\svec}\rangle | \leq C_m \frac{1 + \left|\xin{\nvec_\qvec}\right|^{2d}}{1 + |\qvec - \svec|^m} \leq \frac{C_m}{1 + C^m \hbar^{-\varepsilon m}} \leq C_m \hbar^{\varepsilon m}.
    \end{equation*}
    Since each sum in (\ref{eq:MajorationDeuxiemeTerme}) contains only $O(\h^{-d})$ terms and since $m$ is arbitrary,  we deduce that 
    \begin{equation}
        \label{eq: majoration de (Pv14, v67tilde)}
        |\langle P_\hbar \vA, \Inv \vC\rangle | =O(\h^\infty) \|v_\hbar\|_{L^2(\R^d)}^2.
    \end{equation}
    
    Finally,  to deal with the last term in \eqref{eq: majoration de (v67, P étoile v67tilde) par inégalité triangulaire}, we write
    \begin{equation}\label{eq: majoration de (Pv5, v67tilde)}
    \begin{aligned}
        |\langle P_\hbar \vB,  \Inv \vC\rangle | &\leq \|P_\hbar \vB\|_{L^2} \|\Inv \vC\|_{L^2} \\
        &\leq C \h^{-\frac{1}{2}+\varepsilon}  \|\vB\|_{\widehat{H}_\hbar^2} \|\vC\|_{L^2} ~~~~\text{by \eqref{eq:BoundedAproximateInverse} and Lemma \ref{lemme: P est continu}}\\
        &= O (\h^\infty) \times  \left(\|f_\hbar\|_{L^2} \|v_\hbar\|_{L^2} + \|v_\hbar\|_{L^2}^2\right),
    \end{aligned}
    \end{equation}    
    thanks to Lemma  \ref{lemme: majoration de la norme H de v5} and to (\ref{eq:BornerLesMorceaux}).
    
 Combining   \eqref{eq: majoration de (v67, P étoile v67tilde) par inégalité triangulaire}, \eqref{eq: majoration de (Pv, v67tilde)}, \eqref{eq: majoration de (Pv14, v67tilde)} and \eqref{eq: majoration de (Pv5, v67tilde)}, we obtain the result.
\end{proof}

\begin{corollary}
\label{lemme: majoration de la norme H de v67}
There exists $\h_0 >0$ such that, for any $p\in \N$, we have for $\h \in \LH' \cap (0,\h_0]$
\begin{equation}
\|\vC\|_{\widehat{H}_\hbar^p} =O_p(\h^\infty) \times \left(\|f_\hbar\|_{L^2} + \|v_\hbar\|_{L^2}\right).
\end{equation}
\end{corollary}

\begin{proof}
   Recall that, thanks to Lemma \ref{lemme: erreur entre P étoile v67tilde et v67},  we have $\vC = P_\hbar^*  \Inv \vC - R_\hbar$, with 
    \begin{equation}
        \label{eq: rappel de la majoration de Rh}
        \|R_\hbar\|_{L^2} \leq C \hbar^\frac{\varepsilon}{2} \|\vC\|_{L^2}.
    \end{equation}
    We thus have 
    \begin{align*}
        \|\vC\|_{L^2}^2      &\leq |\langle \vC, P_\hbar^*\Inv  \vC\rangle | + |\langle \vC, R_\hbar \rangle |\\
        &\leq O (\h^\infty) \times  \left(\|f_\hbar\|_{L^2} + \|v_\hbar\|_{L^2}\right)^2+ C \h^{\frac{\varepsilon}{2}} \|\vC\|_{L^2}^{2}, 
    \end{align*}
thanks to Lemma  \ref{lemme: majoration de (v67, P étoile v67tilde)} and to \eqref{eq: rappel de la majoration de Rh}.
For $\h$ small enough,  we have $C \h^{\frac{\varepsilon}{2}}\leq \frac{1}{2}$, and we deduce that
    \begin{equation}
        \label{eq: majoration de la norme L2 de v67}
        \|\vC\|_{L^2} = O (\h^\infty) \left(\|f_\hbar\|_{L^2} + \|v_\hbar\|_{L^2}\right),
    \end{equation}
 so that the result follows from Lemma \ref{lem:LePtiLemKiSauveToutLTemps}.
\end{proof}

\subsection{End of the proof}

Recall that $f_\h$, $v_\h$ are as in the previous subsection.

\begin{lemma}
\label{lemme: majoration du résidu Hp}
 For any $p>0$, there exists $\h_0(p) >0$ such that we
have for $\h \in \LH' \cap (0,\h_0(p)]$
\begin{equation}
\label{eq: majoration du résidu L2}
\|P_\hbar v_\hbar - f_\hbar\|_{\widehat{H}_\hbar^p}
=O_p(\h^\infty)\times \|f_\hbar\|_{L^2}.
\end{equation}
\end{lemma}

\begin{proof}
Thanks to Proposition \ref{prop: micro-localisation de la solution de Galerkine},  we may find $\h_0>0$ such that, for any $\h \in \LH'\cap (0,h_0]$,
    \begin{equation}
        \label{eq: v = v14 + rho}
        v_\hbar = \vA + \rho_\hbar
    \end{equation}
where, for any $p\in \N$,
    $\|\rho_\hbar\|_{\widehat{H}_\hbar^p}= O_p(\h^\infty) \left(\|f_\hbar\|_{L^2} + \|v_\hbar\|_{L^2}\right)$. Furthermore, thanks to Lemma \ref{lemme: si v est micro-localisé alors Pv aussi},  we may find  $\wzero \in W_{\h, 0, 5}\subset W_\h$ such that 
    \begin{equation}
        \label{eq: Pv = w0 + eta}
        P_\hbar \vA = \wzero + \eta_\hbar
    \end{equation}
    where, for any $p\in \N$,  $\|\eta_\hbar\|_{\widehat{H}_\hbar^p} = O_{p}(\h^\infty) \|\vA\|_{L^2} = O_{p}(\h^\infty) \|v_\h\|_{L^2} $, by (\ref{eq:BornerLesMorceaux}).
    
Combining  \eqref{eq: v = v14 + rho} and  \eqref{eq: Pv = w0 + eta},  we get
    \begin{equation*}
        P_\hbar v_\hbar = \wzero + \gamma_\hbar,
    \end{equation*}
with $\gamma_\hbar := \eta_\hbar + P_\hbar \rho_\hbar$.  Thanks to Lemma  \ref{lemme: P est continu}, we have, for any $p\in \N$,
    \begin{equation*}
    \|\gamma_\hbar\|_{\widehat{H}_\hbar^p} \leq \|\eta_\hbar\|_{\widehat{H}_\hbar^p} + C_p \|\rho_\hbar\|_{\widehat{H}_\hbar^{p + 2}}   = O_{p} (\hbar^\infty) \left(\|f_\hbar\|_{L^2} + \|v_\hbar\|_{L^2}\right).
    \end{equation*}

   We may now use Assumption~\ref{Hyp:PolynResolv} to see that  
    \begin{align*}
        \|v_\hbar\|_{L^2} \leq C \hbar^{-N_0} \|P_\hbar v_\hbar\|_{L^2}
        \leq C \hbar^{-N_0} \|P_\hbar v_\hbar - f_\hbar\|_{L^2} + C \hbar^{-N_0} \|f_\hbar\|_{L^2},
    \end{align*}
so that
    \begin{equation}
        \label{eq: majoration de la norme H de gamma}
        \|\gamma_\hbar\|_{\widehat{H}_\hbar^p} = O_{p} (\hbar^\infty) \left(\|P_\hbar v_\hbar - f_\hbar\|_{L^2} + \|f_\hbar\|_{L^2}\right).
    \end{equation}

On the other hand,  we have
    \begin{align*}
        \left\|\wzero - f_\hbar\right\|_{L^2}^2 &= \left|\left\langle P_\hbar v_\hbar - f_\hbar - \gamma_\hbar, \wzero - f_\hbar\right\rangle \right|\\
        &= \left|\left\langle \gamma_\hbar, \wzero - f_\hbar\right\rangle \right|.
    \end{align*}
Indeed,  we have $\left\langle P_\hbar v_\hbar - f_\hbar, \wzero - f_\hbar\right\rangle  = 0$ by \eqref{eq: Problème de Galerkine}, since $\wzero - f_\hbar \in W_\h$.  We therefore have $  \left\|\wzero - f_\hbar\right\|_{L^2} \leq \|\gamma_\hbar\|_{L^2}$, and hence, by Lemma \ref{lem:LePtiLemKiSauveToutLTemps}
    \begin{align*}
        \left\|\wzero - f_\hbar\right\|_{\widehat{H}_\hbar^p}
        &\leq C_p \hbar^{-\frac{d}{2}} \|\gamma_\hbar\|_{L^2}.
    \end{align*}
    In the end, we obtain
    \begin{align*}
    \|P_\hbar v_\hbar - f_\hbar\|_{\widehat{H}_\hbar^p} &\leq \left\|\wzero - f_\hbar\right\|_{\widehat{H}_\hbar^p} + \|\gamma_\hbar\|_{\widehat{H}_\hbar^p}\\
    & \leq C_p \h^{-\frac{d}{2}} \|\gamma_\hbar\|_{L^2} + \|\gamma_\hbar\|_{\widehat{H}_\hbar^p}\\
    &= O_{p} (\h^\infty) \times \left(\|P_\hbar v_\hbar - f_\hbar\|_{L^2} + \|f_\hbar\|_{L^2}\right).
    \end{align*}
Taking $\h$ small enough so that the $O_{p} (\h^\infty)$ factor is $\leq \frac{1}{2}$,  we deduce that $\|P_\hbar v_\hbar - f_\hbar\|_{\widehat{H}_\hbar^p} = O_{p} (\h^\infty)\times  \|f_\hbar\|_{L^2}$.
\end{proof}

We may now give the proof of Theorem \ref{theorem_accuracy}. 

\begin{proof}[Proof of Theorem \ref{theorem_accuracy}]
$ $

\textbf{Existence and uniqueness of the solution.}
Since $T_{W_\h}$ is an endomorphism of the finite-dimensional space $W_\h$,
we only have to show that it is injective. Let $v_\h \in \Ker(T_{W_\h})$, so that
$v_\h$ is solution of~\eqref{eq: Problème de Galerkine} with $f_\h = 0$.
Thanks to Lemma~\ref{lemme: majoration du résidu Hp}, for $\h$ smaller than
$\h_0= \h_0(0)$, we have $\|P_\h v_\h\|_{L^2} = 0$,
so that $P_\h v_\h = 0$, and $v_\h = 0$ by Assumption~\ref{Hyp:PolynResolv}.  

\textbf{Asymptotic convergence. }
If $f_\h \in  W_{\h,0}$ and if $v_\h$ is the unique solution of the
Galerkin problem~\eqref{eq: Problème de Galerkine}, then we know from
Lemma~\ref{lemme: majoration du résidu Hp} that, for all $p\in \N$,
there exists $\h'_0(p) = \min(\h_0(0), \h_0(p))>0$ such that, for all
$\h \in \LH\cap (0, \h'_0(p))$, we have
\begin{equation*}
\|P_\h v_\h - f_\h\|_{\widehat{H}_\h^p}
=
O_{p} (\h^\infty) \|f_\hbar\|_{L^2}.
\end{equation*}
Since
$
\|v_\h - u_\h\|_{\widehat{H}_\h^p} = \|P_\h^{-1} (P_\h v_\h - f_\h)\|_{\widehat{H}_\h^p}
\leq
C_p \h^{-N_0} \|P_\h v_\h- f_\h\|_{\widehat{H}_\h^p}
$ by Assumption~\ref{Hyp:PolynResolv},
we obtain \eqref{eq: convergence asymptotique pour Galerkine Wilson}
for $\h\in \LH \cap (0, \h'_0(p)]$.

On the other hand, for $\h \in \LH\cap \left(\h'_0(p),\h_0(0)\right]$,
we may use the fact that
\begin{align*}
\|v_\h - u_\h\|_{\widehat{H}_\h^p} &\leq C_p \h^{-N_0} \|P_\h v_\h - f_\h\|_{\widehat{H}_\h^p} \\
&\leq  C_p \h^{-N_0} \left( \|f_\h\|_{\widehat{H}_\h^p} + C'_p \|v_{\h}\|_{\widehat{H}_\h^{p+2}} \right) \quad \text{ by Lemma \ref{lemme: P est continu}}\\
&\leq C''_p \h^{-\frac{d}{2} - N_0}   \left( \|f_\h\|_{L^2} + \|v_{\h}\|_{L^2} \right) \quad \text{ by Lemma \ref{lem:LePtiLemKiSauveToutLTemps}}\\
&\leq C''_p \h^{-\frac{d}{2} - N_0}   \left( (1+O(\h^\infty) ) \|f_\h\|_{L^2} + \|u_{\h}\|_{L^2}  \right) \quad \text{ by \eqref{eq: convergence asymptotique pour Galerkine Wilson} for $p=0$}\\
&\leq C''_p \h^{-\frac{d}{2} - N_0}  \times \left (1+O(\h^\infty)  + C \h^{-N_0} \right) \|f_\h\|_{L^2}.
\end{align*}
We therefore see that, for any $N,p\in \N$, the quantity
$\frac{\|v_\h - u_\h\|_{\widehat{H}_\hbar^p}}{\h^N \|f_\h\|_{L^2}}$
is bounded for $\h \in \LH \cap (\h_0(p), \h_0(0)]$,
which gives~\eqref{eq: convergence asymptotique pour Galerkine Wilson}.
\end{proof}

\begin{proof}[Proof of Theorem \ref{theorem_conditioning}]
$ $

\textbf{Bound on $\|T_{W_\h}^{-1}\|_{L^2 \to L^2}$.}
Let $f_\h \in W_\h$. Just as in~\eqref{eq:DecompoVLayers},
for every $0\leq i\leq j \leq J+1$, we set 
\begin{equation}\label{eq:DecompoFLayers}
f_{\h,i,j} = \sum_{\qvec \in \Lambda_{\h, i,j}} \langle f_\hbar, \w{\qvec}^\star\rangle \w{\qvec},
\end{equation}
so that $f_\h = f_{\h,0}+ f_{\h,1,J+1}$,  where we write $f_{\h,0}$ instead of $f_{\h,0,0}$.

First of all,  we know from \eqref{eq: convergence asymptotique pour Galerkine Wilson} that, for $\h \in \LH$ small enough, we have
    \begin{equation}
        \label{eq: utilisation du théorème dans la preuve du conditionnement}
        \left\| T_{W_\h}^{-1} f_{\hbar, 0} - P_\hbar^{-1} f_{\hbar, 0}\right\|_{L^2} \leq C \|f_{\hbar, 0}\|_{L^2}.
    \end{equation}
We deduce from~\eqref{eq: utilisation du théorème dans la preuve du conditionnement}
and Assumption~\ref{Hyp:PolynResolv} that, for $\h \in \LH$ small enough,
\begin{equation}\label{eq:BornerPremierBoutConditionnement}
    \begin{aligned}
        \left\|T_{W_\h}^{-1} f_{\hbar, 0}\right\|_{L^2} &\leq \left\|T_{W_\h}^{-1} f_{\hbar, 0} - P_\hbar^{-1} f_{\hbar, 0}\right\|_{L^2} + \left\|P_\hbar^{-1} f_{\hbar, 0}\right\|_{L^2}\\
        &\leq C \|f_{\hbar, 0}\|_{L^2} + C \hbar^{-N_0} \|f_{\hbar, 0}\|_{L^2}\\
        &\leq C' \hbar^{-N_0} \|f_{\hbar, 0}\|_{L^2}\\
        &\leq C' \hbar^{-N_0} \|f_{\hbar}\|_{L^2}.
    \end{aligned}
    \end{equation}
    In the last inequality, we used the analogue of (\ref{eq:BornerLesMorceaux}).

On the other hand, we know from Lemma \ref{lemme: erreur entre P étoile v67tilde et v67} that
    \begin{equation}
        \label{eq: f27 est à peu près égal à f27 tilde}
        f_{\hbar, 1, J+1} = P_\hbar \Inv f_{\hbar, 1, J+1} - R_\hbar
    \end{equation}
with $\|R_\hbar\|_{L^2} \leq C \hbar^\frac{\varepsilon}{2} \|f_{\hbar, 1, J+1}\|_{L^2}\leq C'\hbar^\frac{\varepsilon}{2} \|f_{\hbar}\|_{L^2}$ by the analogue of (\ref{eq:BornerLesMorceaux}).
Applying $\Pi_{W_\h}$ (the $L^2$-orthogonal projection on $W_\h$) to both sides of \eqref{eq: f27 est à peu près égal à f27 tilde}, we obtain
    \begin{equation*}
        f_{\hbar, 1, J+1} = T_{W_\h} \Inv f_{\hbar, 1, J+1} - \Pi_{W_\h} R_\hbar,
    \end{equation*}
    and thus 
    \begin{equation}\label{eq:BornerDeuxiemeMorceauConditionnement}
\begin{aligned}
\left\|T_{W_\h}^{-1} f_{\hbar, 1, J+1}\right\|_{L^2} &\leq \|\Inv f_{\hbar, 1, J+1}\|_{L^2}+ \| T_{W_\h}^{-1} \Pi_{W_\h} R_\hbar\|_{L^2}\\
&\leq C \h^{- \frac{1}{2}+\varepsilon}  \|f_{\hbar, 1, J+1}\|_{L^2} +  C' \| T_{W_\h}^{-1}\|_{L^2\to L^2} \h^{\frac{\varepsilon}{2}} \|f_\h\|_{L^2}~~~~\text{ by \eqref{eq:BoundedAproximateInverse}}\\
&\leq  C \h^{- \frac{1}{2}+\varepsilon}  \|f_{\hbar}\|_{L^2} +  C' \| T_{W_\h}^{-1}\|_{L^2\to L^2} \h^{\frac{\varepsilon}{2}} \|f_\h\|_{L^2}.
\end{aligned}
\end{equation}    
    
Combining \eqref{eq:BornerDeuxiemeMorceauConditionnement} with (\ref{eq:BornerPremierBoutConditionnement}), we obtain, for $\h \in \LH$ small enough,
    \begin{equation*}
        \left\|T_{W_\h}^{-1} f_\hbar\right\|_{L^2} \leq C \hbar^{-N_0} \|f_\hbar\|_{L^2} + C \hbar^{-\frac{1}{2} + \varepsilon} \|f_\hbar\|_{L^2} + \frac{1}{2} \left\|T_{W_\h}^{-1}\right\|_{L^2 \to L^2} \|f_\hbar\|_{L^2},
    \end{equation*}
    and thus
    \begin{equation*}
        \left\|T_{W_\h}^{-1}\right\|_{L^2 \to L^2} \leq C \hbar^{-N_0} + C \hbar^{-\frac{1}{2} + \varepsilon} + \frac{1}{2} \left\|T_{W_\h}^{-1}\right\|_{L^2 \to L^2},
    \end{equation*}
    which gives us (\ref{eq:BorneConditionnementTh}).

\textbf{Bound on $\|T_{W_\h}\|_{L^2 \to L^2}$.}
Finally, we control the norm of $T_{W_\h}$ as in~\eqref{eq:BorneTh}.
For simplicity, let us write
\begin{equation*}
p_{\max,\h}
\eq
\max_{\bq \in \Lambda_\h}
|p_\h(\bz^{\h,\bq})|,
\qquad
\xi_{\max,\h}
\eq
\max_{\bq \in \Lambda_\h}
|\bxi^{\h,\bn_{\bq}}|.
\end{equation*}
By~\eqref{eq:CouchePasTropGrande}, we have
$p_{\max,\h} \leq \alpha + \lambda'' \h^{1/2-\varepsilon}$.
Besides, Assumption~\ref{Hyp:BoundedLayers} ensures
that $\xi_{\max,\h} \leq D_0$.

In view of the bounds on $p_{\max, \h}$ and $\xi_{\max, \h}$,
we then note that since $\|\Pi_{W_\h}\|_{L^2 \to L^2} \leq 1$
for the orthogonal projection $\Pi_{W_\h}$, it is sufficient to
show that
\begin{equation}
\label{tmp_goal_continuity}
\|P_\h w_\h\|_{L^2}
\leq
C
(p_{\rm max, \h} + (1+\xi_{\max, \h})\h^{(1-\varepsilon)/2})
\|w_\h\|_{L^2}
\end{equation}
for all $w_\h \in W_\h$, to establish~\eqref{eq:BorneTh}.
We thus consider $U_\h \in \ell^2(\Lambda_\h)$, and let $w_\h := \cD_\h U_\h$. We have
\begin{equation}
\label{tmp_goal_continuity_split}
P_\h w_\h
=
\sum_{\bq \in \Lambda_h}
U_{\h,\bq} P_\h \w{\qvec}
=
\sum_{\bq \in \Lambda_h}
U_{\h,\bq} (P_\h- p_\h(\bz^{\h,\bq}))\w{\qvec}
+
\sum_{\bq \in \Lambda_h}
p_\h(\bz^{\h,\bq}) U_{\h,\bq} \w{\qvec}.
\end{equation}
For the second term in the right-hand side
of~\eqref{tmp_goal_continuity_split}, equation \eqref{eq: relation entre la norme de U et celle de u pour Wilson} implies that
\begin{multline}
\label{tmp_goal_continuity_second}
\left \|
\sum_{\bq \in \Lambda_h}
p_\h(\bz^{\h,\bq})U_{\h,\bq} \w{\qvec}
\right \|_{L^2}^2
\leq
C
\sum_{\bq \in \Lambda_h}
|p_\h(\bz^{\h,\bq})U_{\h,\bq}|^2
\\
\leq
C
p_{\max,\h}^2
\sum_{\bq \in \Lambda_h}
|U_{\h,\bq}|^2
\leq
C p_{\max,\h}^2 \|w_\h\|_{L^2}^2.
\end{multline}
The analysis of the first term in the right-hand side of~\eqref{tmp_goal_continuity_split}
is more subtle. We introduce a parameter $\eta > 0$ to be fixed later.
We then expand
\begin{align*}
\left \|
\sum_{\bq \in \Lambda_h}
U_{\h,\bq} (P_\h- p_\h(\bz^{\h,\bq}))\w{\qvec}
\right \|_{L^2}^2
&=
\sum_{\bq \in \Lambda_h}
\sum_{\bq' \in \Lambda_h}
U_{\h,\bq} \overline{U_{\h,\bq'}}
\langle (P_\h- p_\h(\bz^{\h,\bq}))\w{\qvec},(P_\h-p_\h(\bz^{\h,\bq'}))\w{\bq'}\rangle
\\
&=
\sum_{\bq \in \Lambda_h}
\sum_{\substack{\bq' \in \Lambda_h \\|\bq-\bq'| \leq h^{-\eta}}}
U_{\h,\bq} \overline{U_{\h,\bq'}}
\langle (P_\h- p_\h(\bz^{\h,\bq}))\w{\qvec},(P_\h-p_\h(\bz^{\h,\bq'}))\w{\bq'}\rangle
\\
&+
\sum_{\bq \in \Lambda_h}
\sum_{\substack{\bq' \in \Lambda_h\\|\bq-\bq'| > h^{-\eta}}}
U_{\h,\bq} \overline{U_{\h,\bq'}}
\langle (P_\h- p_\h(\bz^{\h,\bq}))\w{\qvec},(P_\h-p_\h(\bz^{\h,\bq'}))\w{\bq'}\rangle
\\
&=:
\Sigma_{\rm near} + \Sigma_{\rm far}.
\end{align*}
By invoking Lemma~\ref{lemme: (P - p) Phi}, we have on the one hand
\begin{equation*}
\begin{aligned}
|\Sigma_{\rm near}|
&\leq
C(1+\xi_{\max,\h})^2 \h
\sum_{\bq \in \Lambda_h}
\sum_{\substack{\bq' \in \Lambda_h\\|\bq-\bq'| \leq h^{-\eta}}}
|U_{\h,\bq}| |U_{\h,\bq'}|\\
&
= 
C(1+\xi_{\max,\h})^2 \h \sum_{\bq,\bq'\in \Lambda_{\h}} \left( |U_{\h,\bq}| \boldsymbol{1}_{|\bq-\bq'| \leq h^{-\eta}}\right) \times  \left( |U_{\h,\bq'}| \boldsymbol{1}_{|\bq-\bq'| \leq h^{-\eta}}\right) \\
&\leq
C(1+\xi_{\max,\h})^2 \h^{1-2d\eta}\|w_\h\|_{L^2}^2,
\end{aligned}
\end{equation*}
by the Cauchy-Schwarz inequality. On the other hand, using the estimates in
Lemmas~\ref{lemme: majoration de psi scalaire phi étoile}
and~\ref{lemme: P w scalaire w}, recalling that
$p_{\max,\h},\xi_{\max,\h} \leq C$, we have
\begin{equation*}
|\langle (P_\h- p_\h(\bz^{\h,\bq}))\w{\qvec},(P_\h-p_\h(\bz^{\h,\bq'}))\w{\bq'}\rangle|
=
O_\eta(\h^\infty)
\end{equation*}
whenever $|\bq-\bq'| > \h^{-\eta}$.
Since the second sum $\Sigma_{\rm far}$ involves a number of terms that is bounded
by a power of $\h^{-1}$, we conclude that
\begin{equation*}
|\Sigma_{\rm far}|
=
\|w_\h\|_{L^2}^2 \times O_\eta(\h^\infty).
\end{equation*}
This leads us to
\begin{equation}
\label{tmp_goal_continuity_first}
\left \|
\sum_{\bq \in \Lambda_h}
U_{\h,\bq} (P_\h- p_\h(\bz^{\h,\bq}))\w{\qvec}
\right \|_{L^2}^2
\leq
C (1+\xi_{\max,\h})^2 \h^{1-2d\eta} \|w_\h\|_{L^2}^2
+
\|w_\h\|_{L^2}^2 \times O_\eta(\h^\infty).
\end{equation}
By selecting $\eta = \varepsilon/(2d)$, we obtain~\eqref{tmp_goal_continuity}
by combining~\eqref{tmp_goal_continuity_split},~\eqref{tmp_goal_continuity_second}
and~\eqref{tmp_goal_continuity_first}, and~\eqref{eq:BorneTh} follows.
\end{proof}

\section{Numerical examples}
\label{section_numerics}

We consider two numerical examples to illustrate the theoretical results obtained above
in a one-dimensional setting. The first test case is the constant wave speed Helmholtz
equation with a PML imposed for $|x| \geq L \eq 4$. Specifically, we consider the problem
of finding $u_k \in H^1(\R)$ such that
\begin{equation*}
-k^2 \mu_0 u_k-(\alpha_0 u_k')' = f_k
\end{equation*}
with coefficients given by
\begin{equation*}
\mu_0
\eq
1-i\vartheta(|x|)
\qquad
\alpha_0
\eq
\frac{1}{1-i\vartheta(|x|)}
\end{equation*}
where $\vartheta$ is a smooth transition function such that
$\vartheta(r) = 0$ whenever $0 \leq r \leq L$ and $\vartheta(r) = 1$ for $r \geq 5$.
The right-hand side is
\begin{equation*}
f_k(x) = -\chi''(|x|)e^{-ik|x|}+2ik\chi'(|x|)e^{-ik|x|},
\end{equation*}
where $\chi$ is a smooth cutoff function such that
$\chi(r) = 0$ for $0 \leq r \leq 1$ and $\chi(r) = 1$
for $r \geq 2$. The corresponding analytical solution (outside the PML)
is given by
\begin{equation*}
u_k(x) = \chi(|x|) e^{-ik|x|}.
\end{equation*}
Note that since the cutoff vanishes in a neighbourhood of the origin,
the right-hand side and the solution are smooth.
For the heterogeneous test case, we keep the same right-hand side,
but {we replace the coefficients $\mu_0$ and $\alpha_0$ by
\begin{equation*}
\mu(x) \eq \mu_0(x) + 0.7(b(x-3)-b(x+3)),
\qquad
\alpha(x) \eq \alpha_0(x) + 0.5(b(x-3)-b(x+3)),
\end{equation*}
where $b$ is a smooth positive bump function such that $b(0) = 1$
and $b(t) = 0$ whenever $|t| \geq 0.5$. This corresponds to
a localised increase of wave speed for negative $x$
(up to $\sqrt{0.5/0.3}$ at $x=-3$) and
a localised decrease for positive $x$
(down to $\sqrt{1.5/1.7}$ at $x=3$).
In this case, the solution is unknown.

In both cases, we consider frequencies
$2^2 \cdot 2\pi$ to $2^{17} \cdot 2\pi$,
and we select degrees of freedom using the rule
\begin{equation*}
|\xxkm| < 7
\qquad
\text{and}
\qquad
|p_\hbar(\xxkm,\xikn)| < \lambda k^{-1/2+\varepsilon}
\end{equation*}
with
\begin{equation}
\label{eq_lambda_varepsilon}
(\lambda,\varepsilon) \in \{(40,0.0),(20,0.1),(10,0.2)\}.
\end{equation}
This rule seems at first slightly different from the one we theoretically
analyse in~\eqref{eq:EnergyLayers}, which we now explain. First, the requirement
that $|x| < L+3=7$ is only made to ensure that the index set remains finite
for small frequencies. In fact, for large frequencies, this condition is already
implied by the second one, and the maximal $|x|$ considered tends to $L$ as the
frequency increases. For the second constraint, we did not use the ``symmetrised''
indices with the pseudometric $\delta$ as in~\eqref{eq:EnergyLayers}, but rather,
we directly employed the indices. As explained in Remark~\ref{remark_symetric_symbol},
this definition is almost equivalent to the one with symmetrised indices since the symbol
is itself symmetric here.

For both examples, we compute the relative residual
\begin{equation*}
\frac{\|P_k v_k-f_k\|_{L^2}}{\|f_k\|_{L^2}}
\end{equation*}
and, in first case (where the analytic solution is available outside
the PML) the error
\begin{equation*}
\frac{\|u_k-v_k\|_{H^1_k(-L,L)}}{\|u_k\|_{H^1_k(-L,L)}},
\end{equation*}
where $\|{\cdot}\|_{H^1_k(-L,L)}^2 \eq k^2\|{\cdot}\|_{L^2(-L,L)}^2+\|{\cdot}'\|_{L^2(-L,L)}$.
The results for the two cases are respectively reported on
Figures~\ref{figure_homogeneous} and~\ref{figure_heterogeneous}.
In the legends, we have just indicated the different values of $\varepsilon$,
but the corresponding values of $\lambda$ can be found as per~\eqref{eq_lambda_varepsilon}.
The configurations where $\varepsilon > 0$ are covered by the theoretical
analysis, and we expect the relative error and residual to decrease as the
frequency increases. On the other hand, we expect to observe constant
relative error and residual in the case where $\varepsilon = 0$.

In both cases, we first observe that there is a preasymptotic regime
where the number of degrees of freedom increases as $k^{0.8}$, before
stabilizing to the expected rate $k^{1/2+\varepsilon}$. In the preasymptotic
regime, the relative error and residual decrease, even when $\varepsilon = 0$.
In the asymptotic regime, we see that the relative error and residual become
constant when $\varepsilon = 0$, whereas the number of degrees of freedom grows as $k^{1/2}$,
which is in line with our expectations. We also see in this regime that the
solution is more accurate for higher values of $\varepsilon$ at cost of extra
degrees of freedom. This is especially visible in the heterogeneous test case.
However, in the homogeneous test case, this effect is harder to see due to
conditioning and quadrature errors.

\begin{figure}
\begin{minipage}{.45\linewidth}
\begin{tikzpicture}
\begin{axis}
[
	width=\linewidth,
	xlabel={$k$},
	ylabel={Relative residual},
	ymode=log,
	xmode=log,
	legend pos=north east
]

\plot[blue ,mark=x     ] table[x=k,y=res] {data/homogeneous/res_40_0.0.txt};
\plot[red  ,mark=o     ] table[x=k,y=res] {data/homogeneous/res_20_0.1.txt};
\plot[black,mark=square] table[x=k,y=res] {data/homogeneous/res_10_0.2.txt};

\legend{$\varepsilon=0.0$,$\varepsilon=0.1$,$\varepsilon=0.2$}

\end{axis}
\end{tikzpicture}
\end{minipage}
\begin{minipage}{.45\linewidth}
\begin{tikzpicture}
\begin{axis}
[
	xlabel={$k$},
	ylabel={\# Dofs},
	width=\linewidth,
	ymode=log,
	xmode=log,
	legend pos=north west
]

\plot[blue ,mark=x     ] table[x=k,y=N] {data/homogeneous/res_40_0.0.txt};
\plot[red  ,mark=o     ] table[x=k,y=N] {data/homogeneous/res_20_0.1.txt};
\plot[black,mark=square] table[x=k,y=N] {data/homogeneous/res_10_0.2.txt};

\plot [dashed,domain=1e1:5e2] {8*x^0.8};
\SlopeTriangle{0.5}{0.2}{0.2}{0.8}{$k^{0.8}$}{}

\plot [dashed,domain=1e3:1e5] {80*x^0.5};
\SlopeTriangle{0.9}{0.2}{0.6}{0.5}{$k^{0.5}$}{}

\legend{$\varepsilon=0.0$,$\varepsilon=0.1$,$\varepsilon=0.2$}

\end{axis}
\end{tikzpicture}
\end{minipage}

\begin{minipage}{.45\linewidth}
\begin{tikzpicture}
\begin{axis}
[
	xlabel={$k$},
	ylabel={Relative error},
	width=\linewidth,
	ymode=log,
	xmode=log,
	legend pos=north east
]

\plot[blue ,mark=x     ] table[x=k,y=err] {data/homogeneous/res_40_0.0.txt};
\plot[red  ,mark=o     ] table[x=k,y=err] {data/homogeneous/res_20_0.1.txt};
\plot[black,mark=square] table[x=k,y=err] {data/homogeneous/res_10_0.2.txt};

\legend{$\varepsilon=0.0$,$\varepsilon=0.1$,$\varepsilon=0.2$}

\end{axis}
\end{tikzpicture}
\end{minipage}
\begin{minipage}{.45\linewidth}
$ $
\end{minipage}

\caption{Numerical results in the homogeneous case.}
\label{figure_homogeneous}
\end{figure}
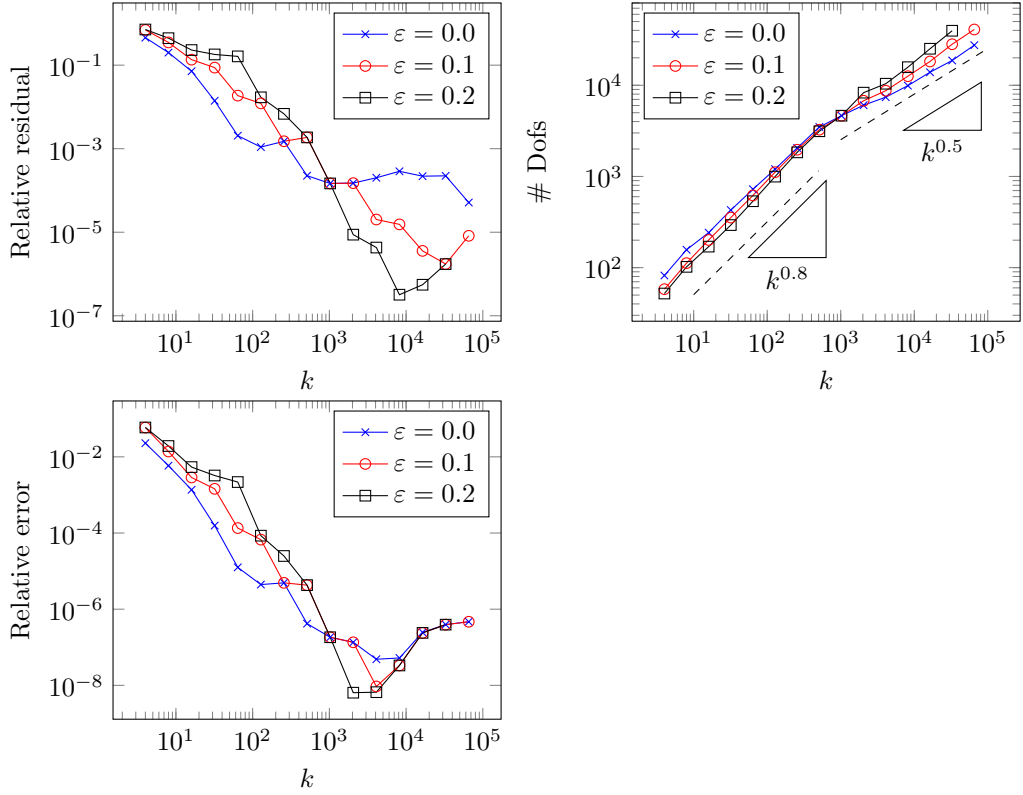

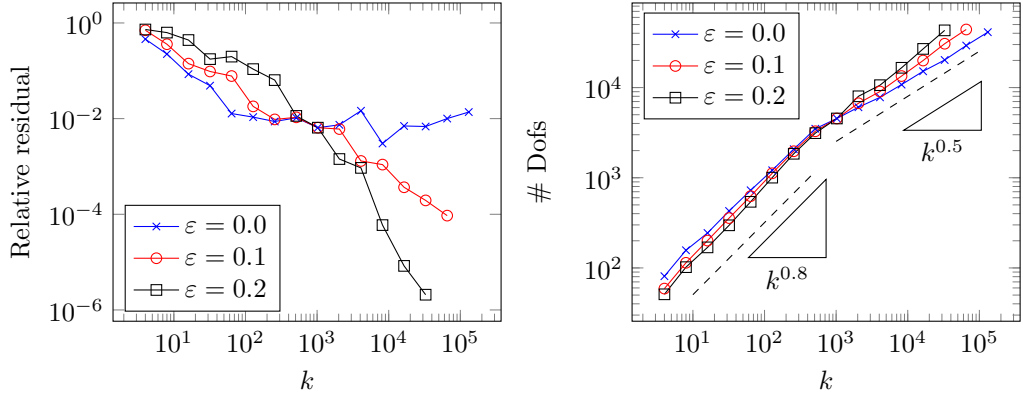
\begin{figure}
\begin{minipage}{.45\linewidth}
\begin{tikzpicture}
\begin{axis}
[
	width=\linewidth,
	xlabel={$k$},
	ylabel={Relative residual},
	ymode=log,
	xmode=log,
	legend pos=south west
]

\plot[blue ,mark=x     ] table[x=k,y=res] {data/heterogeneous/res_40_0.0.txt};
\plot[red  ,mark=o     ] table[x=k,y=res] {data/heterogeneous/res_20_0.1.txt};
\plot[black,mark=square] table[x=k,y=res] {data/heterogeneous/res_10_0.2.txt};

\legend{$\varepsilon=0.0$,$\varepsilon=0.1$,$\varepsilon=0.2$}

\end{axis}
\end{tikzpicture}
\end{minipage}
\begin{minipage}{.45\linewidth}
\begin{tikzpicture}
\begin{axis}
[
	xlabel={$k$},
	ylabel={\# Dofs},
	width=\linewidth,
	ymode=log,
	xmode=log,
	legend pos=north west
]

\plot[blue ,mark=x     ] table[x=k,y=N] {data/heterogeneous/res_40_0.0.txt};
\plot[red  ,mark=o     ] table[x=k,y=N] {data/heterogeneous/res_20_0.1.txt};
\plot[black,mark=square] table[x=k,y=N] {data/heterogeneous/res_10_0.2.txt};

\plot [dashed,domain=1e1:5e2] {8*x^0.8};
\SlopeTriangle{0.5}{0.2}{0.2}{0.8}{$k^{0.8}$}{}

\plot [dashed,domain=1e3:1e5] {80*x^0.5};
\SlopeTriangle{0.9}{0.2}{0.6}{0.5}{$k^{0.5}$}{}

\legend{$\varepsilon=0.0$,$\varepsilon=0.1$,$\varepsilon=0.2$}

\end{axis}
\end{tikzpicture}
\end{minipage}

\caption{Numerical results in the heterogeneous case.}
\label{figure_heterogeneous}
\end{figure}

\appendix

\section{Boundedness of the operator $P_\h$}
\label{app:ProofContinuity}

The aim of this appendix is to prove Lemma \ref{lemme: P est continu}.

\begin{proof}[Proof of Lemma \ref{lemme: P est continu}]
    We have
    \begin{equation}\label{eq:DecompoOperateurP}
        \|P_\hbar w\|_{\widehat{H}_\hbar^p} \leq \hbar^2 \sum_{j,l=1}^d \left\|a_{j,l}^\hbar \partial_{j,l}^2 w\right\|_{\widehat{H}_\hbar^p} + \hbar \sum_{j=1}^d \left\|b_j^\hbar \partial_j w\right\|_{\widehat{H}_\hbar^p} + \left\|c^\hbar w\right\|_{\widehat{H}_\hbar^p}.
    \end{equation}
To bound the first term, we rewrite it as
    \begin{align*}
        \left\|a_{j,l}^\hbar \partial_{j,l}^2 w\right\|_{\widehat{H}_\hbar^p}^2 &= \sum_{[\alphavec] \leq p} \sum_{q \leq p - [\alphavec]} \hbar^{2[\alphavec]} \left\|\;|\xvec|^q \partial^{\alphavec} \left(a_{j,l}^\hbar \partial_{j,l}^2 w\right)\right\|_{L^2}^2\\
        &= \sum_{[\alphavec] \leq p} \sum_{q \leq p - [\alphavec]} \hbar^{2[\alphavec]} \left\|\;|\xvec|^q \sum_{\betavec \leq \alphavec} \binom{\alphavec}{\betavec} \partial^{\betavec} \partial_{j,l}^2 w \ \partial^{\alphavec - \betavec} a_{j,l}^\hbar\right\|_{L^2}^2\\
        &\leq \sum_{[\alphavec] \leq p} \sum_{q \leq p - [\alphavec]} \hbar^{2[\alphavec]} \left(\sum_{\betavec \leq \alphavec} \binom{\alphavec}{\betavec} \left\|\;|\xvec|^q \partial^{\betavec} \partial_{j,l}^2 w \ \partial^{\alphavec - \betavec} a_{j,l}^\hbar \right\|_{L^2}\right)^2\\
        &\leq C_p \sum_{[\alphavec] \leq p} \sum_{q \leq p - [\alphavec]} \hbar^{2[\alphavec]} \sum_{\betavec \leq \alphavec} \left\|\;|\xvec|^q \partial^{\betavec} \partial_{j,l}^2 w \ \partial^{\alphavec - \betavec} a_{j,l}^\hbar \right\|_{L^2}^2.
    \end{align*}
Now, for every  $\alphavec$ such that $[\alphavec] \leq p$,  we have
    \begin{align*}
        \sum_{\betavec \leq \alphavec} \left\|\;|\xvec|^q \partial^{\betavec} \partial_{j,l}^2 w \ \partial^{\alphavec - \betavec} a_{j,l}^\hbar \right\|_{L^2}^2 &= \sum_{\betavec \leq \alphavec} \int_{\R^d} |\xvec|^{2q} \left|\partial^{\betavec} \partial_{j,l}^2 w(\xvec)\right|^2 \left|\partial^{\alphavec - \betavec} a_{j,l}^\hbar(\xvec)\right|^2 d\xvec\\
        &\leq \left\|a_{j,l}^\hbar\right\|_{C^p(\R^d)}^2 \sum_{\betavec \leq \alphavec} \int_{\R^d} |\xvec|^{2q} \left|\partial^{\betavec + \boldsymbol{e}_j + \boldsymbol{e}_l} w(\xvec)\right|^2 d\xvec\\
        &\leq C_{\mathrm{coef}, p}^2 \sum_{[\betavec] \leq [\alphavec]} \left\|\;|\xvec|^q \partial^{\betavec + \boldsymbol{e}_j + \boldsymbol{e}_l} w \right\|_{L^2}^2\\
        &\leq C_{\mathrm{coef}, p}^2 \sum_{[\gammavec] \leq [\alphavec] + 2} \left\|\;|\xvec|^q \partial^{\gammavec} w \right\|_{L^2}^2.
    \end{align*}
Consequently,
    \begin{align*}
        \left\|a_{j,l}^\hbar \partial_{j,l}^2 w\right\|_{\widehat{H}_\hbar^p}^2 &\leq C_p \sum_{[\alphavec] \leq p} \sum_{q \leq p - [\alphavec]} \hbar^{2[\alphavec]} \sum_{[\gammavec] \leq [\alphavec] + 2} \left\|\;|\xvec|^q \partial^{\gammavec} w \right\|_{L^2}^2\\
        &\leq C_p \sum_{[\alphavec] \leq p} \sum_{[\gammavec] \leq [\alphavec] + 2} \sum_{q \leq p - [\alphavec]} \hbar^{2[\gammavec] - 4} \left\|\;|\xvec|^q \partial^{\gammavec} w \right\|_{L^2}^2\\
        &\leq C_p \hbar^{-4} \sum_{[\alphavec] \leq p} \sum_{[\gammavec] \leq [\alphavec] + 2} \sum_{q \leq p + 2 - [\gammavec]} \hbar^{2[\gammavec]} \left\|\;|\xvec|^q \partial^{\gammavec} w \right\|_{L^2}^2\\
        &\leq C_p \hbar^{-4} \sum_{[\alphavec] \leq p} \left(\sum_{[\gammavec] \leq p + 2} \sum_{q \leq p + 2 - [\gammavec]} \hbar^{2[\gammavec]} \left\|\;|\xvec|^q \partial^{\gammavec} w \right\|_{L^2}^2\right)\\
        &\leq C_p \hbar^{-4} \|w\|_{\widehat{H}_\hbar^{p+2}}^2,
    \end{align*}
    so that $\left\|a_{j,l}^\hbar \partial_{j,l}^2 w \right\|_{\widehat{H}_\hbar^p} \leq C_p \hbar^{-2} \|w\|_{\widehat{H}_\hbar^{p+2}}$.  We show in a similar way that $\left\|b_j^\hbar \partial_j w \right\|_{\widehat{H}_\hbar^p(\R^d)} \leq C_p \hbar^{-1} \|w\|_{\widehat{H}_\hbar^{p+1}}$ and that  $\left\|c^\hbar w \right\|_{\widehat{H}_\hbar^p} \leq C_p \|w\|_{\widehat{H}_\hbar^p}$. The result then follows from (\ref{eq:DecompoOperateurP}).
\end{proof}

\printbibliography
\end{document}